\documentclass{amsart}
\usepackage[margin=3cm]{geometry}
\usepackage{amssymb}
\usepackage{dsfont}
\usepackage[utf8]{inputenc}

\usepackage{comment}
\usepackage{xcolor}
\usepackage{mathtools}
\usepackage{stmaryrd}
\usepackage{graphicx}
\usepackage[nolist]{acronym}
\usepackage{subcaption}
\usepackage{float}
\usepackage{hyperref}

\def\a        {{\boldsymbol a}}
\def\x        {{\boldsymbol x}}

\def\dx     {{\rm d}{\boldsymbol x}}

\newcommand{\detector}[1][]{\alpha_{#1}}
\newcommand{\nodes}{\mathcal{N}_h}
\newcommand{\sym}{{\rm sym}}
\def\ij{{ij}}
\newcommand{\np}{{n+1}}
\newcommand{\rr}{\boldsymbol{r}}

\newcommand{\gradient}{\nabla}

\newcommand{\jump}[1]{\left\llbracket #1 \right\rrbracket}
\newcommand{\mean}[1]{%
	\sbox0{%
		\mathsurround=0pt 
		$\left\{\vphantom{#1}\right.\kern-\nulldelimiterspace$%
	}%
	\sbox2{\{}%
	\ifdim\ht0=\ht2
	\{\kern-.625\wd2 \{#1\}\kern-.625\wd2 \}%
	\else
	\left\{\kern-.7\wd0\left\{#1\right\}\kern-.7\wd0\right\}%
	\fi
}

\newtheorem{theorem}{Theorem}[section]
\newtheorem{lemma}{Lemma}[section]
\newtheorem{remark}{Remark}[section]

\graphicspath{{figures/}}

\begin{document}

\begin{acronym}
	\acro{fe}[FE]{finite element}
	\acro{dg}[dG]{discontinuous Galerkin}
	\acro{cg}[cG]{continuous Galerkin}
	\acro{dof}[DOF]{degrees of freedom}
	\acro{ssp}[SSP]{strong stability preserving}
	\acro{rk}[RK]{Runge Kutta}
	\acro{be}[BE]{Backward Euler}
	\acro{dmp}[DMP]{discrete maximum principle}
	\acro{mp}[MP]{maximum principle}
	\acro{afc}[AFC]{algebraic flux correction}
	\acro{fct}[FCT]{flux corrected transport}
	\acro{led}[LED]{local extremum diminishing}
	\acro{dled}[DLED]{discrete local extremum diminishing}
\end{acronym}

\title[Approximation of the Keller--Segel equations ]{Bound-preserving finite element approximations of the Keller-Segel equations}
\author[S. Badia]{Santiago Badia$^\dag$}
\address{$\dag$ School of Mathematics, Monash University, Clayton, Victoria, 3800, Australia \& CIMNE, Centre Internacional de Mètodes Numèrics a l’Enginyeria, Barcelona, 08034, Spain. \newline E-mail: {\tt \href{mailto:santiago.badia@monash.edu}{santiago.badia@monash.edu}}}

\author[J. Bonilla]{Jesús Bonilla$^\ddag$}
\address{$\ddag$ Los Alamos National Laboratory, Los Alamos, NM 87545, USA. E-mail: {\tt \href{mailto:jbonilla@lanl.gov}{jbonilla@lanl.gov}}}
\thanks{$\ddag$ Los Alamos National Laboratory, an affirmative action/equal opportunity employer, is operated by Triad National Security, LLC for the National Nuclear Security
Administration of U.S. Department of Energy under contract 89233218CNA000001. Los Alamos National Laboratory strongly supports academic freedom and a researcher's right to publish; as an institution, however, the Laboratory does not endorse the viewpoint of a publication or guarantee its technical correctness. LA-UR-22-27083}

\author[J. V. Gutiérrez-Santacreu]{Juan Vicente Gutiérrez-Santacreu$^\S$}
\address{$\S$Dpto. de Matemática Aplicada I\\
         E. T. S. I. Informática\\
         Universidad de Sevilla\\
         Avda. Reina Mercedes, s/n.\\
         E-41012 Sevilla\\
         Spain\\
         E-mail: {\tt \href{mailto:juanvi@us.es}{juanvi@us.es}}}

\thanks{JVGS was partially supported by the Spanish Grant No. PGC2018-098308-B-I00 from Ministerio de Ciencias e Innovación - Agencia Estatal de Investigación with the participation of FEDER and by the Andalusian Grant No. P20\_01120 from Junta de Andalucía (Consejería de Economía, Conocimiento, Empresas y Universidad)}
\date{\today}
\begin{abstract} This paper aims to develop numerical approximations of the Keller--Segel equations that mimic at the discrete level the lower bounds and the energy law of the continuous problem. We solve these equations for two unknowns: the organism (or cell) density, which is a positive variable, and the chemoattractant density, which is a nonnegative variable.
We propose two algorithms, which combine a stabilized finite element method and a semi-implicit time integration. The stabilization consists of a nonlinear artificial diffusion that employs a graph-Laplacian operator and a shock detector that localizes local extrema. As a result, both algorithms turn out to be nonlinear.
Both algorithms can generate cell and chemoattractant numerical densities fulfilling lower bounds. However, the first algorithm requires a suitable constraint between the space and time discrete parameters, whereas the second one does not. We design the latter to attain a discrete energy law on acute meshes. 
We report some numerical experiments to validate the theoretical results on blowup and non-blowup phenomena. In the blowup setting, we identify a \textit{locking} phenomenon that relates the $L^\infty(\Omega)$-norm to the $L^1(\Omega)$-norm limiting the growth of the singularity when supported on a macroelement.      
\end{abstract}

\maketitle



\tableofcontents

\section{Introduction}
\subsection{The Keller--Segel equations} 
\emph{Dictyostelium} is an ameba, a unicellular eukaryotic organism that lives in the soil, feeds on bacteria, and reproduces by bipartition. However, what truly sticks out about its life cycle is its behavior as a social ameba. When amebas lack food and therefore cannot divide, Dictyostelium opts for an alternate life cycle and joins with its congeners to lead to a stage of cell development and differentiation. Thus Dictyostelium amebas are capable of movement, thanks to a molecular mechanism, in response to certain chemicals released by themselves. This aggregation phenomenon is known as chemotaxis.

Keller and Segel \cite{Keller_Segel_1970,Keller_Segel_1971}  were the pioneers in deriving the first mathematical model to predict aggregation phenomena in populations of Dictyostelium discoideum or E. coli. The model they proposed reads as follows. Let $\Omega\subset \mathds{R}^d$, $d=2$ or $3$, be a bounded domain, with $\boldsymbol{n}$ being its outward-directed unit normal vector to $\Omega$, and let $[0,T]$ be a time interval. Take $Q=(0,T]\times \Omega$ and $\Sigma=(0,T]\times\partial\Omega$. Then, the boundary-value problem for the Keller--Segel equations consists of finding $u: \bar Q\to (0,\infty)$, the organism (or cell) density, and $v:\bar Q \to [0,\infty)$, the chemoattractant density,  satisfying 
\begin{equation}\label{KS}
\left\{
\begin{array}{rcll}
\partial_tu-\Delta u&=&-\nabla\cdot(u\nabla v)&\mbox{ in } Q,
\\
\partial_t v -\Delta v&=&u-v&\mbox{ in }Q,
\end{array}
\right.
\end{equation}
subject to the initial conditions
\begin{equation}\label{IC}
u(0)=u_0\quad\mbox{ and }\quad v(0)=v_0\quad\mbox{ in }\quad \Omega,
\end{equation}
and the no-flux boundary conditions
\begin{equation}\label{BC}
\nabla u\cdot \boldsymbol{n}=0\quad\mbox{ and }\quad \nabla v\cdot\boldsymbol{n}=0\quad\mbox{ on }\quad \Sigma.
\end{equation}

The dynamics of solutions to \eqref{KS}--\eqref{BC} is governed by an energy law, which can be formally deduced as follows. Suppose we have a classical solution $(u,v)$ to system \eqref{KS}--\eqref{BC}  on $Q$; multiply $\eqref{KS}_1$ by $\log u -v$ and  $\eqref{KS}_2$ by $\partial_t u$ and integrate over $Q$ and, finally, add  the two resulting to get:
\begin{equation}\label{Energy_law}
\frac{d}{{\rm d}t}\mathcal{E}(u(t),v(t))=-\mathcal{D}(u(t),v(t))\quad\forall t\in[0,T],
\end{equation}
where
$$
\mathcal{E}(u,v)\doteq\frac{1}{2}\int_{\Omega} |\nabla v(\x)|^2\, \dx+\frac{1}{2}\int_\Omega v^2(\x)\, \dx-\int_\Omega u(\x) v(\x)\, \dx+\int_\Omega u(\x) \ln u(\x)\, \dx
$$
and
$$
\mathcal{D}(u,v)\doteq\int_\Omega |\Delta v(\x)-v(\x)+u(\x)|^2\, \dx+\int_\Omega|\frac{\nabla u(\x)}{\sqrt{u(\x)}}-\sqrt{u(\x)}\nabla v(\x) |^2\, \dx.
$$
We note that the energy law \eqref{Energy_law} is only well-posed on the condition that $u>0$. 

The importance of \eqref{Energy_law} is twofold. On the one hand, existence theory \cite{Nagai_Senba_Yoshida_1997} of two-dimensional bounded solutions draws heavily on \eqref{Energy_law} and  a Morse-Trudinger--like inequality, together with the condition  $\int_\Omega u_0(\x)\,\dx\in (0, 4\pi)$; if $\Omega\subset\mathds{R}^2$ is a ball and $(u_0, v_0)$ are radially symmetric, one needs $\int_\Omega u_0(\x)\,\dx\in (0, 8\pi)$. In dimension three, smallness conditions for $(u_0, v_0)$ are required to prove bounded solutions \cite{Winkler_2010}.  On the other hand, solutions, which may blow up either in finite or infinite time  \cite{Horstmann_Wang_2001, Winkler_2010}, are also present in system \eqref{KS}--\eqref{BC}. These unbounded solutions emerge from careful asymptotic analysis of \eqref{Energy_law}, i. e. if $T_{\textrm{max}}\in(0,\infty]$ is a blowup time, then 
\begin{equation}\label{E-blowup}
\mathcal{E}(u(t), v(t))\to -\infty\quad\mbox{ as }\quad t\to T_{\textrm{max}}.
\end{equation}
Another essential fact \cite{Horstmann_2001} is that there exist times when both the cell and chemoattractant densities should blow up at such times.  

Furthermore, a sufficient condition \cite{Winkler_2013} for initial data whose corresponding solutions blow up within finite time stems from \eqref{Energy_law} with $\Omega\subset\mathds{R}^3$ being a ball and $(u_0, v_0)$ being radially symmetric. For the two-dimensional radially symmetric case, it is well-known that such a blowup occurs in finite time \cite{Herrero_Velazquez_1997}.

In the light of the above discussion, it is of great interest to construct numerical solutions that satisfy lower bounds --positivity for the organism density and nonnegativity for the chemoattractant density-- and a discrete counterpart of \eqref{Energy_law}. A direct numerical simulation is a simple approach to gaining insight into the issues of potential singularity formations for system \eqref{KS}--\eqref{BC}. Nevertheless, this is a rudimentary approach because an artificial blowup might arise from purely numerical oscillations. After all, the numerical scheme does not guarantee lower bounds. On the contrary, adding numerical diffusion ensuring lower bounds might prevent the detection of a singularity. These are subtle issues for system \eqref{KS}--\eqref{BC}  because specific solutions always have singular rapid-growth behavior and it is difficult to distinguish them from blowup formation.    

Several numerical discretizations of  \eqref{KS}--\eqref{BC} have been proposed to satisfy lower bounds. Among them, we can find geometrical conditions \cite{Saito_2012, GS_RG_2021}  on meshes being weakly acute in two dimensions and acute in three dimensions, flux limiters \cite{Strehl_Sokolov_Kuzmin_Turek_2010, Strehl_Sokolov_Kuzmin_Horstmann_Turek_2013}, discontinuous Galerkin methods, \cite{Li_Shu_Yang_2017} and finite volume methods \cite{Chertock_Epshteyn_Hu_Kurganov_2018, Chertock_Kurganov_2008}.

As for numerical solutions satisfying \eqref{Energy_law} at the discrete level, very few numerical algorithms are available. For example, a discontinuous Galerkin method was proposed \cite{Guo_Li_Yang_2019}, but without proving lower bounds. On the other hand, a standard finite element method \cite{GS_RG_2021} has recently been proved to hold both lower bounds and a discrete energy law for acute meshes.          

This paper aims to construct finite element approximations satisfying lower bounds (as at the continuous level) without any requirement on the mesh in addition to quasi-uniformity. We shall provide two stabilized finite element methods \cite{Badia_Bonilla_2017}  consisting of adding a diffusion artifact. The artificial diffusion combines a graph-Laplacian operator and a shock detector, which minimizes the amount of numerical diffusion introduced in the system. For the first algorithm, we can prove lower bounds for both unknowns. The second algorithm enjoys lower bounds and, on acute meshes, a discrete energy law. In doing so, some terms are lumped, and the algorithm uses a new discretization of the chemotaxis term. Nevertheless, lower bounds might be attained, even without lumping as for the first algorithm. Besides, the stabilization term added in the discrete chemoattractant equation to preserve lower bounds on general meshes impede obtaining a discrete energy inequality. On acute meshes, we can switch off this stabilization keeping the lower bounds and prove a discrete energy law. In any case, both algorithms are nonlinear.

When using both algorithms in the context of approximating solutions that might blow up in finite time, we find a relation between the $L^1(\Omega)$- and  $L^\infty(\Omega)$-norm leading to a \textit{locking} structure in the growth of Dirac-like potential singularities. This limitation appears when the singularity at hand is supported on a macroelement because the $L^\infty(\Omega)$-norm is controlled by the $L^1(\Omega)$-norm, which remains globally bounded.

\subsection{Notation} Here $L^p(\Omega)$ and $W^{1,p}(\Omega)$, $p\in[1,\infty]$, are the Lebesgue and Sobolev spaces endowed with the usual norm $\|\cdot\|_{L^p(\Omega)}$ and $\|\cdot\|_{W^{1,p}(\Omega)}$.  In the particular case of $p=2$, it is denoted as $H^1(\Omega)$. 
The inner product of $L^2(\Omega)$ is denoted by $(\cdot,\cdot)$. Finally, by $C$ we mean a constant which may vary with context but is always independent of the discrete parameters.

\subsection{Layout} Section $2$ states the two numerical algorithms proposed in this work. We firstly give the features needed for constructing the finite element spaces and then announce our two algorithms, where we describe the stabilizing terms in detail. In particular, we motivate an expression of the shock detector for each algorithm. Moreover, the second algorithm incorporates a new discretization of the chemotaxis term. Then, in section $3$, we prove lower bounds for both algorithms and a discrete energy law for the second one. Finally, in section $4$, we end up with some numerical experiments to test our two algorithms, especially on lower bounds.

\section{Construction of numerical schemes}
The section presents our two stabilized finite element methods jointly with the hypotheses and notation required for developing the stabilizing terms.

\subsection{Hypotheses}

Let $\Omega$ be bounded domain in $\mathds{R}^d$, with $d=2$ or $3$. Its boundary is assumed to be polygonal or polyhedral. For $\Omega$, consider $\{\mathcal{K}_h\}_{h>0}$ to be a quasi-uniform family of conforming subdivisions of $\bar\Omega$ into closed, convex subsets $K$,  triangles or quadrilaterals ($d=2$) and tetrahedra or hexahedra  ($d=3$), with $h_K:={\rm diam} (K)$ and $h:=\max_{K\in\mathcal{K}_h} h_K$, so that $\bar\Omega=\sum_{K\in\mathcal{K}_h} $. Moreover, let $\nodes=\{\boldsymbol{a}_i\}_{i=1}^I$ be the coordinates of the nodes of $\mathcal{K}_h$.  Associated to $\mathcal{K}_h$ is the finite element space
$$
X_h=\{x_h\in C^0(\bar\Omega)\,:\, x_h|_K\in \mathds{P}_1\mbox{ or }\mathds{Q}_1\quad \forall K\in\mathcal{K}_h\},
$$ 
where $\mathds{P}_1$ is the set of linear polynomials on $K$, with $K$ being a triangle or tetrahedron, and $\mathds{Q}_1$ is the set of bilinear polynomials on $K$, with $K$ being a quadrilateral or hexahedron. Let $\{\varphi_{\a_i}\}_{i=1}^I$ be the global shape functions for $X_h$, that is, $\varphi_{\a_i}(\boldsymbol{a}_i)=\delta_{ij}$ for $i,j=1,\cdots, I$, for which $\Omega_{\a_i}={\rm supp }\,\varphi_{\a_i}$ denotes the macroelement associated to each $\varphi_{\a_i}$ and $\nodes(\Omega_{\a_i})$ is the set of the nodes belonging to $\Omega_{\a_i}$. Furthermore, one defines the set of indices $I(\Omega_{\a_i})=\{j\in I\,:\, \a_j\in \Omega_{\a_i}\}$. 

Let us introduce $i_h: C(\bar\Omega)\to X_h$ the linear interpolation operator such that $i_h x(\boldsymbol{a}_i)=x_h(\boldsymbol{a}_i)$ for $i=1,\cdots, I$. 
A discrete inner product is defined as
$$
(x_h,\bar x_h)_h=\int_\Omega i_h(x_h(\x)\bar x_h(\x))\, \dx. 
$$
The nodal values $\{x_h(\a_i)\}_{i=1}^I$ are denoted $\{x_i\}_{i=1}^I$.

We now recall a well-known inverse estimate \cite[Lem. 4.5.3]{Brenner_Scott_2008} concerning $X_h$. There exists $C_{\rm inv}>0$, independent of $h$, such that 
\begin{equation}\label{inv_ineq}
\|x_h\|_{W^{1,r}(\Omega)}\le C_{\rm inv} h^{-1+d \min\{\frac{1}{r}-\frac{1}{p},0\}} \|x_h\|_{L^p(\Omega)} \quad \forall x_h\in X_h. 
\end{equation}
Let us finally consider the following averaged interpolation operator defined as follows. Take, for each node $\a_i\in\mathcal{N}_h$, an  element  $K_{\a_i}$ such that $\a_i\in K_{\a_i}$. Thus we have:  
$$
\mathcal{I}_h\psi =\sum_{i\in I} \left(\frac{1}{|K_{\a_i}|}\int_{K_{\a_i}} \phi(\x)\,\dx\right)\varphi_{\a_i}.
$$  
We know \cite{Girault_Lions_2001, Scott_Zhang_1990} that there exists $C_{\rm sta}>0$, independent of $h$, such that  
\begin{equation}\label{I-stability}
\|\mathcal{I}_h \psi\|_{W^{s,p}(\Omega)}\le C_{\rm sta} \| \psi \|_{W^{s,p}(\Omega)}\quad \mbox{for }  s=0,1\mbox{ and } 1\le p\le\infty,
\end{equation}
and 
\begin{equation}\label{I-lower-bounds}
\mathcal{I}_h \psi \ge \mbox{or} > 0\quad\mbox{ if }\quad \psi\ge \mbox{or} >0. 
\end{equation}

\subsection{Finite element approximation}
It is assumed that $(u_0, v_0)\in L^1(\Omega)\times H^1(\Omega)$ with $u_0>0$ and $v_0\ge0$. Then it is considered $u_{0h}=\mathcal{I}_h (u_0)$ and $ v_{0h}=\mathcal{I}_h(v_0)$ satisfying \eqref{I-stability} and \eqref{I-lower-bounds}; namely,
\begin{equation}\label{Bounds:u_0h}
\|u_{0h}\|_{L^1(\Omega)}\le C_{\rm sta} \|u_0\|_{L^1(\Omega)}\quad\mbox{ and }\quad u_{0h}>0,
\end{equation}
and 
\begin{equation}\label{Bounds:v_0h}
\|v_{0h}\|_{H^1(\Omega)}\le C_{\rm sta} \|v_0\|_{H^1(\Omega)}\quad\mbox{ and }\quad v_{0h}\ge0.
\end{equation}

The discretization of problem \eqref{KS} will be based on its variational formulation. As a time integration we use an implicit Euler time-stepping method lagging in time the value of $v$ in $\eqref{KS}_1$ in order to decouple the computation between the unknowns. Given $N\in\mathds{N}$, we let $0 = t_0 < t_1 < ... < t_{N-1} < t_N = T$ be a uniform partitioning of [0,T] with  time step $k=\frac{T}{N}$. Moreover, we denote $\delta_t w_h^\np \doteq k^{-1} (w_h^\np - w_h^n)$. 
\subsection{Algorithm 1} The first proposed algorithm reads as follows.

\vskip0.5cm
\begin{center}
\noindent\fbox{
\begin{minipage}{0.9\textwidth}
\textbf{Algorithm 1}: Let $u_h^0=u_{0h}$ and $v^0_h=v_{0h}$.
\end{minipage}
}
\noindent\fbox{
\begin{minipage}{0.9\textwidth}
\noindent {\bf Step $(n+1)$:} Known $(u^n_h,v^n_h)\in X_h\times X_h$, find $(u^{n+1}_h, v^{n+1}_h)\in X_h\times X_h$ such that, for all $x_h\in X_h$,  
\begin{equation}\label{eq_alg1:u_h}
(\delta_t u^{n+1}_h, x_h)+(\nabla u^{n+1}_h, \nabla x_h)-(u^{n+1}_h\nabla v^n_h, \nabla x_h)+(B^u_1(u^{n+1}_h) u^{n+1}_h, x_h)=0
\end{equation}
and
\begin{equation}\label{eq_alg1:v_h}
\begin{array}{rcl}
(\delta_t v^{n+1}_h, x_h)&+&(\nabla v^{n+1}_h,\nabla x_h)+(v^{n+1}_h, x_h)+(B^v_1(v^{n+1}_h) v^{n+1}_h, x_h)=(u^{n+1}_h, x_h).
\end{array}
\end{equation}

\end{minipage}
}
\end{center}
\vskip0.5cm

The stabilizing terms above are defined as follows. For $\Xi=u$ or $v$, we consider   
\begin{equation}\label{B_1}
(B^\Xi_1(w_h)u_h, x_h)=\sum_{i\in I}\sum_{j\in I(\Omega_{\a_i})}\nu^\Xi_{ij}(w_h, v_h) u_j x_i \ell(i,j)
\end{equation}
with
$$
\nu_{ij}^\Xi(w_h, v_h)=\max\{\alpha_i(w_h) f^\Xi_{ij}, \alpha_j(w_h)f^\Xi_{ji}, 0\}\quad\mbox{ for }\quad i\not= j
$$
and
$$
\nu_{ii}^\Xi(w_h)=\sum_{j\in I(\Omega_{\a_i})\backslash\{i\}}\nu_{ij}^\Xi(w_h),
$$
where $f^\Xi_{ij}$ is given by 
$$
f^u_{ij}=(\nabla \varphi_{\a_j}, \nabla \varphi_{\a_i})-(\varphi_{\a_j}\nabla v_h, \nabla \varphi_{\a_i}) + k^{-1}(\varphi_{\a_j}, \varphi_{\a_i})
$$
and
$$
f^v_{ij}=(\nabla\varphi_{\a_j},\nabla\varphi_{\a_i}) + (k^{-1} + 1)(\varphi_{\a_j}, \varphi_{\a_i}),
$$
respectively. Additionally, $\ell(i,j)\doteq 2\delta_{ij}-1$ is the graph-Laplacian operator, where $\delta_{ij}$ is the Kronecker delta.

The shock detector  $\detector[i](w_h)$ modulates the action of the stabilizing terms. In the present work we use the shock detector proposed in \cite{Badia_Bonilla_2017}.  Let $\a_i\in\nodes$ and $\a_j\in\nodes(\Omega_{\a_i})\backslash\{\a_i\}$. Define $\rr_{ij} = \a_j - \a_i$ as being the vector pointing from nodes $\a_i$ to $\a_j$, with $\hat{\rr}_{ij} \doteq \frac{\rr_{ij}}{|\rr_{ij}|}$ being its normalized vector. Construct $\a_\ij^\sym$ as the point at the intersection between the line that passes through $\a_i$ and $\a_j$ and $\partial\Omega_{\a_i}$ that is not $\a_j$ (see Fig. \ref{fig:usym}). The set of all symmetric nodes with respect to node $\a_i$ is represented with $\nodes^\sym(\Omega_{\a_i})$. Finally, define $\rr_\ij^\sym \doteq \a_\ij^\sym - \a_i$, and $u_j^\sym \doteq u_h(\a_\ij^\sym)$. Then one can define the gradient's linear approximation of the jump and the mean at node $\a_i$ in direction $\rr_\ij$ as
\begin{align}
\jump{\gradient w_h}_\ij &\doteq \frac{w_j - w_i}{|\rr_\ij|} + \frac{w_j^\sym - w_i}{|\rr_\ij^\sym|}, 
\\ 
\mean{|\gradient w_h\cdot \hat{\rr}_\ij|}_{ij} &\doteq \frac{1}{2}
\left(\frac{|w_j-w_i|}{|\rr_\ij|}+\frac{|w_j^\sym-w_i|}{|\rr_\ij^\sym|}\right).
\end{align}
\begin{figure}
	\centering
	\includegraphics[width=0.25\textwidth]{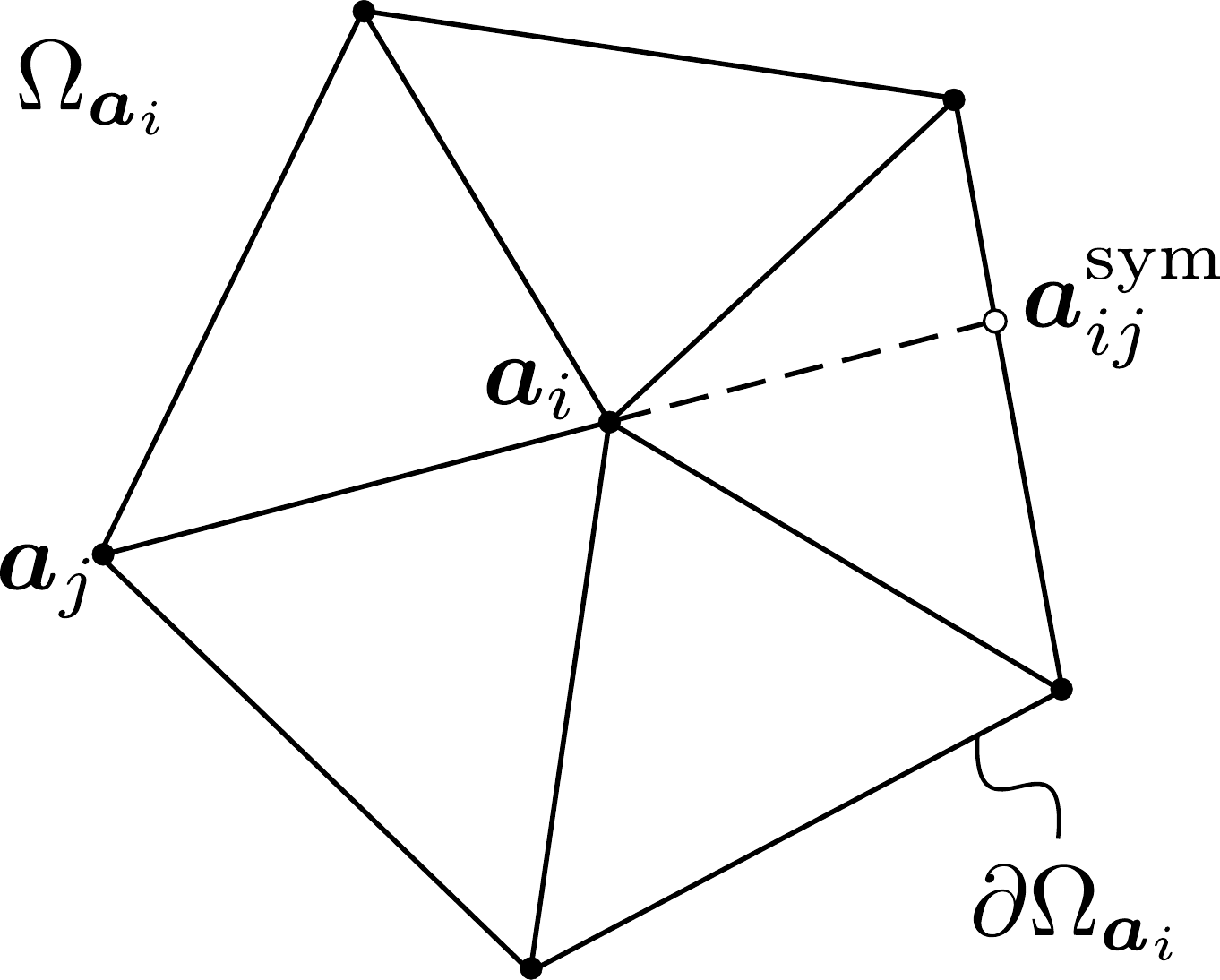}
	\hspace{5em}
	\includegraphics[width=0.21\textwidth]{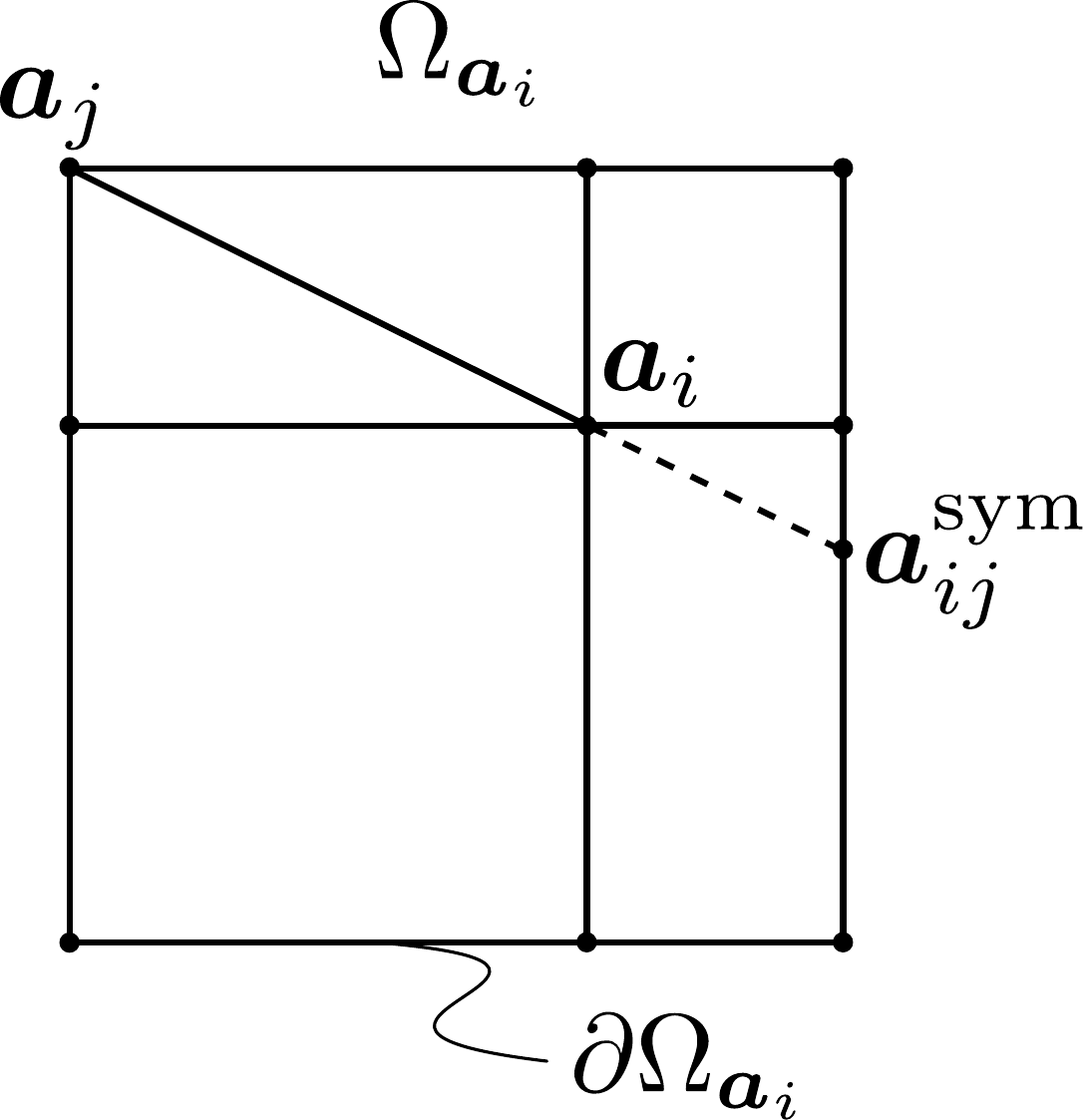}
	\caption{Representation of the symmetric node $\a_{ij}^{\sym}$ of $\a_j$ concerning $\a_i$ in a macroelement of triangles (left) and right quadrilaterals (right).}
	\label{fig:usym}
\end{figure}

Making use of these definitions, the shock detector at node $\boldsymbol{a}_i \in \nodes$ for $w_h$ reads:
\begin{equation}\label{def:alpha_min_max}
\alpha_i(w_h) \doteq \left\{\begin{array}{cc}  \left[\frac{\left|{\sum_{j\in I(\Omega_{\a_i})} \jump{\gradient w_h}_{ij}}\right|}{\sum_{j\in I(\Omega_{\a_i})} 2\mean{\left|\gradient w_h \cdot \hat{\rr}_{ij}\right|}_{ij}} \right]^q & \text{if } \sum_{j\in I(\Omega_{\a_i})} \mean{\left|\gradient w_h\cdot \hat{\rr}_{ij}\right|}_{ij} \neq 0, \\
0 & \text{otherwise},
\end{array}\right. 
\end{equation}
for some $q \in \mathbb{R}^+$. 

One key property of the shock detector is that it localizes local extremes. See    \cite[Lemma 3.1]{Badia_Bonilla_2017} for a proof.
\begin{lemma}\label{lm: alpha_i_alg1}
It follows that $0\leq \alpha_i(w_h)\leq 1$ and that $\detector[i](w_h)=1$ for any extreme value at $\a_i$.
\end{lemma}

\subsection{Algorithm 2} It is not obvious at all whether discrete solutions to Algorithm $1$ defined by equations  \eqref{eq_alg1:u_h}--\eqref{eq_alg1:v_h} have the discrete counterpart of \eqref{Energy_law}. Thus our next task is designing a new algorithm that leads to discrete solutions fulfilling \eqref{Energy_law}.  

The second algorithm is a variant of Algorithm 1 in which the chemotaxis and stabilizing terms are modified as follows. Let $x_h, \tilde x_h, \bar x_h\in X_h$ and assume $x_h>0$. Then, as $\sum_{i\in I} \nabla \varphi_{\a_i}=0$, one finds 
$$
\begin{array}{rcl}
( x_h\nabla \tilde x_h,\nabla \bar x_h)&=&\displaystyle
\sum_{k,j,i\in I} x_k \tilde x_j  \bar x_i (\varphi_{\a_k} \nabla \varphi_{\a_j}, \nabla\varphi_{\a_i}) 
\\
&=& \displaystyle
\sum_{k,j,i\in I} x_k(\tilde x_j-\tilde x_j)  \bar x_i (\varphi_{\a_k} \nabla \varphi_{\a_j}, \nabla\varphi_{\a_i})
\\
&=&2 \displaystyle
\sum_{k,j,i\in I} x_k(\tilde x_j-\tilde x_j)  (\bar x_i-\bar x_j) (\varphi_{\a_k} \nabla \varphi_{\a_j}, \nabla\varphi_{\a_i})
\\
&=&\displaystyle
\sum_{\tiny\begin{array}{c}k\in I\\i<j\in I\end{array}} x_k (\tilde x_j- \tilde x_i)  (\bar x_i-\bar x_j) (\varphi_{\a_k} \nabla \varphi_{\a_j}, \nabla\varphi_{\a_i}).
\end{array}
$$

From the simple observation that 
$$
\lim_{x_j\to x_i} \frac{x_j-x_i}{\log x_j - \log x_i} =x_i
$$
is an approximation of the identity, we now approximate $x_k$ by $x_{ij}$ where
\begin{equation}\label{x_ij}
x_{ij}=\left\{
\begin{array}{ccl}
\displaystyle
\frac{x_j-x_i}{\log x_j-\log x_i}&\mbox{ if }& x_j\not=x_i,
\\
x_i&\mbox{ if }& x_j=x_i.
\end{array}
\right.
\end{equation}
Observe as well that $\a_k, \a_j,\a_i$ must belong to $\Omega_{\a_k}\cap\Omega_{\a_j}\cap \Omega_{\a_i}$ so that $(\varphi_{\a_k} \nabla\varphi_{\a_j}, \nabla\varphi_{\a_i})\not=0$; 
therefore, as $h\to 0$, one has $\a_i, \a_j\to \a_k$, and hence $x_i:=x(\a_i)\to x_k:=x(\a_k)$. This way $x_{ij}\to x_k$ as $h\to 0$.  Thus
$$
( x_h\nabla \tilde x_h,\nabla \bar x_h)\approx\sum_{i<j\in I} x_{ji} (\tilde x_j- \tilde x_i)  (\bar x_i-\bar x_j) (\nabla \varphi_{\a_j}, \nabla\varphi_{\a_i}).
$$

Since we do not know \emph{a priori} whether the discrete solution will be positive, we need to construct an auxiliary function that somehow extends the logarithmic function to a non-positive value so that the coefficient \eqref{x_ij} makes sense. Let $\varepsilon>0$ and
$$
g_\varepsilon(s)=\left\{
\begin{array}{ccl}
s\log s-s&\mbox{ if }& s>\varepsilon,
\\
\frac{s^2-\varepsilon^2}{2\varepsilon}+(\log \varepsilon-1)s&\mbox{ if }& s\le \varepsilon,
\end{array}
\right.
$$
and hence
$$
g'_\varepsilon(s)=\left\{
\begin{array}{ccl}
\log s&\mbox{ if }& s>\varepsilon,
\\
\frac{s}{\varepsilon}+\log \varepsilon-1&\mbox{ if }& s\le \varepsilon. 
\end{array}
\right.
$$
Let us thus define 
\begin{equation}\label{KS_term_new}
(x_h\nabla \tilde x_h, \nabla \bar x_h)_*=\sum_{i<j\in I}\gamma_{ji}(x_h) ( \tilde x_j-\tilde x_i)(\bar x_i-\bar x_j)(\nabla \varphi_{\a_j},\nabla\varphi_{\a_i}),
\end{equation}
with
\begin{equation}\label{def:gamma_ij}
\gamma_{ji}(x_h)=\left\{
\begin{array}{ccl}
\displaystyle
\frac{x_j-x_i}{g'_\varepsilon(x_j)-g'_\varepsilon(x_i)}&\mbox{ if }& x_j\not=x_i,
\\
\left[x_i\right]_+&\mbox{ if }& x_j=x_i,
\end{array}
\right.
\end{equation}
where the operator $[x]_+=\max\{0,x\}$ stands for the positivity part. It should be noted that $g'_\varepsilon$ is bijective so $g'_\varepsilon(x_j)=g'_\varepsilon(x_i)$ implies $x_j=x_i$. 

\vskip0.5cm
Algorithm $2$ reads as follows. 
\vskip0.5cm
\begin{center}
\noindent\fbox{
\begin{minipage}{0.95\textwidth}
\textbf{Algorithm $2$}: Let $u_h^0=u_{0h}$ and $v^0_h=v_{0h}$.

\end{minipage}
}
\noindent\fbox{
\begin{minipage}{0.95\textwidth}
\noindent {\bf Step $(n+1)$:} Known $(u^n_h,v^n_h)\in X_h\times X_h$, find $(u^{n+1}_h, v^{n+1}_h)\in X_h\times X_h$ such that, for all $x_h\in X_h$,  
\begin{equation}\label{eq_alg2_eps:u_h}
\begin{array}{rcl}
(\delta_t u^{n+1}_h, x_h)_h&+&(\nabla u^{n+1}_h, \nabla x_h)-(u^{n+1}_h\nabla v^n_h, \nabla x_h)_*
\\
&+&(B^u_2(u^{n+1}_h, v_h^n)u^{n+1}_h, x_h)=0,
\end{array}
\end{equation}
and
\begin{equation}\label{eq_alg2_eps:v_h}
(\delta_t v^{n+1}_h, x_h)_h+(\nabla v^{n+1}_h,\nabla x_h)+(v^{n+1}_h, x_h)_h+\gamma(B_2^v(v^{n+1}_h) v^{n+1}_h,x_h)=(u^{n+1}_h, x_h)_h,
\end{equation}
with $\gamma\in\{0,1\}$.
\end{minipage}
}

\end{center}
\vskip0.5cm
Here the stabilizing term in \eqref{eq_alg2_eps:u_h} is given by
\begin{equation}\label{Bu_2}
(B^u_2(w_h, v_h)u_h, x_h)=\sum_{i<j\in I}\nu_{ji}^u(w_h,v_h) (u_j- u_i) (x_j-x_i),
\end{equation} 
where
$$
\nu_{ji}^u(w_h, v_h)=\max\{\bar\alpha_i(w_h) f_{ij}, \bar\alpha_j(w_h)f_{ji}, 0\}\quad\mbox{ for }\quad i\not= j,
$$
with
$$
f^u_{ij}=\left\{
\begin{array}{rcl}
\displaystyle
\left(1-\frac{v_j-v_i}{g'_\varepsilon(w_j)-g'_\varepsilon(w_i)}\right)(\nabla \varphi_{\a_j}, \nabla \varphi_{\a_i})&\mbox{ if }& w_j\not=w_i,
\\
0&\mbox{ if }&w_j=w_i.
\end{array}
\right.
$$
and
$$
\nu_{ii}^u(w_h, v_h)=\sum_{j\in I(\Omega_{\a_i})\backslash\{i\}}\nu_{ij}^u(w_h, v_h).
$$
The stabilizing term $B_2^v$ becomes
$$
f^v_{ij}=(\nabla\varphi_{\a_j},\nabla\varphi_{\a_i}).
$$
Now the shock detector at node $\boldsymbol{a}_i \in \nodes$ for $w_h$ is written as:
\begin{equation}\label{def:alpha_min}
\bar\alpha_i(w_h) \doteq \left\{\begin{array}{cc}  \left[\frac{\left[{\sum_{j\in I(\Omega_{\a_i})} \jump{\gradient w_h}_{ij}}\right]_+}{\sum_{j\in I(\Omega_{\a_i})} 2\mean{\left|\gradient w_h \cdot \hat{\rr}_{ij}\right|}_{ij}} \right]^q & \text{if } \sum_{j\in I(\Omega_{\a_i})} \mean{\left|\gradient w_h\cdot \hat{\rr}_{ij}\right|}_{ij} \neq 0, \\
0 & \text{otherwise},
\end{array}\right. 
\end{equation}
for some $q \in \mathbb{R}^+$. 

This shock detector is a modification of the one in \eqref{def:alpha_min_max}; in the numerator, we consider $[\cdot]_+$ instead of the absolute value. This way, the shock detector \eqref{def:alpha_min} only acts on minima. We have:

\begin{lemma}\label{lm: alpha_i_alg2}
It follows that $0\leq \bar\alpha_i(w_h)\leq 1$ and that $\bar\alpha_i(w_h)=1$ for any minimum value at $\a_i$.
\end{lemma}

Here it is worthwhile pointing out the contrasts between Algorithms $1$ and $2$. Basically, Algorithm~$1$ is a standard finite element method for which the nonlinear stabilizing terms $B_1^u$ and $B_1^v$ in \eqref{B_1} have been added to obtain \eqref{eq_alg1:u_h} and \eqref{eq_alg1:v_h}, respectively. Nevertheless, Algorithm 2 does use a mass lumping technique for some terms in equations \eqref{eq_alg2_eps:u_h} and \eqref{eq_alg2_eps:v_h}. The stabilization of equation \eqref{eq_alg2_eps:u_h} is based on the new discretization \eqref{KS_term_new} of the chemotaxis term, which obligates to redesign the stabilizing term $B_2^u$ in \eqref{Bu_2}  and the shock detector $\bar\alpha_i$ in \eqref{def:alpha_min} as well. 
\begin{remark}
In principle, there is no drawback with applying the shock detector $\alpha_i$ defined in \eqref{def:alpha_min_max}, which acts on both maxima and minima to limit the action of the stabilizing term \eqref{Bu_2}. The reason for not doing so is that the nonlinear solver used to obtain the solution of each time step does not work properly for some numerical tests. It is more natural to act only on minima, seeing that some solutions to the Keller-Segel equations present blowup phenomena at maxima. Therefore, introducing numerical diffusion at maxima might impact the values of the $L^\infty(\Omega)$-norm of discrete solutions to determine blowup configurations since maxima might not grow sufficiently. The effect of using \eqref{def:alpha_min}   
for defining \eqref{B_1} is innocuous in numerical examples.  
\end{remark}
\begin{remark} Another critical remark is concerned with the coefficients $f^u_{ij}$ and $f^v_{ij}$ in Algorithm~$2$. If one regards the mass-matrix entries in the definition of both coefficients; namely,  
$$
f^u_{ij}=\left\{
\begin{array}{rcl}
\displaystyle
k^{-1}(\varphi_{\a_j},\varphi_{\a_i})+\left(1-\frac{v_j-v_i}{g'_\varepsilon(w_j)-g'_\varepsilon(w_i)}\right)(\nabla \varphi_{\a_j}, \nabla \varphi_{\a_i})&\mbox{ if }& w_j\not=w_i,
\\
k^{-1}(\varphi_{\a_j},\varphi_{\a_i})&\mbox{ if }&w_j=w_i,
\end{array}
\right.
$$
and 
$$
f^v_{ij}=(\nabla\varphi_{\a_j},\nabla\varphi_{\a_i}) + (k^{-1} + 1)(\varphi_{\a_j}, \varphi_{\a_i}),
$$
mass lumping might be avoided. But, as for algorithm $1$, we cannot obtain a discrete version of \eqref{Energy_law} with this choice of the coefficients. 
\end{remark}

\section{Lower and $L^1(\Omega)$-bounds}\label{sec.low-l1-bounds}
\subsection{Algorithm 1}
To start with, we prove that Algorithm $1$ enjoys a discrete maximum principle under certain conditions on the discrete parameters $(h,k)$.
\begin{lemma}[Lower bounds]\label{lm:lower_bounds_alg1} Let $q\in (1,+\infty)$ for $d=2$ and $q=6$ for $d=3$.  Assume that $u_h^n>0$ and $v_h^n\ge0$  such that 
\begin{equation}\label{L2_bound:v^n_h}
\|v^n_h\|_{L^2(\Omega)}\le\|v_{0h}\|_{L^2(\Omega)}+ C \frac{T^\frac{1}{2}}{h^{\frac{d}{q}}}\|u_{0h}\|_{L^1(\Omega)}.
\end{equation}
Moreover, assume that $(k,h)$ are such that 
\begin{equation}\label{Restriction_(h,k)}
1-C \frac{k}{h^{2+\frac{d}{2}}} \Big(\|v_{0h}\|_{L^2(\Omega)}+ C\frac{T^\frac{1}{2}}{h^\frac{d}{q}}\|u_{h0}\|_{L^1(\Omega)}\Big)>0.
\end{equation}
Then it follows that the discrete solution pair $(u^{n+1}_h, v^{n+1}_h)$ provided by \eqref{eq_alg1:u_h}--\eqref{eq_alg1:v_h} is such that 
\begin{equation}\label{lower_bounds_alg1}
u^{n+1}_h>0\quad\mbox { and }\quad v^{n+1}_h\ge0
\end{equation}
hold.
\end{lemma}
\begin{proof} We proceed by contradiction. Suppose that there exists $\boldsymbol{a}_i\in\mathcal{N}_h$ being a local minimum such that $u^{n+1}_i:=u^{n+1}_h(\boldsymbol{a}_i)\le0$, and choose $x_h=\varphi_{\a_i}$ in \eqref{eq_alg1:u_h} to get
$$
\begin{array}{rcl}
\displaystyle
\sum_{j\in I(\Omega_{\a_i})} u_j^{n+1}\Big[k^{-1}(\varphi_{\a_j},\varphi_{\a_i})+(\nabla\varphi_{\a_j},\nabla\varphi_{\a_i}) -(\varphi_{\a_j}\nabla v^n_h, \nabla\varphi_{\a_i})&& 
\\
\displaystyle
+(B^u_1(u^{n+1}_h)\varphi_{\a_j},\varphi_{\a_i})\Big] -k^{-1}\sum_{j\in I(\Omega_{\a_i})} u_j^{n}(\varphi_{\a_j},\varphi_{\a_i}) &=& 0.
\end{array}
$$
From Lemma \ref{lm: alpha_i_alg1},  it is known that $\detector[i](u^{n+1}_h)=1$. Then, noting that, for $j\not= i$,
$$
\Big[k^{-1}(\varphi_{\a_j},\varphi_{\a_i}) + (\nabla\varphi_{\a_j},\nabla\varphi_{\a_i})-(\varphi_{\a_j}\nabla v^n_h,\nabla\varphi_{\a_i})+(B^u_1(u^{n+1}_h)\varphi_{\a_j},\varphi_{\a_i})\Big]\le0,
$$
we have
\begin{equation}\label{lm3.1-lab1}
\begin{array}{r}
\displaystyle
\sum_{j\in I(\Omega_{\a_i})} u_i^{n+1}\Big[k^{-1}(\varphi_{\a_j},\varphi_{\a_i})+(\nabla\varphi_{\a_j},\nabla\varphi_{\a_i}) -(\varphi_{\a_j}\nabla v^n_h,  \nabla\varphi_{\a_i})
\\ 
\displaystyle
+(B^u_1(u^{n+1}_h)\varphi_{\a_j},\varphi_{\a_i})\Big] -k^{-1}\sum_{j\in I(\Omega_{\a_i})} u_j^n(\varphi_{\a_j},\varphi_{\a_i}) \ge 0.
\end{array}
\end{equation}
Using the fact that 
\begin{equation}
\begin{array}{rcl}
0=(\nabla 1, \nabla \varphi_{\a_i}) &= & (\nabla v^n_h,\nabla\varphi_{\a_i})-k^{-1}(1, \varphi_{\a_i}) 
\\ 
& &\displaystyle
+\sum_{j\in I(\Omega_{\a_i})} \Big[ k^{-1}(\varphi_{\a_j}, \varphi_{\a_i}) + (\nabla \varphi_{\a_j}, \nabla \varphi_{\a_i} )- (\varphi_{\a_j}\nabla v^n_h, \nabla\varphi_{\a_i})\Big],
\end{array}
\end{equation}
we compute 
\begin{equation}\label{lm3.1-lab2}
\begin{array}{rcl}
k^{-1}(\varphi_{\a_i}, \varphi_{\a_i}) &+& (\nabla\varphi_{\a_i}, \nabla\varphi_{\a_i})-(\varphi_{\a_i}\nabla v^n_h,\nabla\varphi_{\a_i}) 
\\
&=& \displaystyle
k^{-1}(1,\varphi_{\a_i})-(\nabla v^n_h,\nabla\varphi_{\a_i})
\\
&&\displaystyle
-\sum_{j\in I(\Omega_{\a_i})\backslash\{i\}} k^{-1}(\varphi_{\a_j}, \varphi_{\a_i}) + (\nabla \varphi_{\a_j}, \nabla \varphi_{\a_i} )- (\varphi_{\a_j}\nabla v^n_h,\nabla\varphi_{\a_i}).
\end{array}
\end{equation}
Thus, from \eqref{B_1} for $\Xi=u$ and \eqref{lm3.1-lab2}, we can write 
$$
\begin{array}{c}
k^{-1}(\varphi_{\a_i}, \varphi_{\a_i}) + (\nabla\varphi_{\a_i}, \nabla \varphi_{\a_i})-(\varphi_{\a_i} \nabla v^n_h, \nabla \varphi_{\a_i})+(B^u_1(u^{n+1}_h)\varphi_{\a_i},\varphi_{\a_i})
\\
=k^{-1}(\varphi_{\a_i}, \varphi_{\a_i}) + (\nabla\varphi_{\a_i}, \nabla \varphi_{\a_i})-(\varphi_{\a_i} \nabla v^n_h, \nabla \varphi_{\a_i})+\nu_{ii}^u(u^{n+1}_h)
\\
\displaystyle
=k^{-1}(\varphi_{\a_i}, \varphi_{\a_i})+(\nabla\varphi_{\a_i}, \nabla \varphi_{\a_i})-(\varphi_{\a_i} \nabla v^n_h, \nabla \varphi_{\a_i})+\sum_{j\in I(\Omega_{\a_i})\backslash\{i\}} \max\{ f^u_{ij}, \alpha_j(u^{n+1}_h)f^u_{ji}, 0\}
\\
\displaystyle
=k^{-1}(\varphi_{\a_i}, \varphi_{\a_i})+(\nabla\varphi_{\a_i}, \nabla \varphi_{\a_i})-(\varphi_{\a_i} \nabla v^n_h, \nabla \varphi_{\a_i})-\sum_{j\in I(\Omega_{\a_i})\backslash\{i\}} \max\{ f^u_{ij}, \alpha_j(u^{n+1}_h)f^u_{ji}, 0\}\varphi_{\a_j}\varphi_{\a_i}\ell(i,j)
\\
\\
\displaystyle
=k^{-1}(\varphi_{\a_i}, \varphi_{\a_i})+(\nabla\varphi_{\a_i}, \nabla \varphi_{\a_i})-(\varphi_{\a_i} \nabla v^n_h, \nabla \varphi_{\a_i})-\sum_{j\in I(\Omega_{\a_i})\backslash\{i\}} ( B^u_1(u^{n+1}_h) \varphi_{\a_j},\varphi_{\a_i})
\\
\displaystyle
=k^{-1}(1, \varphi_{\a_i})-(\nabla v^n_h,\nabla\varphi_{\a_i})-\sum_{j\in I(\Omega_{\a_i})\backslash\{i\}} k^{-1}(\varphi_{\a_j}, \varphi_{\a_i}) + (\nabla\varphi_{\a_j}, \nabla \varphi_{\a_i})-(\varphi_{\a_j} \nabla v^n_h, \nabla \varphi_{\a_i})
\\
\displaystyle
-\sum_{j\in I(\Omega_{\a_i})\backslash\{i\}} (B^u_1(u^{n+1}_h) \varphi_{\a_j},\varphi_{\a_i})
\end{array}
$$
and hence
\begin{equation}\label{lm3.1-lab3}
\begin{array}{r}
\displaystyle
\sum_{j\in I(\Omega_{\a_i})} \Big[k^{-1}(\varphi_{\a_j}, \varphi_{\a_i}) + (\nabla\varphi_{\a_j}, \nabla \varphi_{\a_i})-(\varphi_{\a_j} \nabla v^n_h, \nabla \varphi_{\a_i})+(B^u_1(u^{n+1}_h)\varphi_{\a_j},\varphi_{\a_i}) \Big] \\ 
 = k^{-1}(1 , \varphi_{\a_i})-(\nabla v^n_h,\nabla \varphi_{\a_i}).
\end{array}
\end{equation}

Substituting \eqref{lm3.1-lab3} back into \eqref{lm3.1-lab1}, we get 
$$
0\le \|\varphi_{\a_i}\|_{L^1(\Omega)} u_i^{n+1}- u_i^\np k(\nabla v^n_h, \nabla\varphi_{\a_i}) - \sum_{j\in I(\Omega_{\a_i})} (\varphi_{\a_j}, \varphi_{\a_i}) u_j^n
$$
or, equivalently,
$$
- u^{n+1}_i\le - \sum_{j\in I(\Omega_{\a_i})} \frac{(\varphi_{\a_j}, \varphi_{\a_i})}{\|\varphi_{\a_i}\|_{L^1(\Omega)} } u_j^n -\frac{k}{\|\varphi_{\a_i}\|_{L^1(\Omega)} } (\nabla v^n_h, \nabla\varphi_{\a_i}) u^{n+1}_i
$$
If $(\nabla v^n_h, \nabla\varphi_{\a_i})\le0$ holds, then it follows that 
$$
0\le- u^{n+1}_i \left(1-\frac{k}{\|\varphi_{\a_i}\|_{L^1(\Omega)} } (\nabla v^n_h, \nabla\varphi_{\a_i})\right)\le - \sum_{j\in I(\Omega_{\a_i})} \frac{(\varphi_{\a_j}, \varphi_{\a_i}) }{\|\varphi_{\a_i}\|_{L^1(\Omega)} } u_j^n <0,
$$
which is a contradiction since $u^{n+1}_i\le 0$. Otherwise, if $(\nabla v^n_h, \nabla\varphi_{\a_i}) \geq 0$ holds, we have, by \eqref{inv_ineq}, that  
$$
\begin{array}{rcl}
- u^{n+1}_i&\le&\displaystyle
- \sum_{j\in I(\Omega_{\a_i})} \frac{(\varphi_{\a_j}, \varphi_{\a_i}) }{\|\varphi_{\a_i}\|_{L^1(\Omega)} } u_j^n + \frac{k}{\|\varphi_{\a_i}\|_{L^1(\Omega)} } (\nabla v^n_h, \nabla\varphi_{\a_i}) (-u^{n+1}_i) 
\\
&\le&\displaystyle
- \sum_{j\in I(\Omega_{\a_i})} \frac{(\varphi_{\a_j}, \varphi_{\a_i}) }{\|\varphi_{\a_i}\|_{L^1(\Omega)} } u_j^n + \frac{k}{\|\varphi_{\a_i}\|_{L^1(\Omega)}}\|\nabla v^n_h\|_{L^2(\Omega)} \|\nabla\varphi_{\a_i}\|_{L^2(\Omega)}(-u^{n+1}_i)
\\
&\le&\displaystyle
- \sum_{j\in I(\Omega_{\a_i})} \frac{(\varphi_{\a_j}, \varphi_{\a_i}) }{\|\varphi_{\a_i}\|_{L^1(\Omega)} } u_j^n + C\frac{k}{h^{2+\frac{d}{2}}}\|v^n_h\|_{L^2(\Omega)} (-u^{n+1}_i).
\end{array}
$$
As a result, under conditions \eqref{L2_bound:v^n_h} and \eqref{Restriction_(h,k)},
$$
(1-\frac{k}{h^{2+\frac{d}{2}}}(\|v_{0h}\|_{L^2(\Omega)}+ C \frac{T^\frac{1}{2}} {h^{\frac{d}{q}}}\|u_{0h}\|_{L^1(\Omega)}))(-u^{n+1}_i)\le - \sum_{j\in I(\Omega_{\a_i})} \frac{(\varphi_{\a_j}, \varphi_{\a_i})}{\|\varphi_{\a_i}\|_{L^1(\Omega)}} u_j^n\le0,
$$
again contradicting our assumption on $u^{n+1}_i\le 0$.

The proof of $v^{n+1}_h\ge0$ is easier. Choose $x_h = \varphi_{\a_i} $ in \eqref{eq_alg1:v_h} to get
\begin{align*}
\sum_{j\in I(\Omega_{\a_i})} v_j^\np\Big[(1+k^{-1})(\varphi_{\a_j},\varphi_{\a_i}) +& (\nabla\varphi_{\a_j},\nabla\varphi_{\a_i}) + (B^v_1(v_h^\np)\varphi_{\a_j},\varphi_{\a_i}) \Big]  = \\
&\sum_{j\in I(\Omega_{\a_i})} (\varphi_{\a_j},\varphi_{\a_i}) (k^{-1}v_j^n + u_j^\np).
\end{align*}
By \eqref{B_1}, for $\Xi=v$ and  $j\neq i$, we have 
$$
\Big[(1+k^{-1})(\varphi_{\a_j},\varphi_{\a_i}) +(\nabla\varphi_{\a_j},\varphi_{\a_i}) + (B^v_1(v_h^\np)\varphi_{\a_j},\varphi_{\a_i})\Big] \leq 0.
$$
Thus, 
$$
\begin{array}{rcl}
\displaystyle
\sum_{j\in I(\Omega_{\a_i})} v_i^\np\Big[(1+k^{-1})(\varphi_{\a_j},\varphi_{\a_i}) +(\nabla\varphi_{\a_j},\nabla\varphi_{\a_i}) + (B^v_1(v_h^\np)\varphi_{\a_j},\varphi_{\a_i}) \Big]  \geq &&
\\
\displaystyle
\sum_{j\in I(\Omega_{\a_i})} (\varphi_{\a_j},\varphi_{\a_i}) (k^{-1}v_j^n + u_j^\np) &\geq& 0.
\end{array}
$$
Since $(\nabla1,\nabla \varphi_{\a_j})=0$ and $(B^v_1(v_h^\np) 1,\varphi_{\a_i})=0$ by \eqref{B_1} for $\Xi=v$ again, we see that 
$$
v_i^\np (1+k^{-1})\|\varphi_{\a_i}\|_{L^1(\Omega)} \geq \sum_{j\in I(\Omega_{\a_i})} (\varphi_{\a_j},\varphi_{\a_i}) (k^{-1}v_j^n + u_j^\np),
$$
where $\|\varphi_{\a_i}\|_{L^1(\Omega)} = \sum_{j\in I(\Omega_{\a_i})} (\varphi_{\a_j},\varphi_{\a_i})$. Therefore,
$$
0> v_i^\np (1+k^{-1}) \|\varphi_{\a_i}\|_{L^1(\Omega)} \geq \sum_{j\in I(\Omega_{\a_i})} (\varphi_{\a_j}, \varphi_{\a_i}) (k^{-1}v_j^n + u_j^\np) > 0,
$$
which is a contradiction. It closes the proof.
\end{proof}
In light of condition \eqref{L2_bound:v^n_h} of Lemma \ref{lm:lower_bounds_alg1}, $L^1(\Omega)$-bounds are the crux of positivity in \eqref{eq_alg1:u_h}. We prove these bounds in the following lemma. 

\begin{lemma}[$L^1(\Omega)$-bounds]\label{lm:L1-alg1} Suppose  $(u^{n+1}_h, v_h^{n+1})\in X_h^2$ solves  \eqref{eq_alg1:u_h} and \eqref{eq_alg1:v_h} with $u^{n+1}_h>0$ and $v^{n+1}_h\ge0$. Then the following estimates hold:
\begin{equation}\label{L1-Bound-uh_alg1}
\|u^{n+1}_h\|_{L^1(\Omega)}=\|u_h^0\|_{L^1(\Omega)}
\end{equation}
and
\begin{equation}\label{L1-Bound-vh_alg1}
\|v^{n+1}_h\|_{L^1(\Omega)}\le \|v_{0h}\|_{L^1(\Omega)}+\|u_{0h}\|_{L^1(\Omega)}.
\end{equation}

\end{lemma}
\begin{proof} First of all, observe that 
\begin{equation}\label{lm3.2_lab1}
(B^u_1(u^{n+1}_h) u^{n+1}_h, 1)=0
\end{equation}
and 
\begin{equation}\label{lm3.2_lab2}
(B^v_1(v^{n+1}_h)v^{n+1}_h, 1)=0.
\end{equation}
Substitute $x_h=1$ into \eqref{eq_alg1:u_h} to get 
\begin{equation}\label{lm3.2-lab3}
\int_\Omega u^{n+1}_h(\boldsymbol{x})\,\dx=\int_\Omega u^0_h(\x)\,\dx,
\end{equation}
which implies \eqref{L1-Bound-uh_alg1}, owing to $u_{0h}>0$, $u^{n+1}_h>0$ and \eqref{lm3.2_lab1}. Now let $x_h=1$ in \eqref{eq_alg1:v_h} to get 
$$
\int_\Omega v^{n+1}_h(\x)\,\dx+k \int_\Omega v^{n+1}_h(\x)\,\dx=\int_\Omega v^n_h(\x)\,\dx+k\int_\Omega u^{n+1}_h(\x)\,\dx
$$
in view of \eqref{lm3.2_lab2}. A straightforward calculation shows that 
$$
\int_\Omega v^{n+1}_h(\x)\,\dx=\frac{1}{(1+k)^{n+1}} \int_\Omega v^0_h(\x)\,\dx+\left(\int_\Omega u^0_h(\x)\,\dx\right) \sum_{j=1}^{n+1}\frac{k}{(1+k)^j},
$$
where we have used \eqref{lm3.2-lab3}. Inequality \eqref{L1-Bound-vh_alg1} is then proved by invoking $v^0_h\ge0$ and $v^{n+1}_h\ge0$.
\end{proof}

Now, we can prove \eqref{L2_bound:v^n_h} for $n+1$. 
\begin{lemma}\label{lm:L2-bound} Assume that $u^{n+1}_h>0$ and $v^{n+1}_h\ge0$ hold. Then the discrete solution $v^{n+1}_h$ generated by \eqref{eq_alg1:v_h} satisfies  
\begin{equation}\label{L2_bound:v^{n+1}_h}
\|v^{n+1}_h\|^2_{L^2(\Omega)}\le C \frac{T}{h^{\frac{2d}{q}}} \|u_{0h}\|^2_{L^1(\Omega)}+\|v_{0h}\|_{L^2(\Omega)}^2.
\end{equation}
\end{lemma}
\begin{proof} We insert $x_h=v_h^{n+1}$ into \eqref{eq_alg1:v_h} to get 
\begin{equation}\label{lm3.3-lab1}
\begin{array}{rcl}
\|v^{n+1}_h\|^2_{L^2(\Omega)}&+& \|v^{n+1}_h-v^n_h\|^2_{L^2(\Omega)}+ 2 k \|v^{n+1}_h\|^2_{H^1(\Omega)}
\\
&+&(B^v_1(v^{n+1}_h)v^{n+1}_h,v^{n+1}_h)=2\, k(u^{n+1}_h, v^{n+1}_h)+ \|v^n_h\|^2_{L^2(\Omega)}.
\end{array}
\end{equation}
The first term on the right-hand side of \eqref{lm3.3-lab1} is estimated as follows. Let $q\in (1,\infty)$ for $d=2$ and $q=6$ for $d=3$ with $p$ being its conjugate, i. e., $\frac{1}{p}+\frac{1}{q}=1$. Then, Hölder's, Young's, and Sobolev's inequalities yield 
$$
\begin{array}{rcl}
2\, k(u^{n+1}_h, v^{n+1}_h)&\le&2k \|u^{n+1}_h\|_{L^p(\Omega)}\|v^{n+1}_h\|_{L^q(\Omega)} 
\\
&\le& C k \|u^{n+1}_h\|^2_{L^p(\Omega)}+k \|v^{n+1}_h\|^2_{H^1(\Omega)}.
\end{array}
$$
On applying \eqref{inv_ineq} and recalling \eqref{L1-Bound-uh_alg1}, it is straightforward to deduce that 
\begin{equation}\label{lm3.3-lab2}
2\, k(u^{n+1}_h, v^{n+1}_h)\le C \frac{k}{h^{\frac{2d}{q}}} \|u_{0h}\|^2_{L^1(\Omega)}+k \|v^{n+1}_h\|^2_{H^1(\Omega)}.
\end{equation}
If we compile \eqref{lm3.3-lab1} and \eqref{lm3.3-lab2}, we find, after summing adequately, that \eqref{L2_bound:v^{n+1}_h} holds.
\end{proof} 

An induction argument on $n$ applied to \eqref{lower_bounds_alg1}, \eqref{L1-Bound-uh_alg1}, \eqref{L1-Bound-vh_alg1}, and \eqref{L2_bound:v^{n+1}_h}  leads to the following. 
\begin{theorem} Assume that $(h,k)$ are such that \eqref{Restriction_(h,k)} holds. Then it follows that the sequence of the discrete solution pair $\{(u^n_h, v^n_h)\}_{n=0}^N$ defined by Algorithm 1 satisfies 
\begin{itemize}
\item[a)] Lower bounds:
$$ u^n_h>0\quad\mbox{ and }\quad v^n_h\ge 0.$$
\item[b)]$L^1(\Omega)$-bounds:
$$
\|u^n_h\|_{L^1(\Omega)}=\|u_{0h}\|_{L^1(\Omega)}
$$
and
$$
\|v^{n}_h\|_{L^1(\Omega)}\le \|v_{0h}\|_{L^1(\Omega)}+\|u_{0h}\|_{L^1(\Omega)}.
$$
\end{itemize}
for all $n=0, \cdots, N$.
\end{theorem}
\begin{proof} The induction argument requires to verify that \eqref{L2_bound:v^n_h} holds for $n=0$. Then, according to Lemma \ref{lm:lower_bounds_alg1}, the lower bounds \eqref{lower_bounds_alg1} are satisfied for $n=0$ from \eqref{Bounds:u_0h} and\eqref{Bounds:v_0h}. In Lemma \ref{lm:L1-alg1}, we see that  the $L^1(\Omega)$-bounds \eqref{L1-Bound-uh_alg1} and \eqref{L1-Bound-vh_alg1} hold for $n=0$. Finally Lemma \ref{lm:L2-bound} provides the $L^2(\Omega)$-bound \eqref{L2_bound:v^{n+1}_h} for $n=0$.  The general case follows from sequentially  applying Lemmas \ref{lm:lower_bounds_alg1}, \ref{lm:L1-alg1} and \ref{lm:L2-bound}, respectively.    
\end{proof}

\subsection{Algorithm 2}
We now address the question of lower bounds for Algorithm 2. 
\begin{lemma}[Lower bounds]\label{lm:lower_bounds_alg2} Let $\gamma=1$.  Assume that $u_h^n>0$ and $v^n_h\ge0$. Then it follows that the discrete solution $(u^{n+1}_h, v^{n+1}_h)$ to \eqref{eq_alg2_eps:u_h}--\eqref{eq_alg2_eps:v_h} satisfies 
\begin{equation}\label{lower_bounds_alg2}
u^{n+1}_h>0\quad\mbox{ and }\quad v^{n+1}_h\ge 0.
\end{equation} 
\end{lemma}
\begin{proof} Suppose that there exists $\boldsymbol{a}_i\in\mathcal{N}_h$ being the minimum such that $u^{n+1}_i:=u^{n+1}_h(\boldsymbol{a}_i)\le 0$. Moreover, without loss of generality,  we may assume that $i<j$ for all $ j\in I(\Omega_{\a_i})$; otherwise one should only take $ j\in I(\Omega_{\a_i})$ such that $i<j$. Let $I^*(\Omega_{\a_i})=\{j\in I(\Omega_{\a_i}): u^{n+1}_j=u^{n+1}_i\}$ and let  $I^*_c(\Omega_{\a_i})$ be  its complementary.  Then choose $x_h=\varphi_{\a_i}$ in \eqref{eq_alg2_eps:u_h} to get, using
$$
(u^{n+1}_h\nabla  v^n_h, \nabla \varphi_i)_*=\sum_{j\in I^*_c(\Omega_{\a_i})} \frac{v^n_j-v^n_i}{g'_\varepsilon(u^{n+1}_j)-g'_\varepsilon(u^{n+1}_i)}(\nabla \varphi_{\a_j},\nabla\varphi_{\a_i}) (u^{n+1} _j- u^{n+1}_i)(\varphi_{\a_i}(\a_i)-\varphi_{\a_i}(\a_j)),
$$
that 
\begin{align*}
k^{-1}(1,\varphi_{\a_i}) u^{n+1}_i&+\sum_{j\in I^*_c(\Omega_{\a_i})} a_{ji}(u^{n+1}_h,v^n_h)  (u_j^{n+1}-u_i^{n+1}) (\varphi_i(\a_i)-\varphi_i(\a_j)  )
\\
&+\sum_{j\in I^*_c(\Omega_{\a_i})}
\nu_{ji}^u(u^{n+1}_h,v^n_h) (u^{n+1}_j- u^{n+1}_i)(\varphi_{\a_i}(\a_j)-\varphi_{\a_i}(\a_i))
\\
= k^{-1} u_i^{n} (1,\varphi_{\a_i})&,
\end{align*}
where we used
$$
\sum_{j\in I(\Omega_{\a_i})} u_j^{n+1}(\nabla \varphi_{\a_j},\nabla\varphi_{\a_i})=\sum_{j\in I_c(\Omega_{\a_i})} (\nabla\varphi_{\a_j},\nabla\varphi_{\a_i})(u_j^{n+1}-u^{n+1}_i)(\varphi_i(\a_i)-\varphi_i(\a_j)  ),
$$
and defined
$$
a_{ji}(u^{n+1}_h,v^n_h)=\left(1-\frac{v^n_j-v^n_i}{g'_\varepsilon(u^{n+1}_j)-g'_\varepsilon(u^{n+1}_i)}\right)(\nabla\varphi_{\a_j},\nabla\varphi_{\a_i}).
$$
Therefore,
\begin{equation}\label{lm3.4-lab1}
\begin{array}{rll}
\displaystyle
k^{-1}(1, \varphi_{\a_i}) u^{n+1}_i&\displaystyle
+\sum_{j\in I^*_c(\Omega_{\a_i})}\Big[a_{ji}(u^{n+1}_h,v^n_h)
-\nu^u_{ji}(u^{n+1}_h, v^n_h)\Big] (u_j^{n+1}-u_i^{n+1})
\\
&= k^{-1}(1,\varphi_{\a_i})u^n_i.
\end{array}
\end{equation}
Since $\bar
\alpha_i(u^{n+1}_h)=1$ from Lemma \ref{lm: alpha_i_alg2}, we have
\begin{equation}\label{lm3.4-lab2}
a_{ji}(u^{n+1}_h,v^n_h)-\nu^u_{ji}(u^{n+1}_h, v^n_h)\le 0\quad \mbox{ for all }\quad j\in I^*_c(\Omega_{\a_i}).
\end{equation}

Applying \eqref{lm3.4-lab2} to \eqref{lm3.4-lab1} yields
$$
0<u^n_i\le  u^{n+1}_i,
$$
which contradicts our assumption.

An argument similar to the one in the proof of 
 Lemma \ref{lm:lower_bounds_alg1} shows that $v^{n+1}_h\ge0$. 
\end{proof}

\begin{remark} Once we proved that the discrete solution $u^{n+1}_h$ is positive, the use of the truncating operator in \eqref{def:gamma_ij} can be neglected. 
\end{remark}

The following lemma is the counterpart of Lemma \ref{lm:L1-alg1}.  
\begin{lemma}[$L^1(\Omega)$-bounds] Under the conditions of Lemma \ref{lm:lower_bounds_alg2}, the discrete solution pair \linebreak $(u^{n+1}_h, v_h^{n+1})\in X_h^2$ computed via \eqref{eq_alg2_eps:u_h}--\eqref{eq_alg2_eps:v_h} fulfills 
\begin{equation}\label{L1-Bound-uh_alg2_eps}
\|u^{n+1}_h\|_{L^1(\Omega)}=\|u_h^0\|_{L^1(\Omega)}
\end{equation}
and
\begin{equation}\label{L1-Bound-vh_alg2_eps}
\|v^{n+1}_h\|_{L^1(\Omega)}\le \|v^0_h\|_{L^1(\Omega)}+\|u^0_h\|_{L^1(\Omega)}.
\end{equation}

\end{lemma}
\begin{proof} Using the reasoning of Lemma \ref{lm:L1-alg1} proves \eqref{L1-Bound-uh_alg2_eps} and \eqref{L1-Bound-vh_alg2_eps} from testing \eqref{eq_alg2_eps:u_h} and \eqref{eq_alg2_eps:v_h} by $x_h=1$ and noting
$$
(B^u_2(u^{n+1}_h,v^n_h) u^{n+1}_h, 1)=0
$$
and 
$$
(B^v_2(v^{n+1}_h)v^{n+1}_h, 1)=0. 
$$
 
\end{proof}

It is now shown that system \eqref{eq_alg2_eps:u_h}-\eqref{eq_alg2_eps:v_h} for $\gamma=0$ possesses a discrete energy law associated with
$$
\mathcal{E}_h(u_h,v_h)\doteq\frac{1}{2}\|\nabla v_h\|^2+\frac{1}{2}\|v_h\|^2_h-(u_h, v_h)_h+(u_h, \log u_h)_h.
$$

\begin{lemma}[A discrete energy law]\label{lm:energy_law_alg2} Let $\gamma=0$. Assume that $\mathcal{T}_h$ is a weakly acute triangulation for $d=2$ and acute triangulation for $d=3$ and define $I^*=\{(i,j)\in I\times I : u^{n+1}_i= u^{n+1}_j\}$ with $I^*_c$ being its complement. Then the discrete solution $(u^{n+1}_h, v^{n+1}_h)\in X_h^2$ computed via \eqref{eq_alg2_eps:u_h} and \eqref{eq_alg2_eps:v_h} satisfies 
\begin{align}
\mathcal{E}_h(u^{n+1}_h,v^{n+1}_h)-\mathcal{E}_h(u^{n}_h,v^{n}_h)&+\mathcal{ND}( u^{n+1}_h,v^{n+1}_h)+k\|\delta_t v^{n+1}_h\|^2_h
\nonumber
\\
-k\displaystyle
\sum_{i<j\in I^*_c}&(1-\bar\alpha_\#(u^{n+1}_h))\left|\left(\frac{g'_\varepsilon(u^{n+1}_j)-g'_\varepsilon(u^{n+1}_i)}{u^{n+1}_j-u^{n+1}_i}\right)^{-\frac{1}{2}}(u^{n+1}_j-u^{n+1}_i)\right.
\nonumber
\\
&\left.-\left(\frac{g'_\varepsilon(u^{n+1}_j)-g'_\varepsilon(u^{n+1}_i)}{u^{n+1}_j-u^{n+1}_i}\right)^{\frac{1}{2}}(v^n_j-v^n_i)\right|^2(\nabla\varphi_{\a_j},\nabla\varphi_{\a_i})
\label{Local-Energy-Law}
\\
-k\sum_{i<j\in I^*}&u^{n+1}_i(v^n_j-v^n_i)^2 (\nabla\varphi_{\a_j},\nabla\varphi_{\a_i})
=0,
\nonumber
\end{align}
where
$$
\begin{array}{rcl}
\mathcal{ND}(u^{n+1}_h,v^{n+1}_h)&=&\displaystyle
\|\nabla(v^{n+1}_h-v^n_h)\|^2_{L^2(\Omega)}+\|v^{n+1}_h-v^n_h\|^2_h+k (g''_\varepsilon(u^{n+\theta}_h),(\delta_t u^{n+1}_h)^2)_h,
\end{array}
$$ and  $\bar\alpha_\#\in\{\alpha_i, \alpha_j, 0\}$ is chosen such that 
$$\nu^u_{ji}(u^{n+1}_h,v^n_h)=\bar\alpha_\#(u^{n+1}_h)(1-\frac{v^n_{\bar \#}-v^n_\#}{g'_\varepsilon( u^{n+1}_{\bar \#})-g'_\varepsilon(u^{n+1}_\#)}(\nabla\varphi_{\a_{\bar \#}},\nabla\varphi_{\a_\#}),$$
with ${\bar \#}=j$ when $\#=i$ and ${\bar \#}=i$ when $\#=j$.
\end{lemma}
\begin{proof}  Select $x_h= i_h g'_\varepsilon(u^{n+1}_h)-v^n_h$ in \eqref{eq_alg2_eps:u_h} and $x_h=\delta_t v^{n+1}_h$ in \eqref{eq_alg2_eps:v_h} to get
\begin{equation}\label{lm4.2-lab1}
\begin{array}{rcl}
(\delta_t u^{n+1}_h, i_h g'_\varepsilon(u^{n+1}_h)-v^n_h)_h&+&(\nabla u^{n+1}_h, \nabla(i_h g'_\varepsilon(u^{n+1}_h)-v^n_h))
\\
&&-(u^{n+1}_h\nabla v^{n}_h, \nabla (i_h g'_\varepsilon(u^{n+1}_h)-v_h^n))_*
\\
&&+(B^u_2(u^{n+1}_h,v^n_h) u^{n+1}_h,  i_h g'_\varepsilon(u^{n+1}_h)-v^n_h)=0
\end{array}
\end{equation}
and
\begin{equation}\label{lm4.2-lab2}
\begin{array}{l}
\displaystyle
\frac{1}{2 k}\|\nabla v^{n+1}_h\|^2_{L^2(\Omega)}+\frac{1}{2 k}\|v^{n+1}_h\|^2_h-\frac{1}{2 k}\|\nabla v^{n}_h\|^2_{L^2(\Omega)}-\frac{1}{2 k}\|v^n_h\|^2_h
\\[1.5ex]
\displaystyle
+\frac{1}{2k}\|\nabla(v^{n+1}_h-v^n_h)\|^2_{L^2(\Omega)}+\frac{1}{2k}\|v^{n+1}_h-v^n_h\|^2_h+\|\delta_t v^{n+1}_h\|^2_h-(u_h^{n+1}, \delta_t v^{n+1}_h)_h=0.
\end{array}
\end{equation}

We next pair some terms from \eqref{lm4.2-lab1} and \eqref{lm4.2-lab2} in order to handle them together. It is not hard to see that  
\begin{equation}\label{lm4.2-lab3}
-(\delta_t u^{n+1}_h, v^n_h)_h-(u_h^{n+1}, \delta_t v^{n+1}_h)_h=-\frac{1}{k}(u^{n+1}_h, v^{n+1}_h)_h+\frac{1}{k}(u^n_h, v^n_h)_h.
\end{equation}

Now write 
$$
\begin{array}{rcl}
(\nabla u^{n+1}_h, \nabla(i_h g_\varepsilon'(u^{n+1}_h)-v_h^n))&=&\displaystyle
-\sum_{i<j\in I^*_c}\frac{g'_\varepsilon(u^{n+1}_j)-g'_\varepsilon(u^{n+1}_i)}{u^{n+1}_j-u^{n+1}_i} (u^{n+1}_j-u^{n+1}_i)^2(\nabla \varphi_{\a_j},\nabla\varphi_{\a_i})
\\
&&\displaystyle
+\sum_{i<j\in I^*_c}(u^{n+1}_j-u^{n+1}_i)(v^n_j-v^n_i)(\nabla\varphi_{\a_j},\nabla\varphi_{\a_i})
\end{array}
$$
and
$$
\begin{array}{rcl}
(u^{n+1}_h\nabla v^{n}_h, \nabla (i_h g'_\varepsilon(u^{n+1}_h)-v_h^n))_*&=&\displaystyle
-\sum_{i<j\in I^*_c}(u^{n+1}_j-u^{n+1}_i)(v^n_j-v^n_i)(\nabla\varphi_{\a_j},\nabla\varphi_{\a_i})
\\
&&\displaystyle
+\sum_{i<j\in I^*_c}\frac{u^{n+1}_j-u^{n+1}_i}{g'_\varepsilon(u^{n+1}_j)-g'_\varepsilon(u^{n+1}_i)}(v^n_j-v^n_i)^2(\nabla\varphi_{\a_j},\nabla\varphi_{\a_i})
\\
&&\displaystyle
+\sum_{i<j\in I^*}u^{n+1}_i(v^n_j-v^n_i)^2(\nabla\varphi_{\a_j},\nabla\varphi_{\a_i}).
\end{array}
$$
Therefore,
\begin{align*}
(\nabla u^{n+1}_h, \nabla(i_h g'_\varepsilon(u^{n+1}_h)-v_h^n))-(u^{n+1}_h\nabla v^{n}_h, &\nabla (i_h g_\varepsilon'(u^{n+1}_h)-v_h^n))_*
\\
=\displaystyle
-\sum_{i<j\in I^*_c}&\left|\left(\frac{u^{n+1}_j-u^{n+1}_i}{g'_\varepsilon(u^{n+1}_j)-g'_\varepsilon(u^{n+1}_i)}\right)^{\frac{1}{2}}(u^{n+1}_j-u^{n+1}_i)\right.
\\
&\left.
-\left(\frac{g'_\varepsilon(u^{n+1}_j)-g'_\varepsilon(u^{n+1}_i)}{u^{n+1}_j-u^{n+1}_i}\right)^{\frac{1}{2}}(v^n_j-v^n_i)\right|^2(\nabla\varphi_{\a_j},\nabla\varphi_{\a_i})
\\
\displaystyle
-\sum_{i<j\in I^*}&u^{n+1}_i(v^n_j-v^n_i)^2(\nabla\varphi_{\a_j},\nabla\varphi_{\a_i})>0,
\end{align*}
since $(\nabla\varphi_{\a_j},\nabla\varphi_{\a_i})\le0$ due to the acuteness of $\mathcal{T}_h$.
Moreover, 
\begin{align}
(B_2(u^{n+1}_h,v^n_h) u^{n+1}_h, &i_h g'_\varepsilon(u^{n+1}_h)-v^n_h)
\nonumber
\\
&=\displaystyle
\sum_{i<j\in I^*}\nu^u_{ji}(u^{n+1}_h,v^n_h)(u^{n+1}_i-u^{n+1}_j)(g'_\varepsilon(u^{n+1}_j)-g'_\varepsilon(u^{n+1}_i)-(v^n_i-v^n_j)).
\nonumber
\end{align}
Let us assume that 
$$\nu^u_{ji}(u^{n+1}_h,v^n_h)=\bar\alpha_i(u^{n+1}_h)(1-\frac{v^n_j-v^n_i}{g'_\varepsilon( u^{n+1}_j)-g'_\varepsilon(u^{n+1}_i)}(\nabla\varphi_{\a_j},\nabla\varphi_{\a_i}).$$
Then 
\begin{align*}
(1-\frac{v^n_j-v^n_i}{g'_\varepsilon(u^{n+1}_j)-g_\varepsilon(u^{n+1}_i)})(u^{n+1}_j-u^{n+1}_i)(g'_\varepsilon(u^{n+1}_j)-g'_\varepsilon(u^{n+1}_i)-(v^n_i-v^n_j))
\\
=\frac{g'_\varepsilon(u^{n+1}_j)-g'_\varepsilon(u^{n+1}_i)}{u^{n+1}_j-u^{n+1}_i} (u^{n+1}_j-u^{n+1}_i)^2-2(v^n_j-v^n_i)(u^{n+1}_j-u^{n+1}_i)
\\
+\frac{u^{n+1}_j-u^{n+1}_i}{g'_\varepsilon(u^{n+1}_j)-g'_\varepsilon(u^{n+1}_i)}(v^n_j-v^n_i)^2.
\\
=\left|\left(\frac{u^{n+1}_j-u^{n+1}_i}{g'_\varepsilon(u^{n+1}_j)-g'_\varepsilon(u^{n+1}_i)}\right)^{\frac{1}{2}}(u^{n+1}_j-u^{n+1}_i)-\left(\frac{g'_\varepsilon(u^{n+1}_j)-g'_\varepsilon(u^{n+1}_i)}{u^{n+1}_j-u^{n+1}_i}\right)^{\frac{1}{2}}(v^n_j-v^n_i)\right|^2.
\end{align*}
As a result, we have
\begin{align}
(\nabla u^{n+1}_h, &\nabla(i_h g'_\varepsilon(u^{n+1}_h)-v_h^n))-(u^{n+1}_h\nabla v^{n}_h, \nabla (i_h g'_\varepsilon(u^{n+1}_h)-v_h^n))_*
\nonumber
\\
&+(B^u_2(u^{n+1}_h,v^n_h) u^{n+1}_h, i_h g'_\varepsilon(u^{n+1}_h)-v^n_h)
\nonumber
\\
=\displaystyle
-\sum_{i<j\in I^*}&(1-\bar\alpha_i(u^{n+1}_h))\left|\left(\frac{g'_\varepsilon(u^{n+1}_j)-g'_\varepsilon(u^{n+1}_i)}{u^{n+1}_j-u^{n+1}_i}\right)^{-\frac{1}{2}}(u^{n+1}_j-u^{n+1}_i)\right.
\label{lm4.2-lab4}
\\
&\left.-\left(\frac{g'_\varepsilon(u^{n+1}_j)-g'_\varepsilon(u^{n+1}_i)}{u^{n+1}_j-u^{n+1}_i}\right)^{\frac{1}{2}}(v^n_j-v^n_i)\right|^2(\nabla\varphi_{\a_j},\nabla\varphi_{\a_i})>0,
\nonumber
\end{align}
since $0\le\bar\alpha_i(u^{n+1}_h)\le1$. Analogously, when
$$\nu^u_{ji}(u^{n+1}_h,v^n_h)=\bar\alpha_j(u^{n+1}_h)\left(1-\frac{v^n_i-v^n_j}{g'_\varepsilon(u^{n+1}_i)-g'_\varepsilon(u^{n+1}_j)}\right)(\nabla\varphi_{\a_i},\nabla\varphi_{\a_j})
$$ 
holds, a Taylor polynomial of $g_\varepsilon$ round $u^{n+1}_h$ evaluated at $u^{n}_h$ yields
$$
g_\varepsilon(u^n_h)=g_\varepsilon(u^{n+1}_h)-g'_\varepsilon(u^{n+1}_h)(u^{n+1}_h-u^n_h)+\frac{g''_\varepsilon(u^{n+\theta}_h)}{2}(u^{n+1}_h-u^n_h)^2,
$$
where $\theta\in (0,1)$ such that $u^{n+\theta}_h=\theta u^{n+1}_h+(1-\theta) u^n_h$. Hence, 
$$
(\delta_t u^{n+1}_h, g'_\varepsilon(u^{n+1}_h))_h=\frac{1}{k}(g_\varepsilon(u^{n+1}_h),1)_h- \frac{1}{k}(g_\varepsilon(u^{n}_h),1)_h+ \frac{k}{2} (g''_\varepsilon(u^{n+\theta}_h),(\delta_t u^{n+1}_h)^2)_h.
$$

The equality \eqref{Local-Energy-Law} follows by adding \eqref{lm4.2-lab1} and \eqref{lm4.2-lab2} and invoking \eqref{lm4.2-lab3} and \eqref{lm4.2-lab4}. 

\end{proof}

\begin{remark} Lemma \ref{lm:energy_law_alg2} can readily be (at least) extended to meshes with right quadrilaterals ($d=2$) or hexahedra ($d=3$) when bilinear finite elements are used, since $(\nabla\varphi_{\a_i},\nabla\varphi_{\a_j})\le0$ for all $i\not= j\in I$. 
\end{remark}

\begin{theorem} Let $\gamma=0$. Assume that $\mathcal{T}_h$ is a weakly acute triangulation for $d=2$ and acute triangulation for $d=3$. Then it follows that the sequence of the discrete solution pairs $\{(u^n_h, v^n_h)\}_{n=0}^N$ defined by Algorithm 2 satisfies 
\begin{itemize}
\item[a)] Lower bounds:
$$ u^n_h>0\quad\mbox{ and }\quad v^n_h\ge 0$$
\item[b)]$L^1(\Omega)$-bounds:
$$
\|u^n_h\|_{L^1(\Omega)}=\|u_{0h}\|_{L^1(\Omega)}
$$
and
$$
\|v^{n}_h\|_{L^1(\Omega)}\le \|v_{0h}\|_{L^1(\Omega)}+\|u_{0h}\|_{L^1(\Omega)}.
$$
\item [c)] Energy law: 
\begin{align}
\mathcal{E}_h(u^{n}_h,v^{n}_h)-\mathcal{E}_h(u^{0}_h,v^{0}_h)&+\mathcal{ND}( v^{n+1}_h)+k\|\delta_t v^{n+1}_h\|^2_{L^2(\Omega)}
\nonumber
\\
-k\displaystyle\sum_{m=0}^{n-1}
\sum_{i<j\in I^*_c} &(1-\bar\alpha_\#(u^{m+1}_h))\left|\left(\frac{g'_\varepsilon(u^{m+1}_j)-g'_\varepsilon(u^{m+1}_i)}{u^{m+1}_j-u^{m+1}_i}\right)^\frac{1}{2}(u^{m+1}_j-u^{m+1}_i)\right.
\nonumber
\\
&\left.-\left(\frac{u^{m+1}_j-u^{m+1}_i}{g'_\varepsilon(u^{m+1}_j)-g'_\varepsilon(u^{m+1}_i)}\right)^{-\frac{1}{2}}(v^m_j-v^m_i)\right|^2(\nabla\varphi_{\a_j},\nabla\varphi_{\a_i})
\label{Global-Energy-Law}
\\
-k\sum_{i<j\in I^*}&u^{n+1}_i(v^n_j-v^n_i)^2 (\nabla\varphi_{\a_j},\nabla\varphi_{\a_i})
=0,
\nonumber
\end{align}
\end{itemize}
for all $n=1, \cdots, N$.
\end{theorem}

\begin{remark} Observe that
$$
\lim_{\varepsilon\to 0} g_{\varepsilon}(s)=\log s\quad\mbox{ for all }\quad s>0.
$$
Therefore, if we take the limit as $\varepsilon\to 0$,  we formally obtain, from \eqref{lower_bounds_alg2}, Algorithm $2$ with the following definitions. The stabilizing term is now given by
$$
(B^u_2(u^{n+1}_h, v^n_h)u^{n+1}_h, x_h)=\sum_{i<j\in I}\nu_{ji}^u(u^{n+1}_h,v^n_h) (u^{n+1}_j- u^{n+1}_i) (x_j-x_i),
$$
where
$$
\nu_{ji}^u(u^{n+1}_h, v^n_h)=\max\{\bar\alpha_i(u^{n+1}_h) f_{ij}, \bar\alpha_j(u^{n+1}_h)f_{ji}, 0\}\quad \mbox{ for }\quad i\not= j,
$$
with
$$
f^u_{ij}=\left\{
\begin{array}{rcl}
\displaystyle
\left(1-\frac{v^n_j-v^n_i}{\log(u^{n+1}_j)-\log(u^{n+1}_i)}\right)(\nabla \varphi_{\a_j}, \nabla \varphi_{\a_i})&\mbox{ if }& u^{n+1}_j\not=u^{n+1}_i,
\\
0&\mbox{ if }&u^{n+1}_j=u^{n+1}_i,
\end{array}
\right.
$$
and the chemotaxis term is approximated as 
$$
(x_h\nabla \tilde x_h, \nabla \bar x_h)_*=\sum_{i<j\in I}\gamma_{ji}(x_h) ( \tilde x_j-\tilde x_i)(\bar x_i-\bar x_j)(\nabla \varphi_{\a_j},\nabla\varphi_{\a_i}),
$$
where
$$
\gamma_{ji}(x_h)=\left\{
\begin{array}{ccl}
\displaystyle
\frac{x_j-x_i}{\log(x_j)-\log(x_i)}&\mbox{ if }& x_j\not=x_i,
\\
x_i&\mbox{ if }& x_j=x_i.
\end{array}
\right.
$$
\end{remark}

\section{Numerical results}
Two new algorithms have been proposed and analyzed in the previous sections for approximating problem \eqref{KS}-\eqref{BC}. Discrete solutions to Algorithms $1$ and Algorithm $2$ for $\gamma=1$ have lower bounds and preserve mass, whereas discrete solutions to Algorithm $2$ for $\gamma=0$ satisfy a discrete energy law as well provided that the mesh be weakly acute in dimension two or acute in dimension three.   

In this section, we present numerical results for two test problems using both algorithms to validate the theoretical results: The first one has a smooth solution and is used to investigate the effect of the numerical diffusion stemming from the stabilizing terms. The second example concerns blowup phenomena; thus, we assess if the two stabilized algorithms can deal with solutions that might develop singularities in finite time.

As Algorithms $1$ and $2$ are nonlinear, Picard's method with backtracking is used to carry out the iterations for solving the nonlinear system at each time step. The two stopping criteria used are to require the $L^\infty(\Omega)$-norm of the residual to be less than $10^{-6}$ or the increment to be less than $10^{-16}$. Moreover, we take $q=2$ for Algorithms $1$ and $2$ in \eqref{def:alpha_min_max} and \eqref{def:alpha_min}, respectively. Recall that the shock detector \eqref{def:alpha_min_max} for Algorithm 1 acts on both maxima and minima, whereas the shock detector $\eqref{def:alpha_min}$ acts on minima only. In addition, for Algorithm $2$, we set $\gamma=1$ and $\varepsilon=10^{-6}$. 

\subsection{Smooth coalescence}
In this test we consider the approximation of the Keller-Segel model \eqref{KS}-\eqref{BC} for $\Omega=(-\pi,\pi)^2$ and the initial conditions
$$
    u_0 = \sin^2 (x) \sin^2 (y) \quad \text{and} \quad v_0 = \cos(x) + \cos(y) + 2.
$$
For both algorithms, we consider $X_h$ to be a bilinear finite element space constructed over a uniformly structured $40\times40$ grid having mesh size $h =0.11107$. The time step size is $k =0.02$.

Since $\int_\Omega u_0(\x)\,\dx=\pi^2\in(0, 4\pi)$, the expected dynamics is a smooth coalescence of the cell density because the chemoattractant density initially concentrates around the center of the domain.

The figref{fig.e1-conservation} shows the evolution in time of the $L^1(\Omega)$-norm for the chemoattractant and cell densities. In particular, the mass conservation is numerically verified, as Lemmata \ref{lm:lower_bounds_alg1} and \ref{lm:lower_bounds_alg2} predict for the cell density. It also shows the competition between production and degradation of chemoattractant, where the latter prevails over the former, which implies a mass loss.
\begin{figure}
    \begin{subfigure}[b]{0.32\textwidth}
        \centering
        \includegraphics[width=1.0\textwidth]{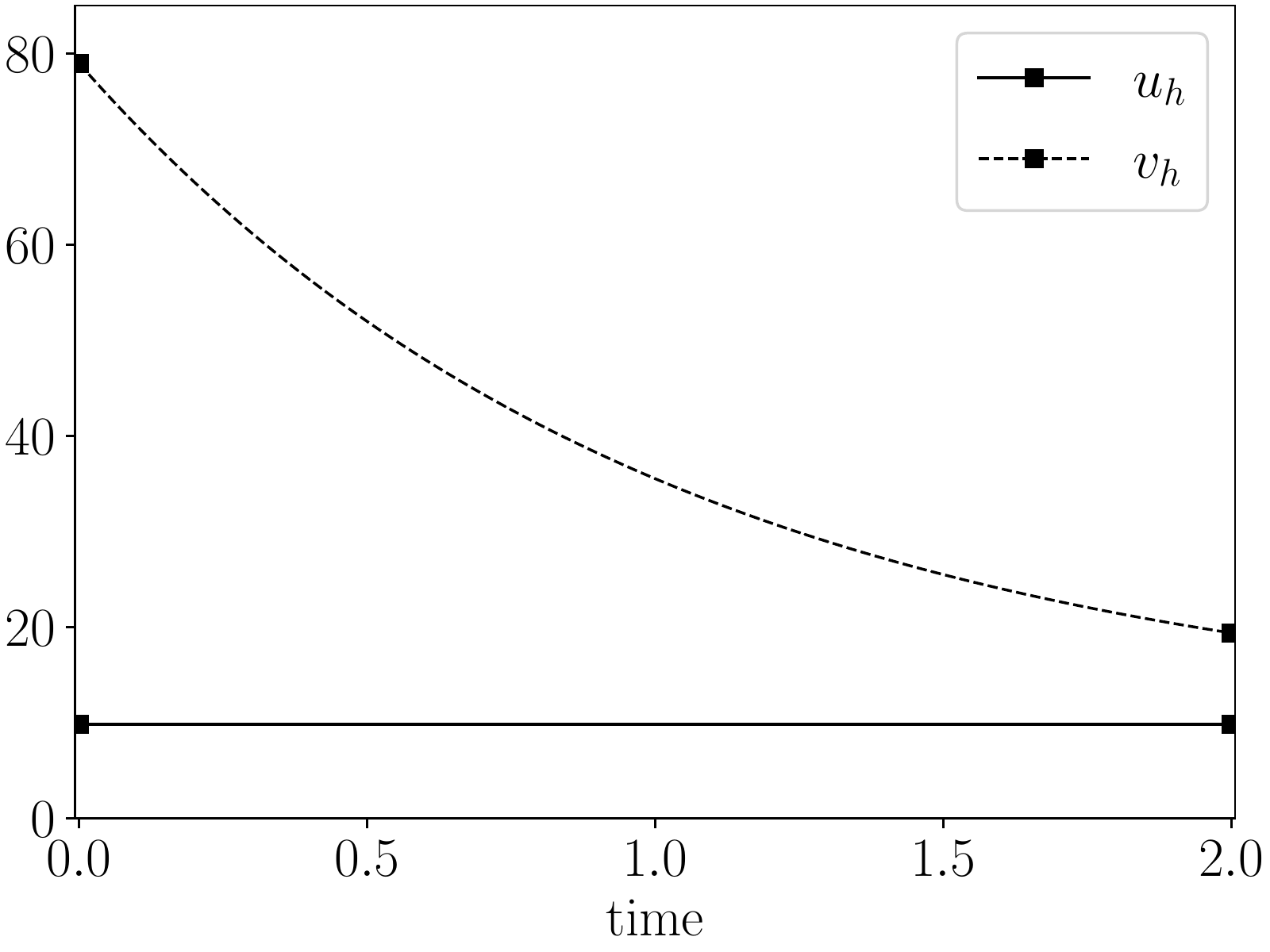}
	    \caption*{Algorithm 1}
    \end{subfigure}
    \begin{subfigure}[b]{0.32\textwidth}
        \centering
        \includegraphics[width=1.0\textwidth]{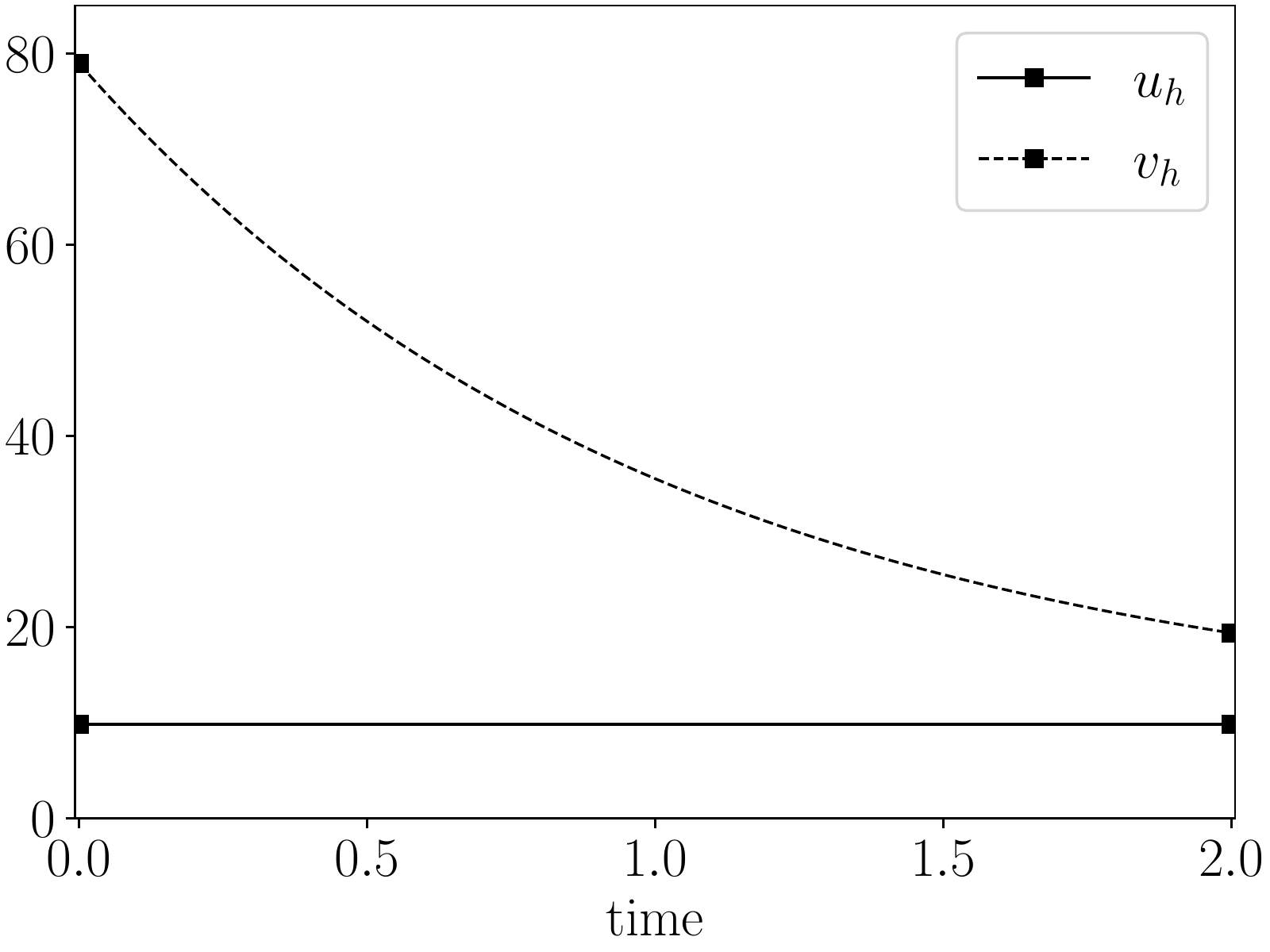}
	    \caption*{Algorithm 2}
    \end{subfigure}
    \begin{subfigure}[b]{0.32\textwidth}
        \centering
        \includegraphics[width=1.0\textwidth]{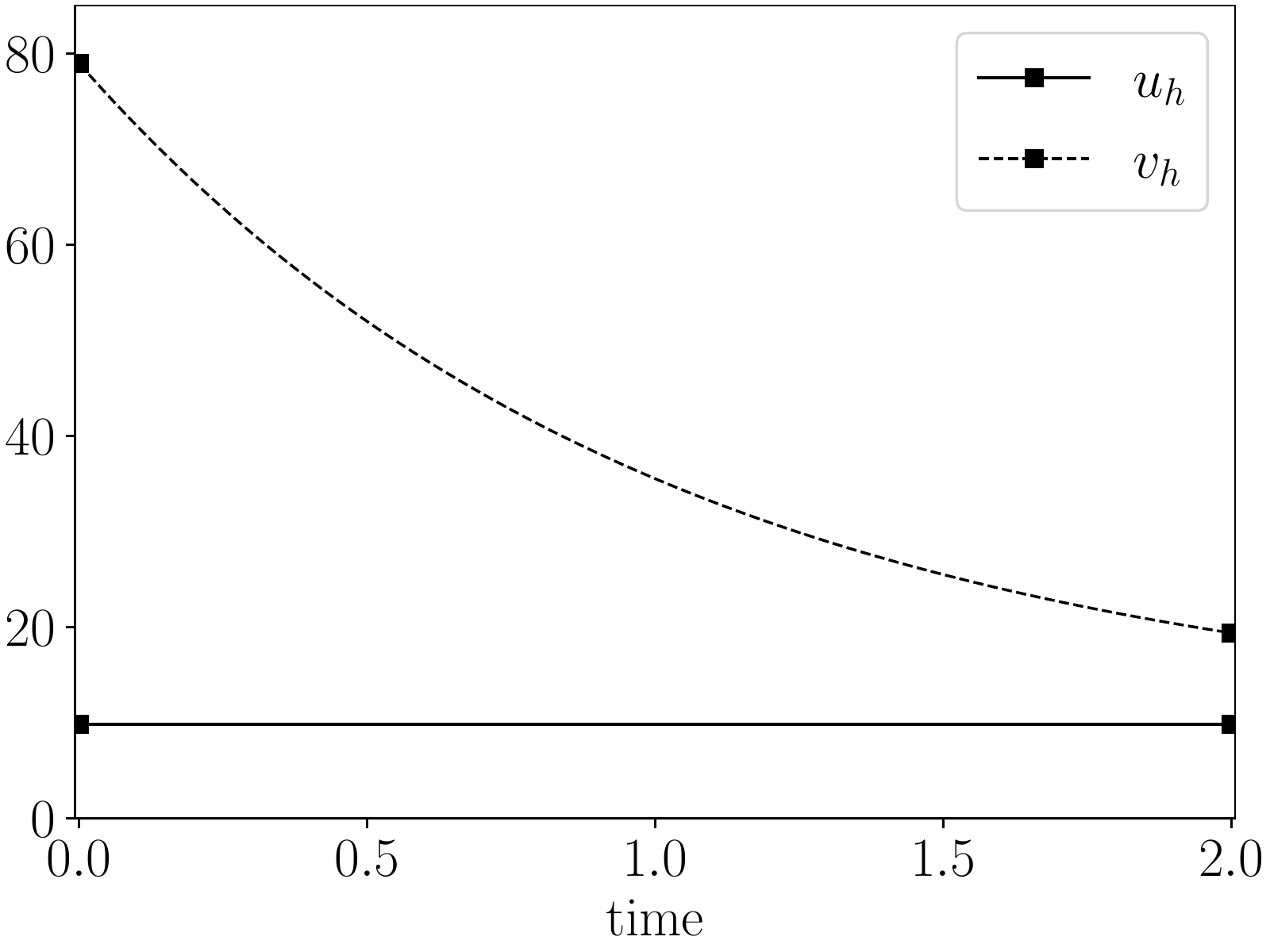}
	    \caption*{No stabilization}
    \end{subfigure}
    \caption{Evolution of $\|u_h\|_{L^1(\Omega)}$ and  $\|v_h\|_{L^1(\Omega)}$.}
    \label{fig.e1-conservation}
\end{figure}

As can be seen from Figure \ref{fig.e1-evolution-uh} (first column), cells are initially distributed in four groups with the highest concentrations in the center of each quadrant. Instead, the highest concentration of chemoattractant is located at  the origin as shown in Figure\ref{fig.e1-evolution-vh}. Then the four groups of cells move toward the origin attracted by the chemoattractant while being diffused as displayed in Figures \ref{fig.e1-evolution-uh} and \ref{fig.e1-evolution-vh} at times $t=1.3$ and $2$. Furthermore, the profile along the diagonal of the domain and the evolution of the energy functional $\mathcal{E}_h(u_h, v_h)$ can be found in Figures \ref{fig.e1-profiles} at times $t=0$ and $2$, and Figure \ref{fig.e1-energy}, respectively. It is evident that $\mathcal{E}_h(u_h,v_h)$ decays over time; therefore, the system evolves toward a steady-state solution.    
\begin{figure}
    \begin{subfigure}[b]{0.3\textwidth}
        \centering
        \includegraphics[width=1.0\textwidth]{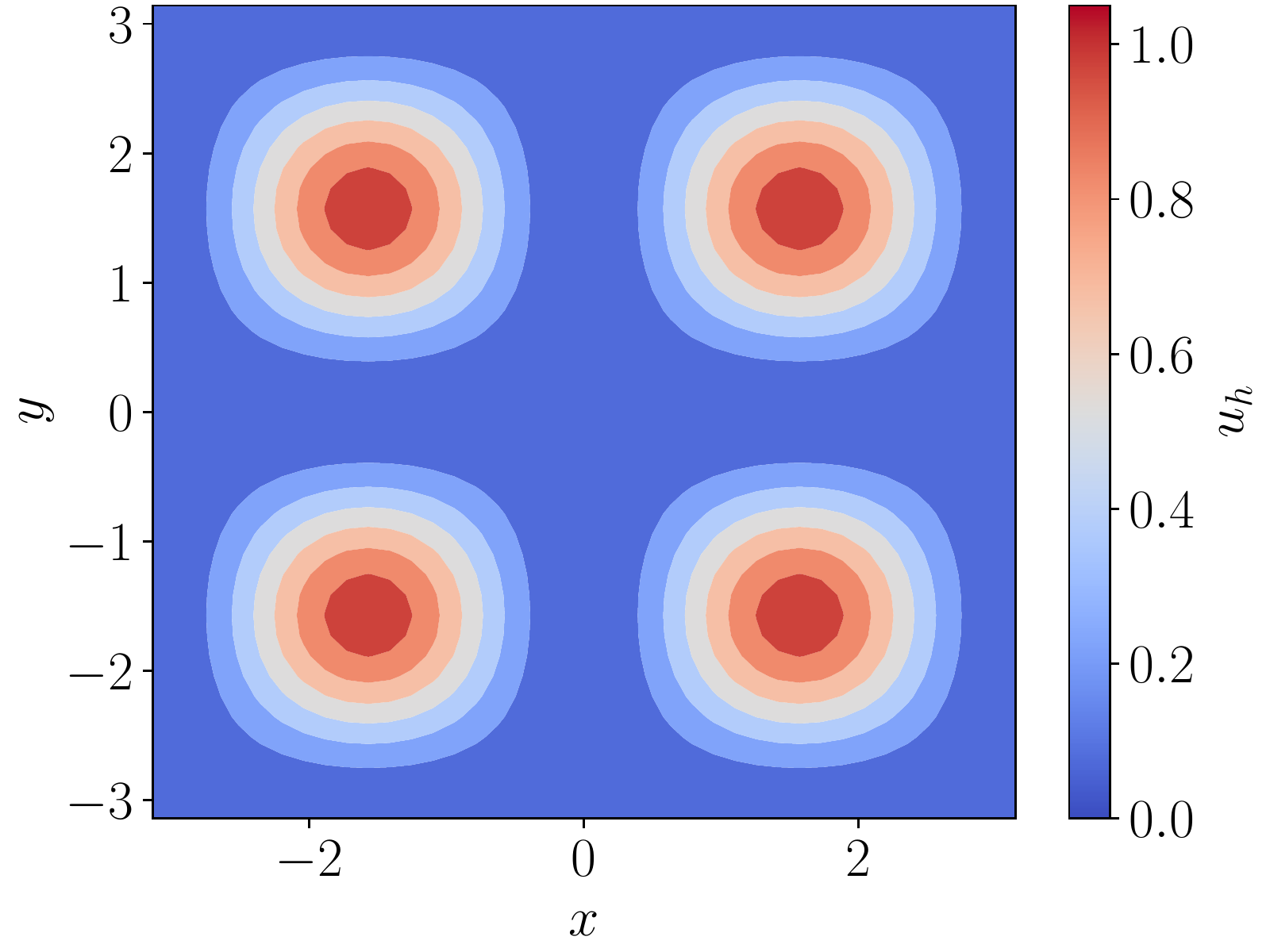}
    \end{subfigure}
    \begin{subfigure}[b]{0.3\textwidth}
        \centering
        \includegraphics[width=1.0\textwidth]{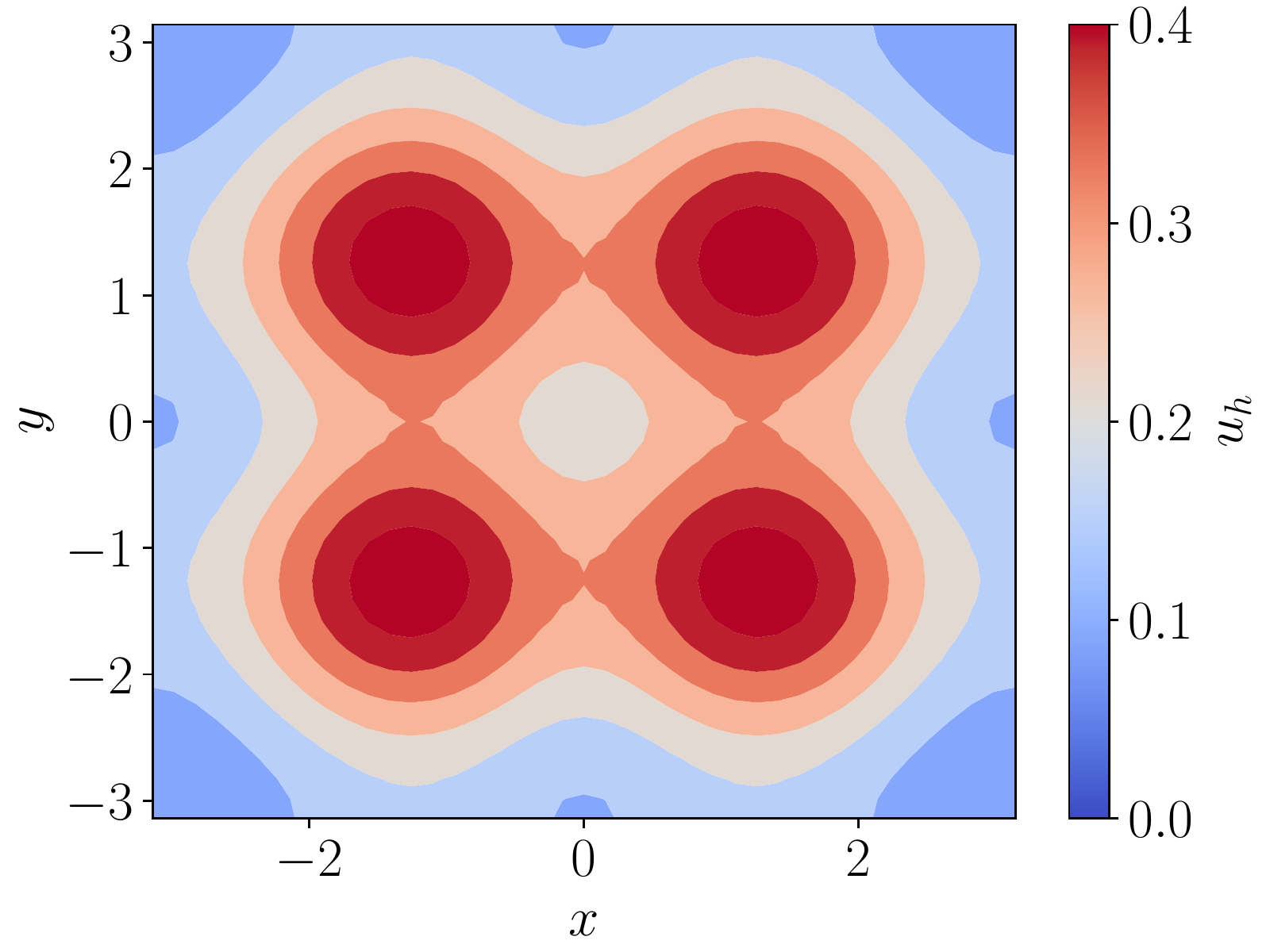}
    \end{subfigure}
    \begin{subfigure}[b]{0.3\textwidth}
        \centering
        \includegraphics[width=1.0\textwidth]{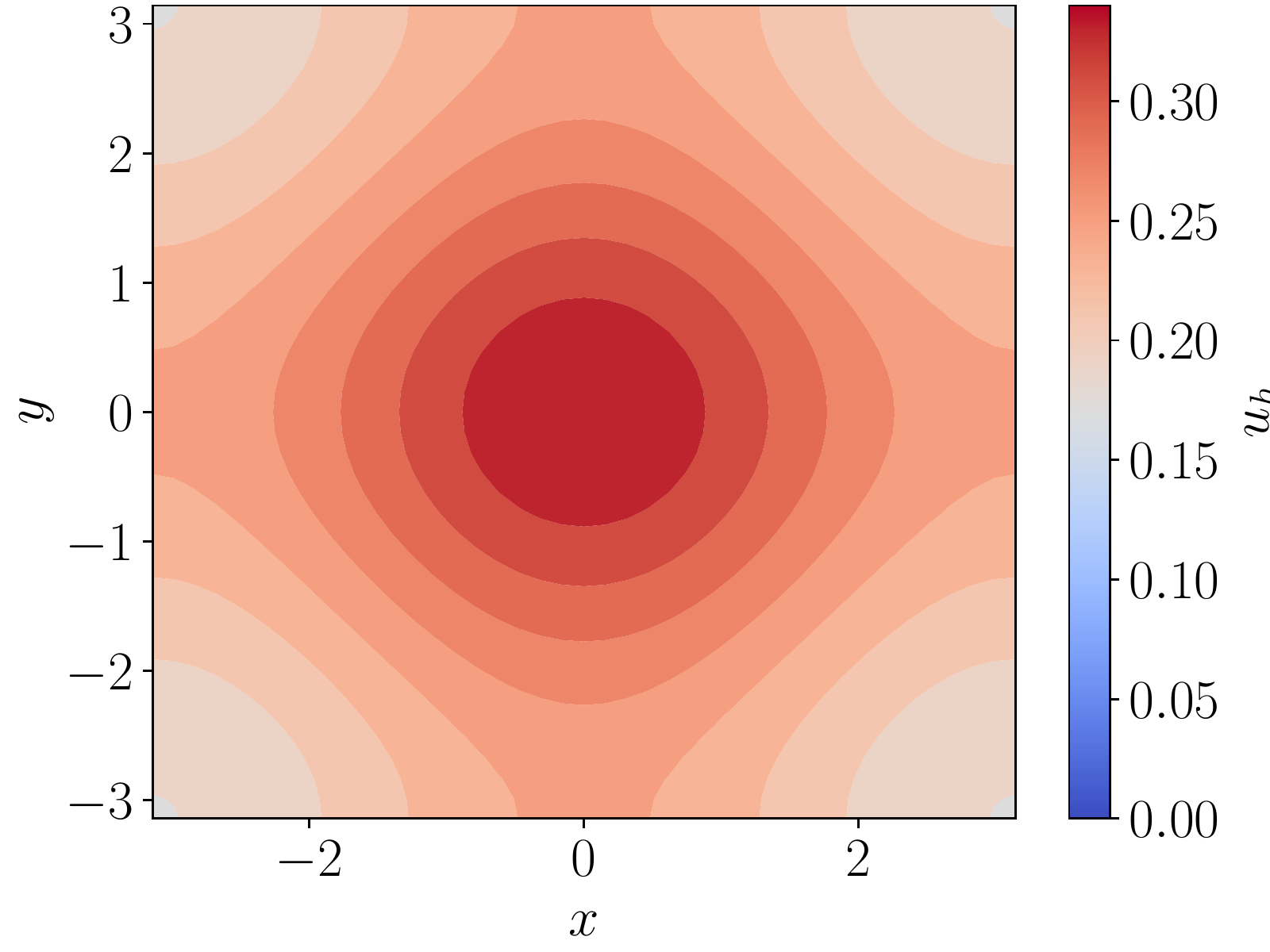}
    \end{subfigure}
            \caption*{Algorithm 1}
            \vspace{0.05\textwidth}
     \begin{subfigure}[b]{0.3\textwidth}
        \centering
        \includegraphics[width=1.0\textwidth]{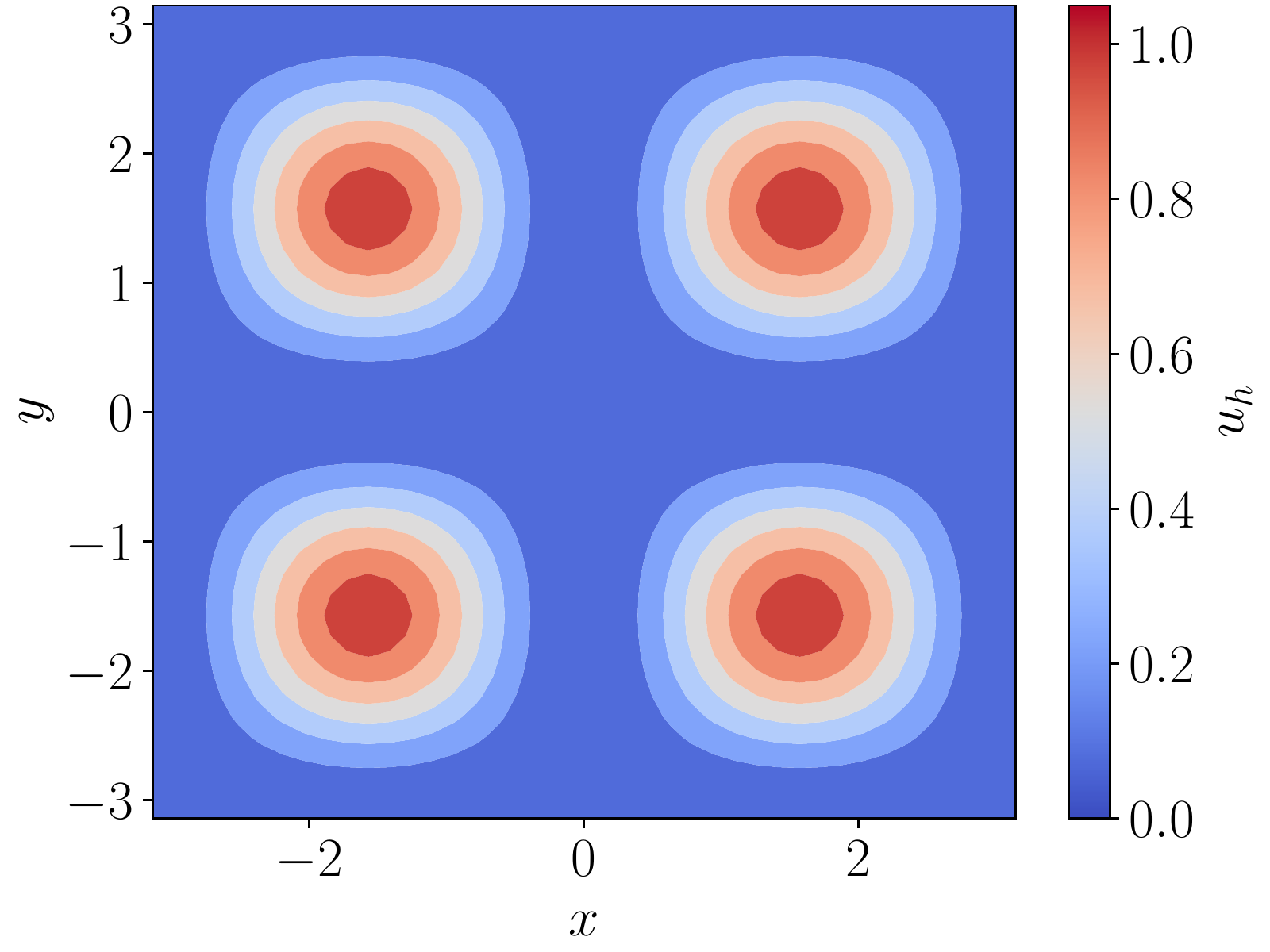}
    \end{subfigure}
    \begin{subfigure}[b]{0.3\textwidth}
        \centering
        \includegraphics[width=1.0\textwidth]{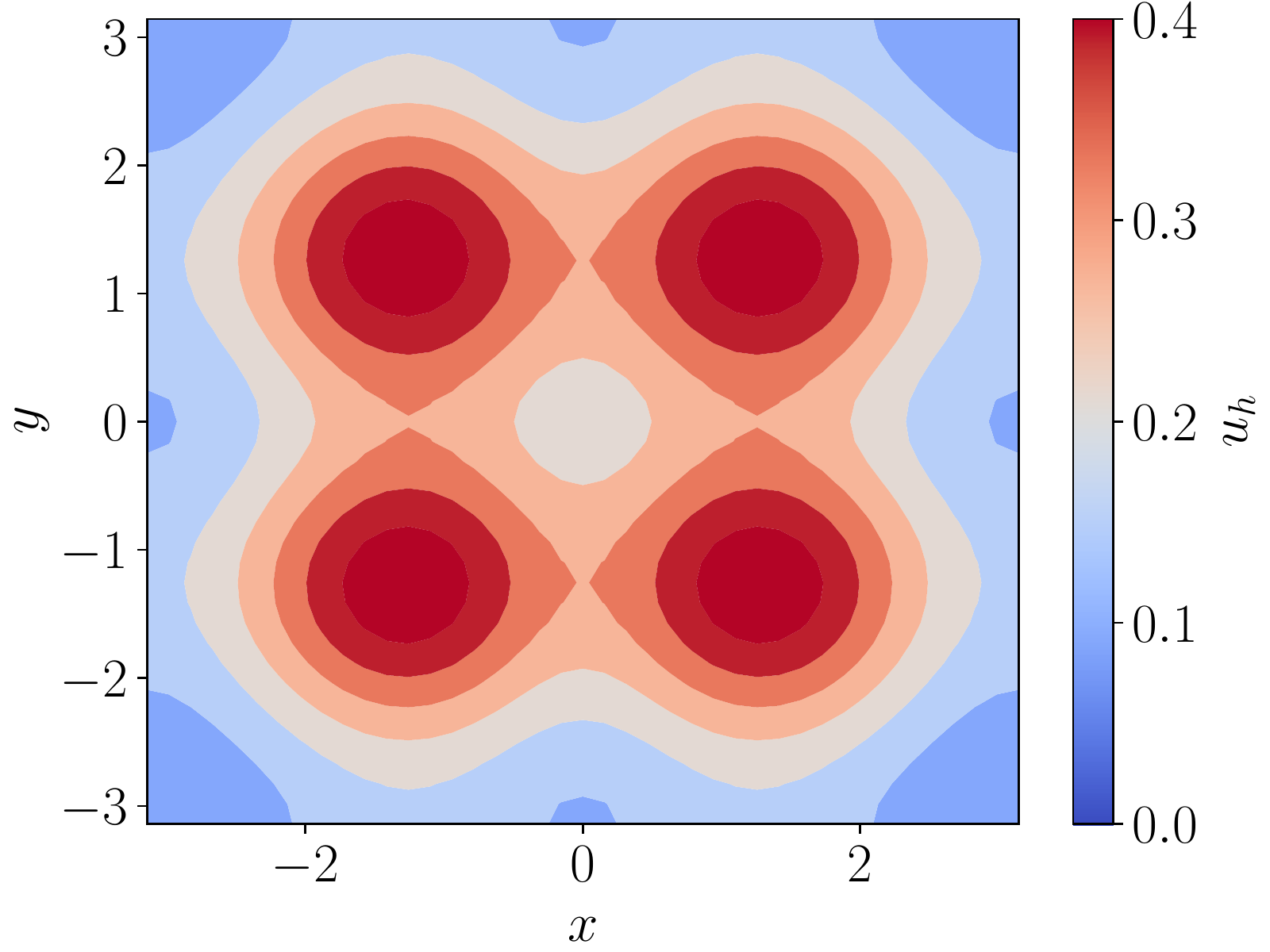}
    \end{subfigure}
    \begin{subfigure}[b]{0.3\textwidth}
        \centering
        \includegraphics[width=1.0\textwidth]{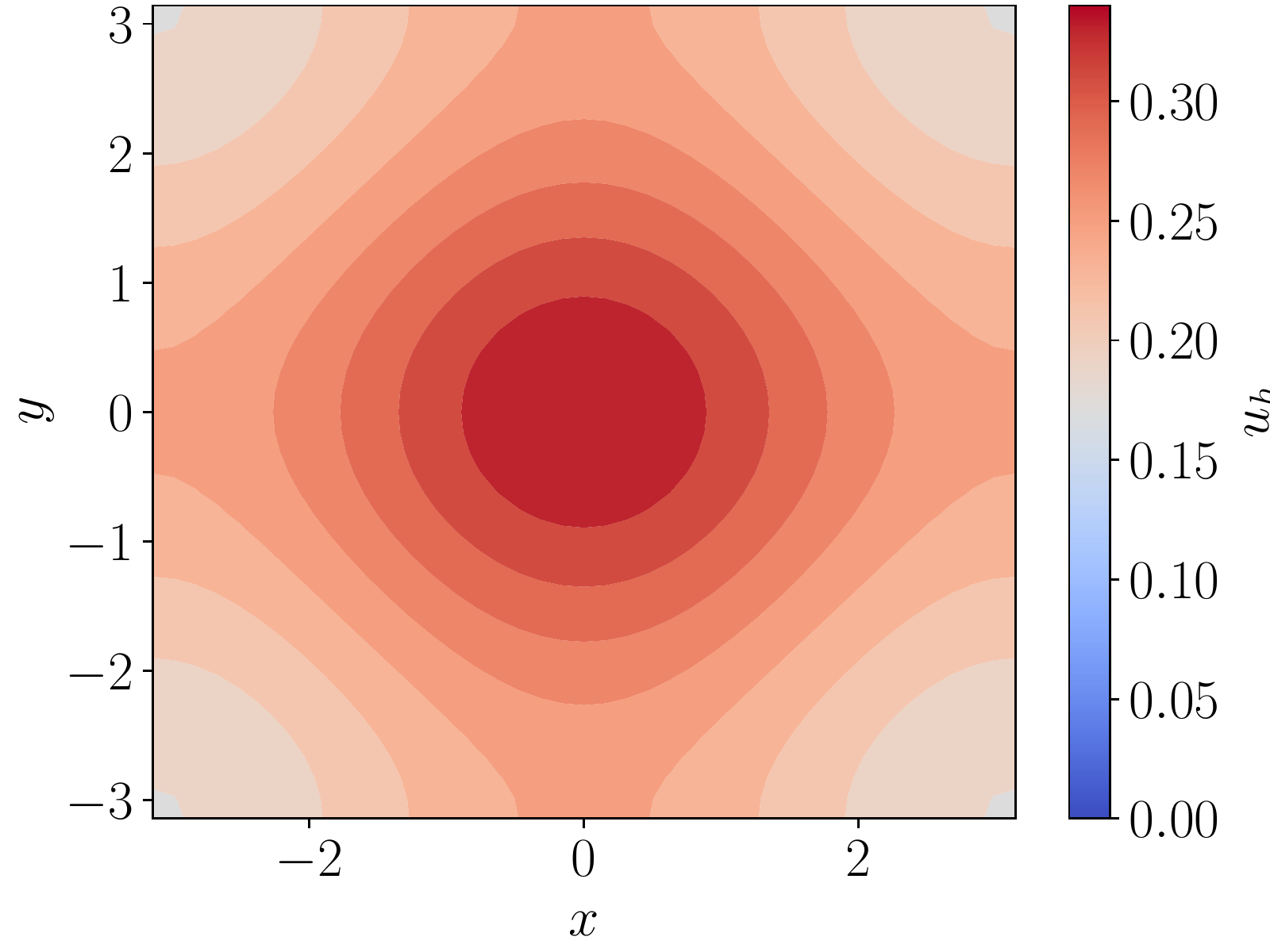}
    \end{subfigure}
            \caption*{Algorithm 2}
            \vspace{0.05\textwidth}
    \begin{subfigure}[b]{0.3\textwidth}
        \centering
        \includegraphics[width=1.0\textwidth]{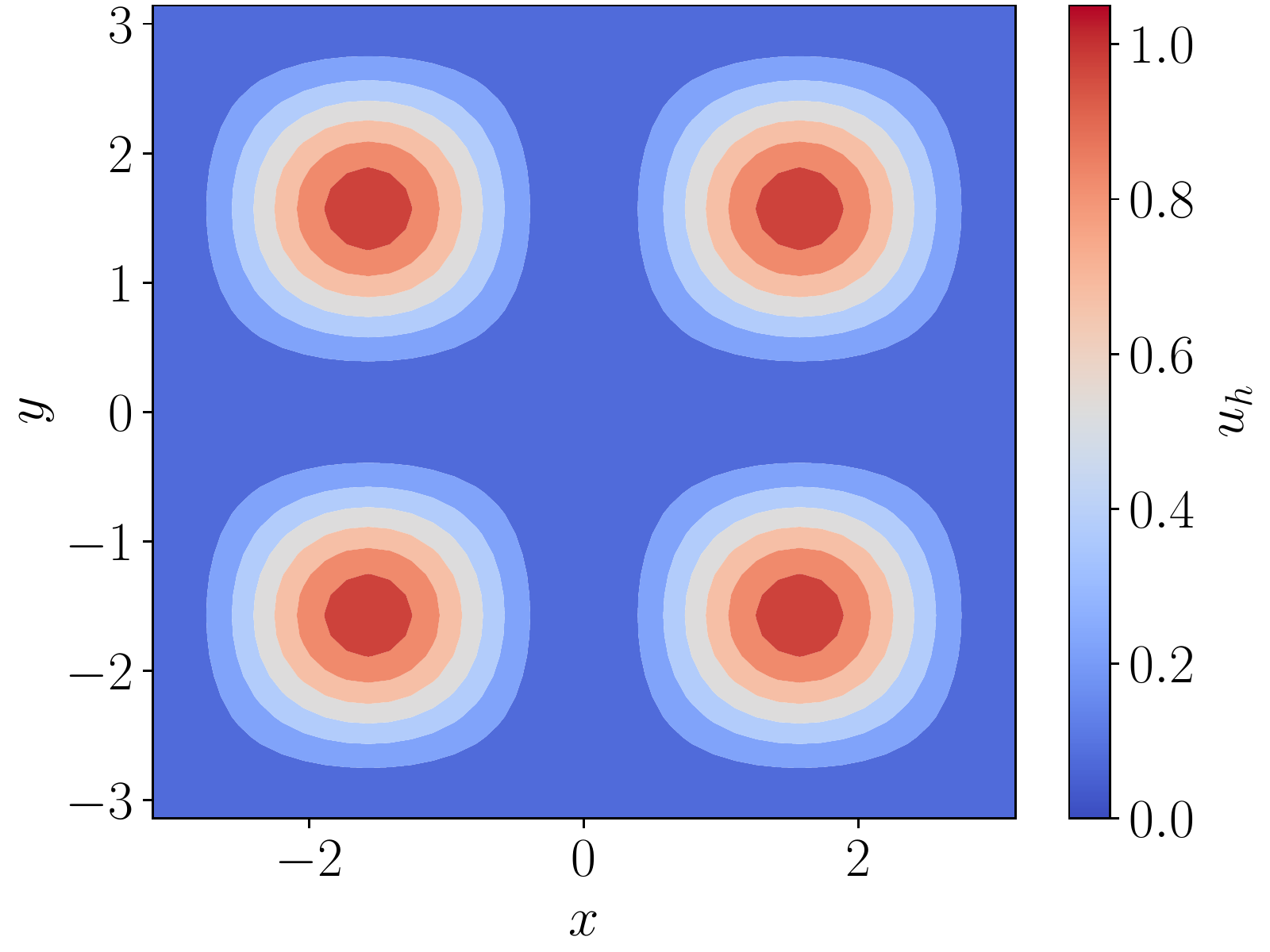}
    \end{subfigure}
    \begin{subfigure}[b]{0.3\textwidth}
        \centering
        \includegraphics[width=1.0\textwidth]{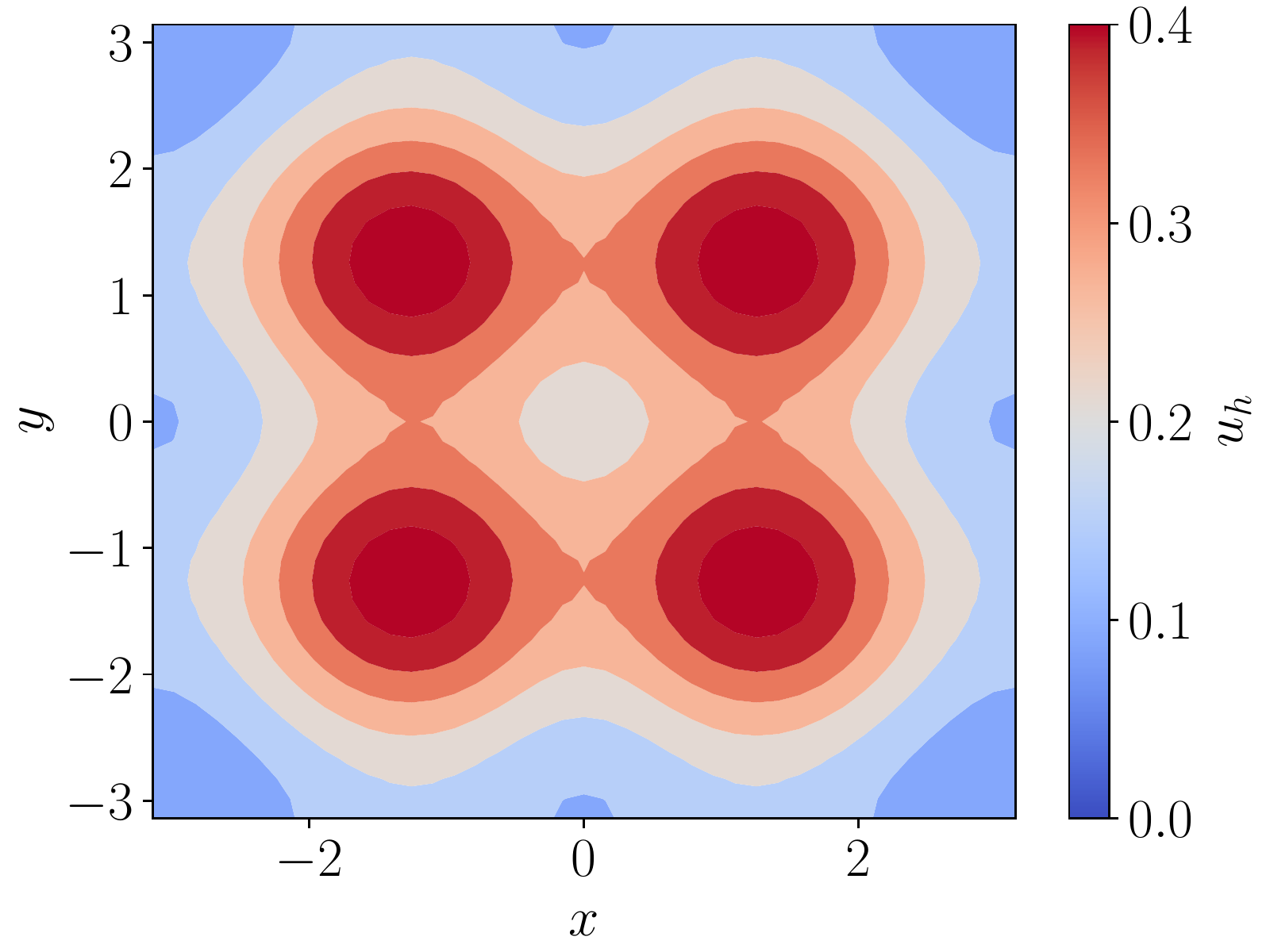}
    \end{subfigure}
    \begin{subfigure}[b]{0.3\textwidth}
        \centering
        \includegraphics[width=1.0\textwidth]{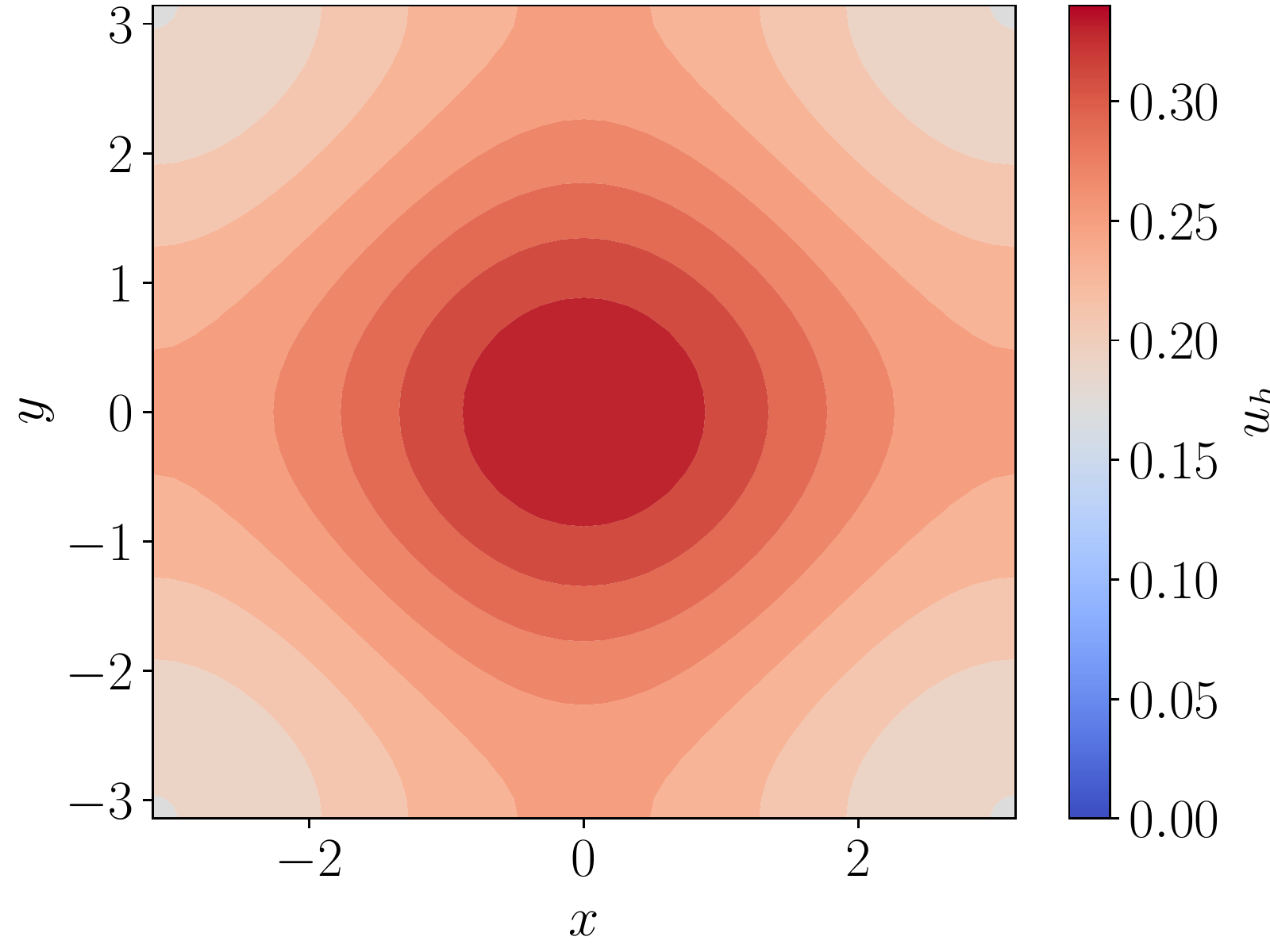}
    \end{subfigure}
            \caption*{Non-stabilized Algorithm}
        \caption{Colormaps of $u_h$ at times $t=0, 0.3$, and $2$.}\label{fig.e1-evolution-uh}
\end{figure}
\begin{figure}
    \begin{subfigure}[b]{0.3\textwidth}
        \centering
        \includegraphics[width=1.0\textwidth]{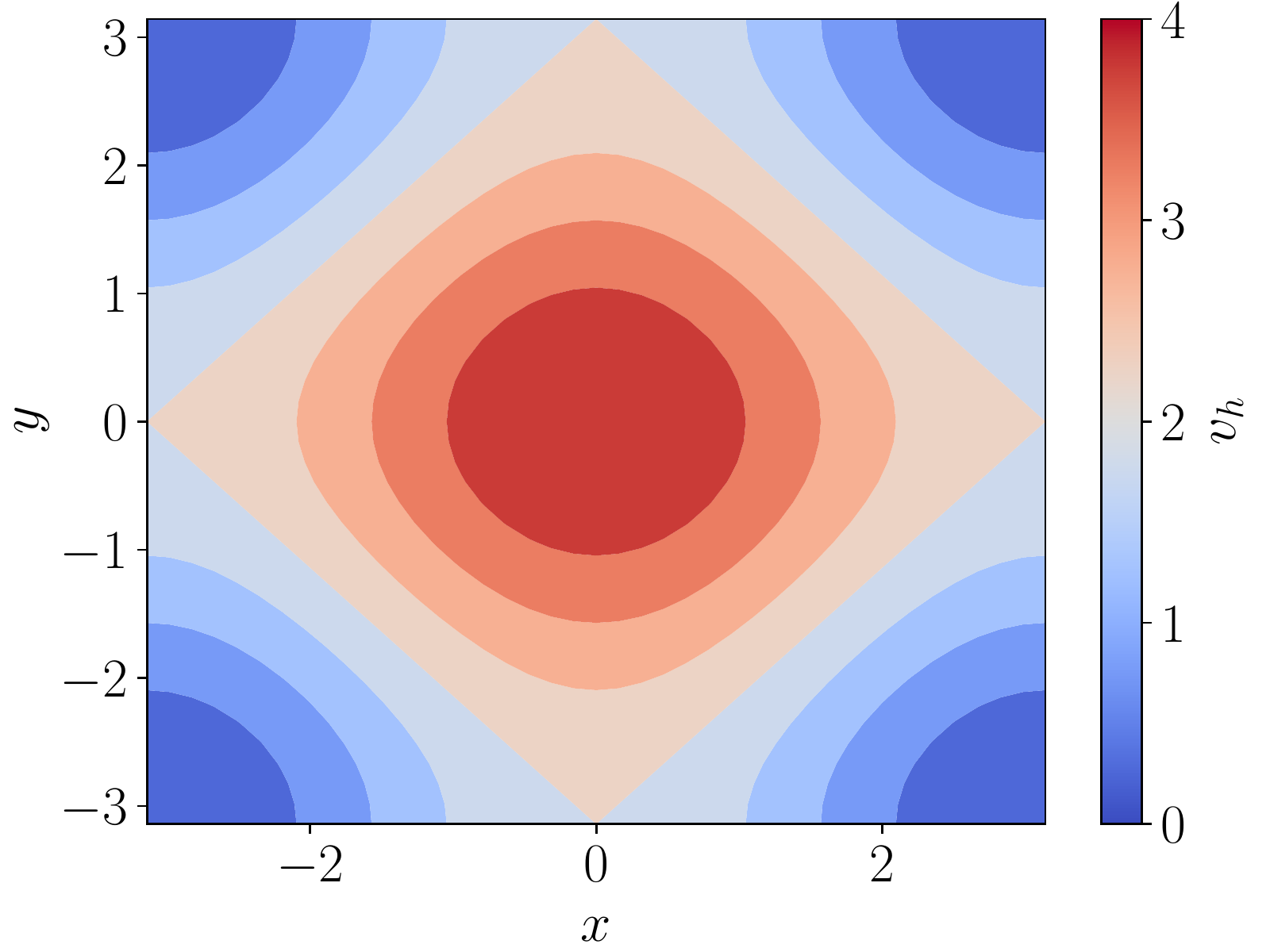}
    \end{subfigure}
    \begin{subfigure}[b]{0.3\textwidth}
        \centering
        \includegraphics[width=1.0\textwidth]{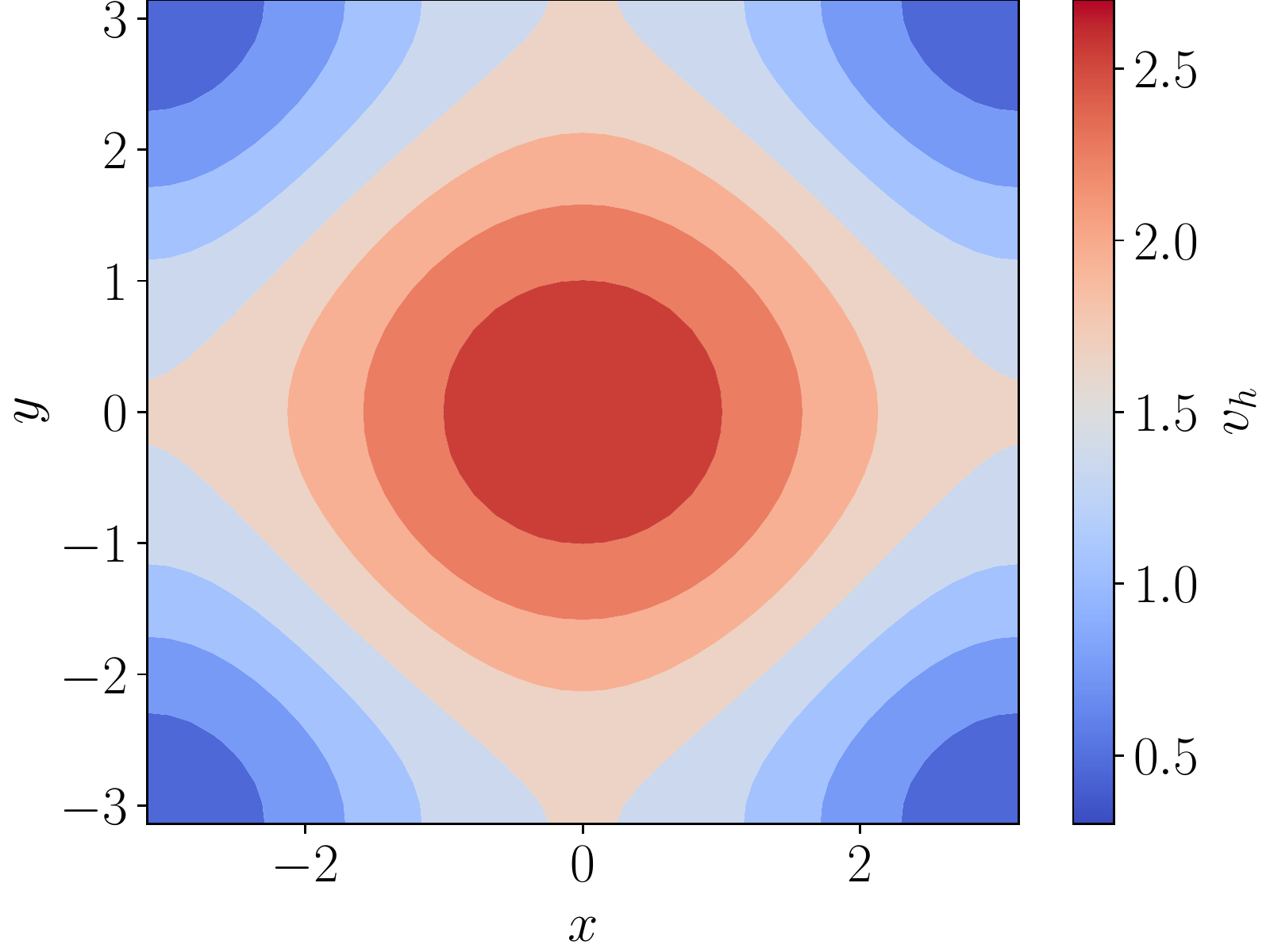}
    \end{subfigure}
    \begin{subfigure}[b]{0.3\textwidth}
        \centering
        \includegraphics[width=1.0\textwidth]{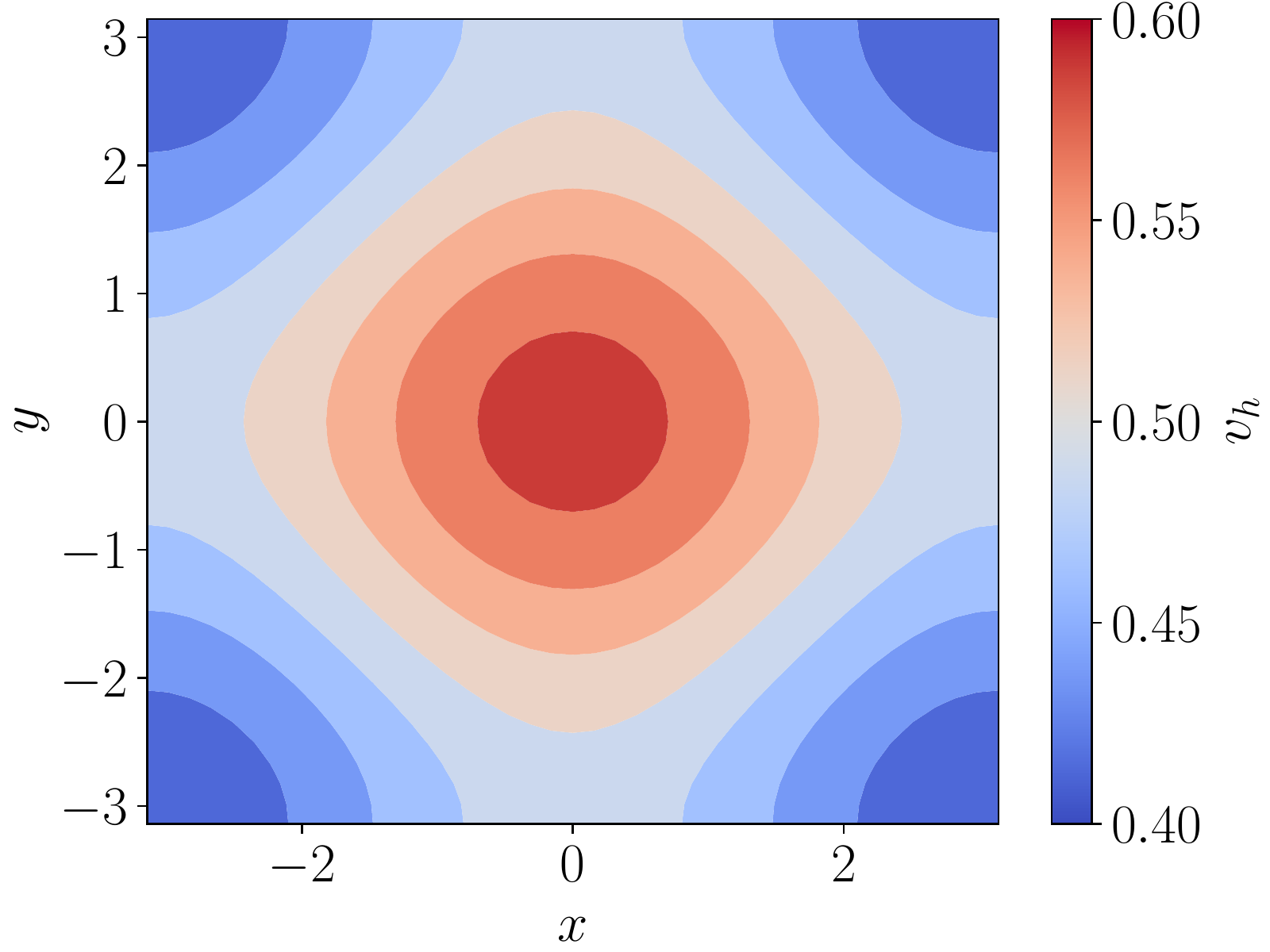}
    \end{subfigure}
        \caption*{Algorithm $1$}
        \vspace{0.05\textwidth}
    \begin{subfigure}[b]{0.3\textwidth}
        \centering
        \includegraphics[width=1.0\textwidth]{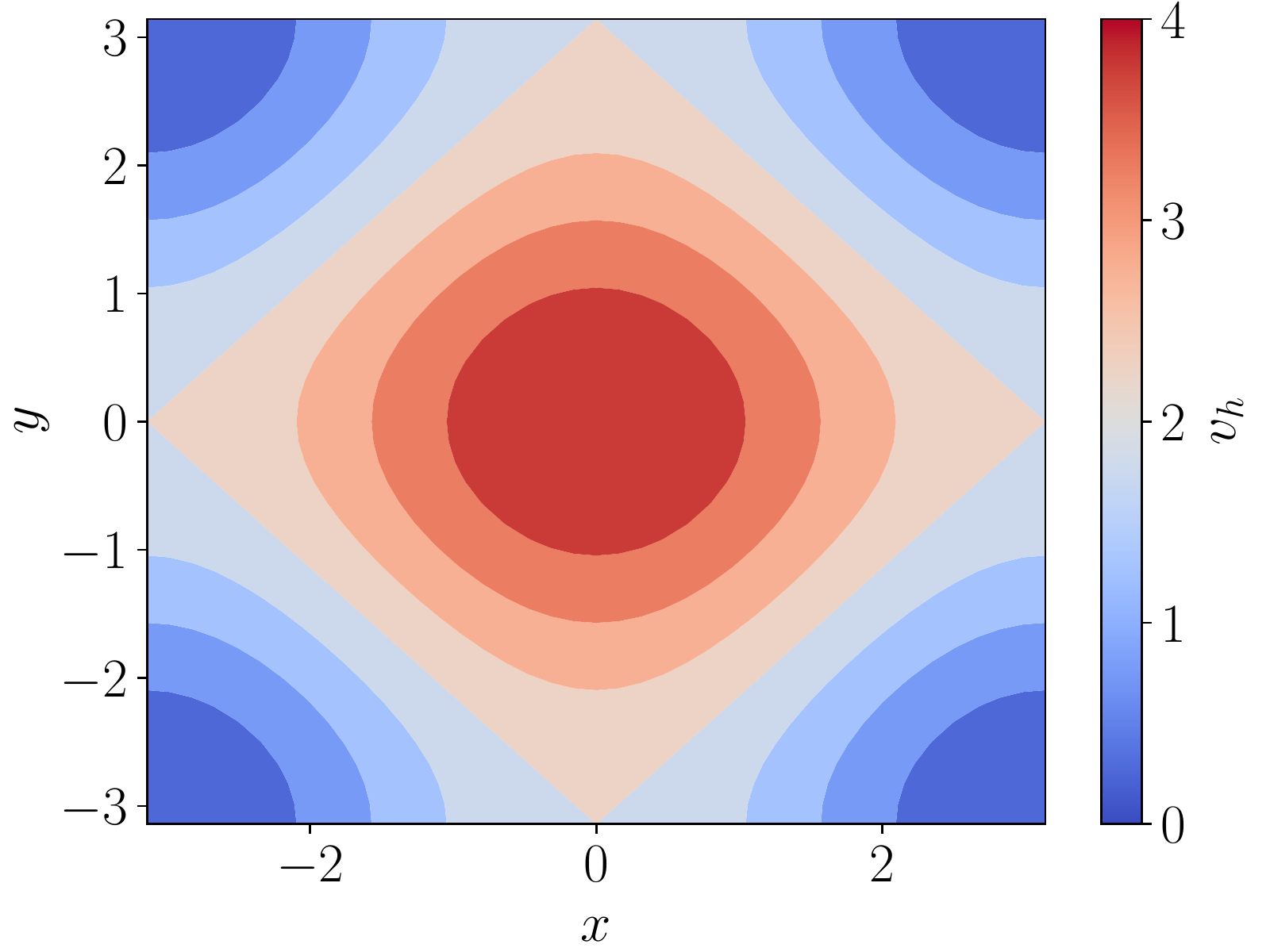}
    \end{subfigure}
    \begin{subfigure}[b]{0.3\textwidth}
        \centering
        \includegraphics[width=1.0\textwidth]{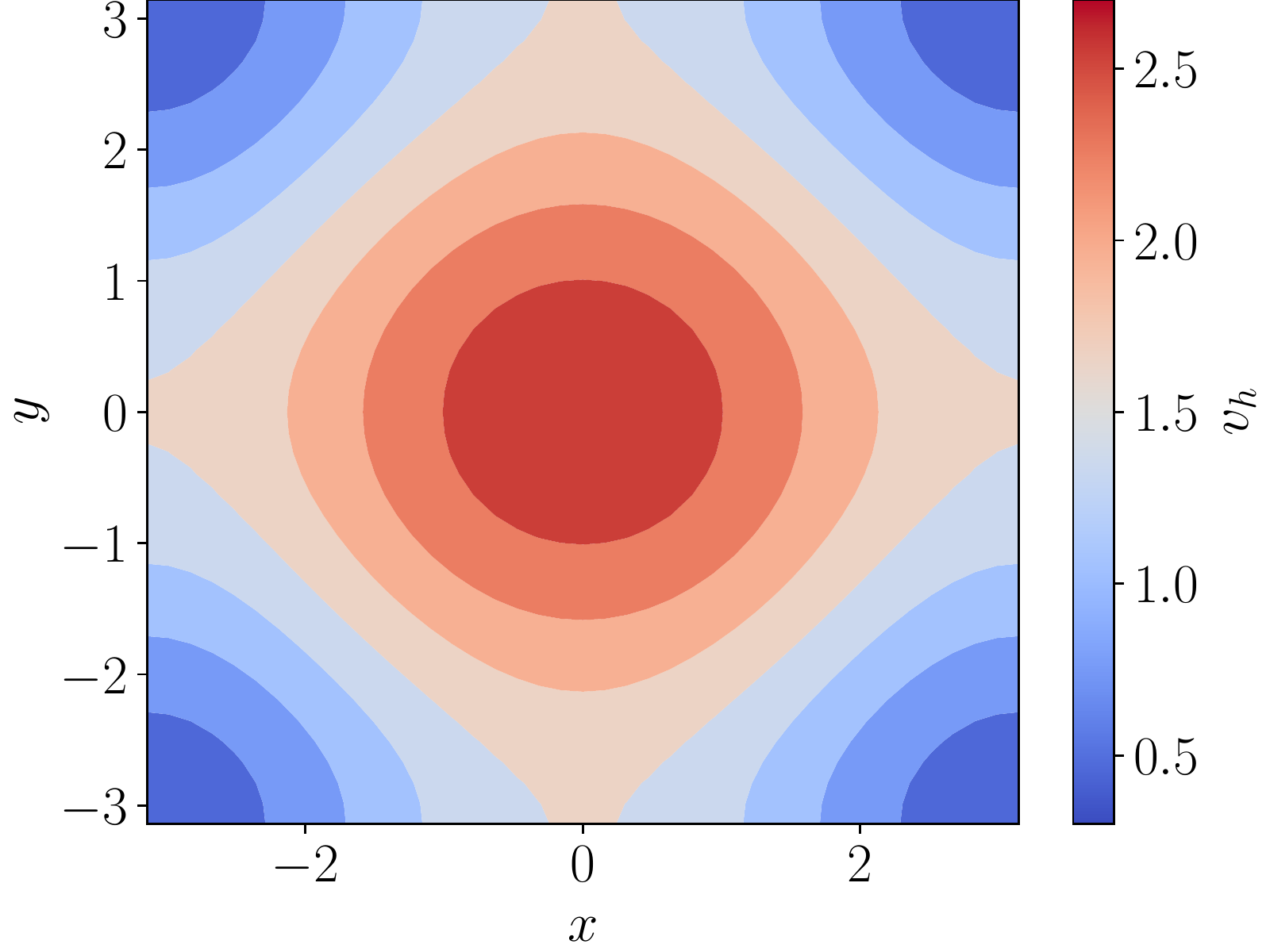}
    \end{subfigure}
    \begin{subfigure}[b]{0.3\textwidth}
        \centering
        \includegraphics[width=1.0\textwidth]{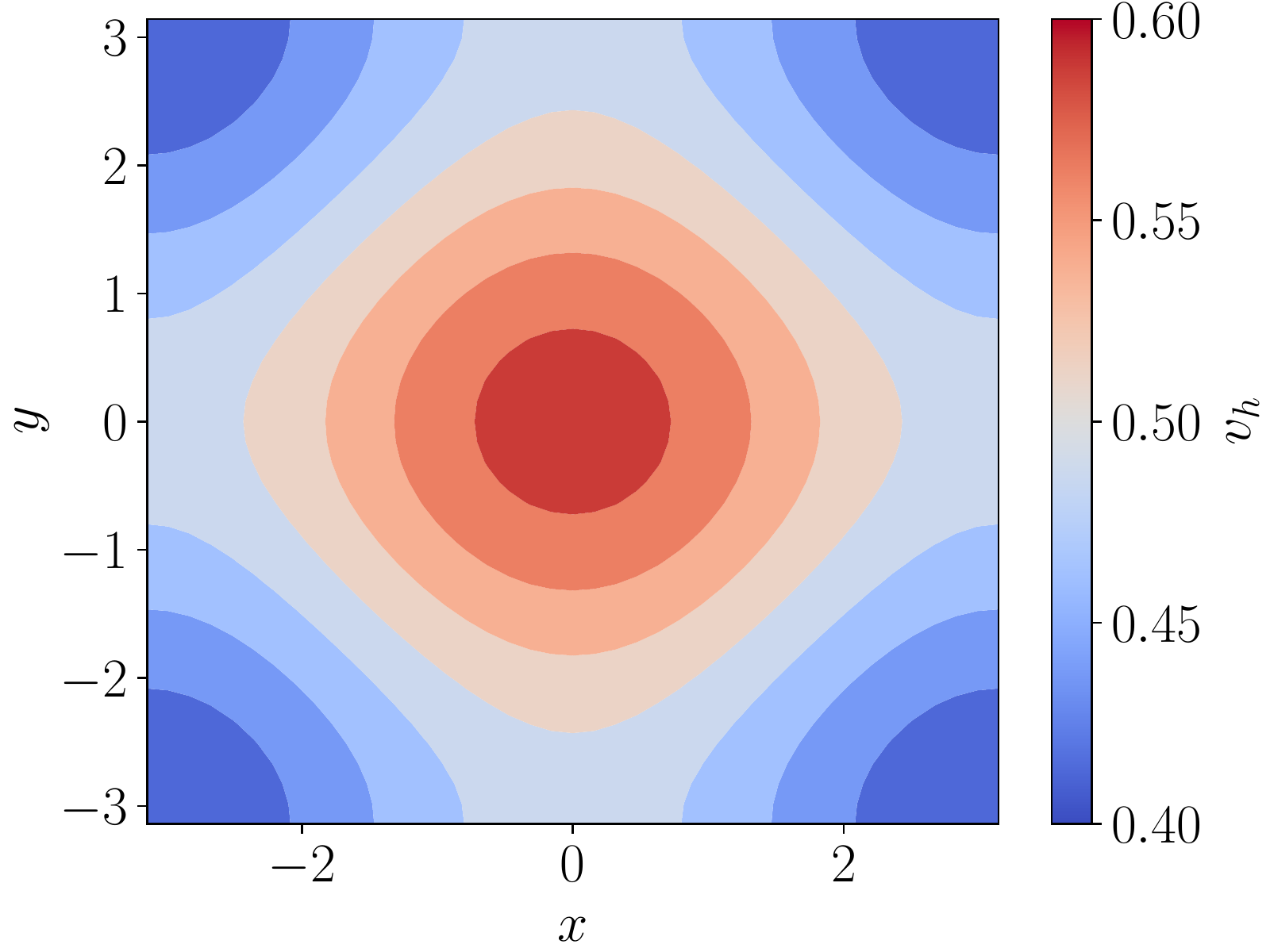}
    \end{subfigure}
    \caption*{Algorithm $2$}
    \vspace{0.05\textwidth}
    \begin{subfigure}[b]{0.3\textwidth}
        \centering
        \includegraphics[width=1.0\textwidth]{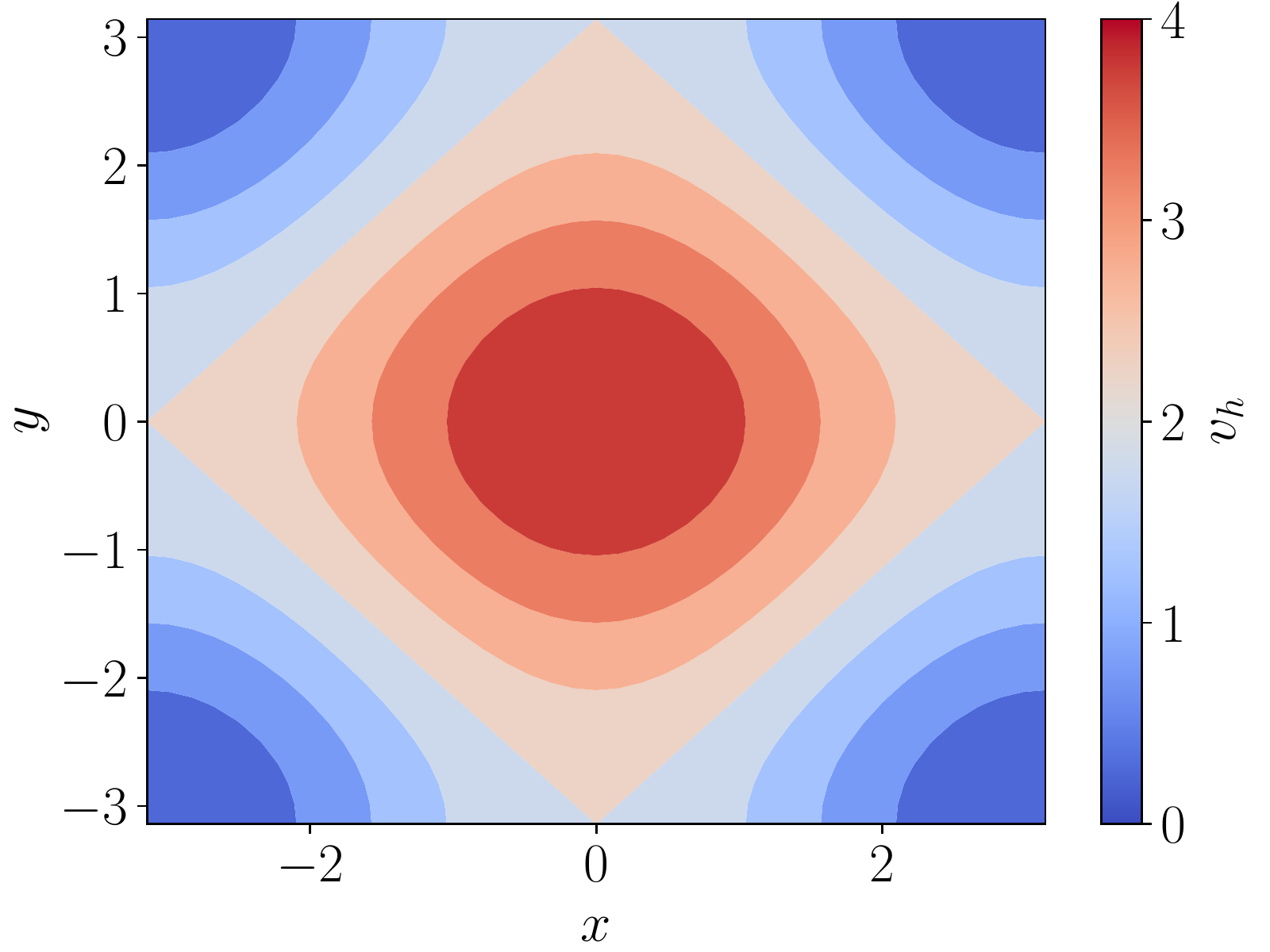}
    \end{subfigure}
    \begin{subfigure}[b]{0.3\textwidth}
        \centering
        \includegraphics[width=1.0\textwidth]{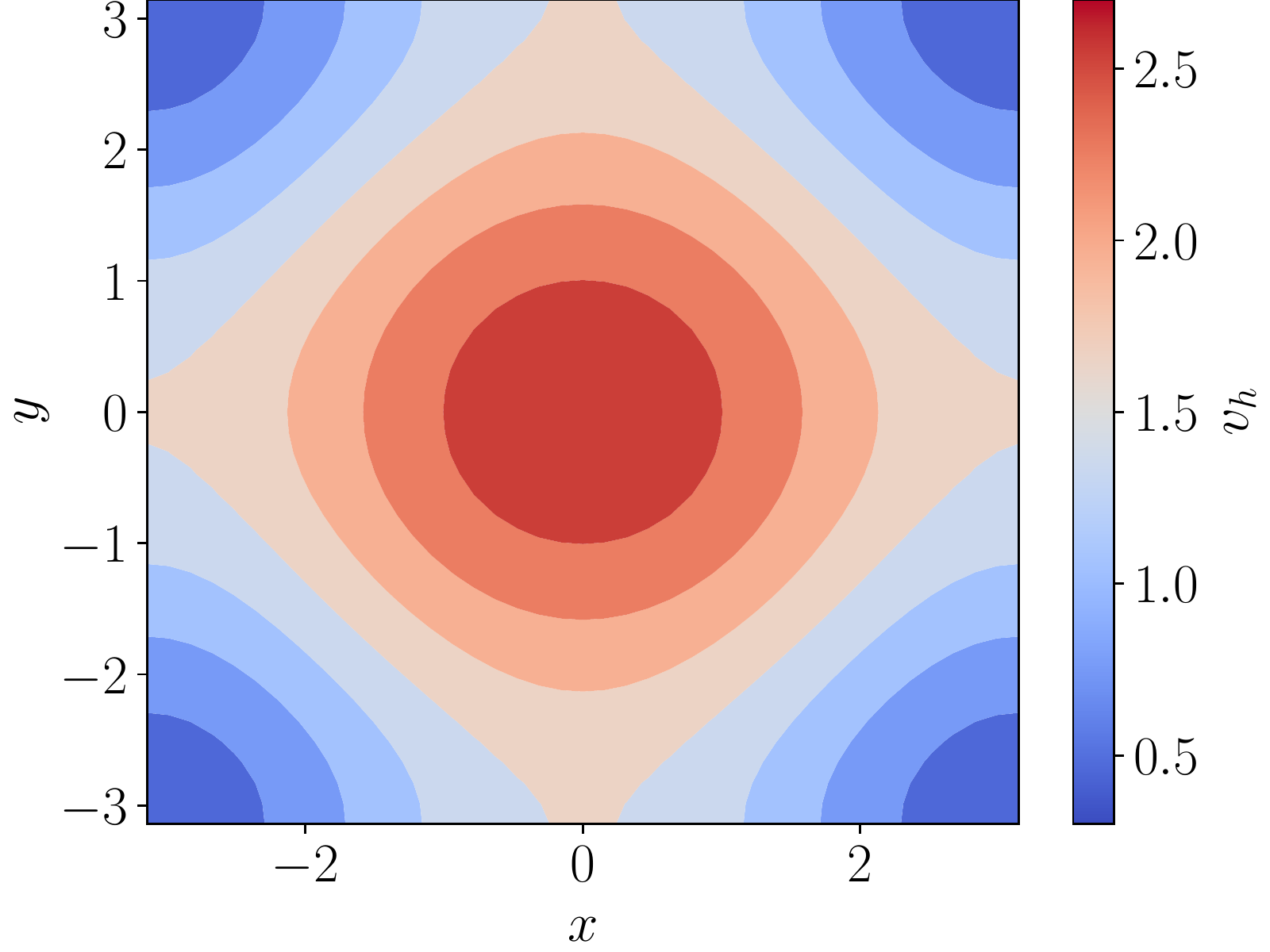}
    \end{subfigure}
    \begin{subfigure}[b]{0.3\textwidth}
        \centering
        \includegraphics[width=1.0\textwidth]{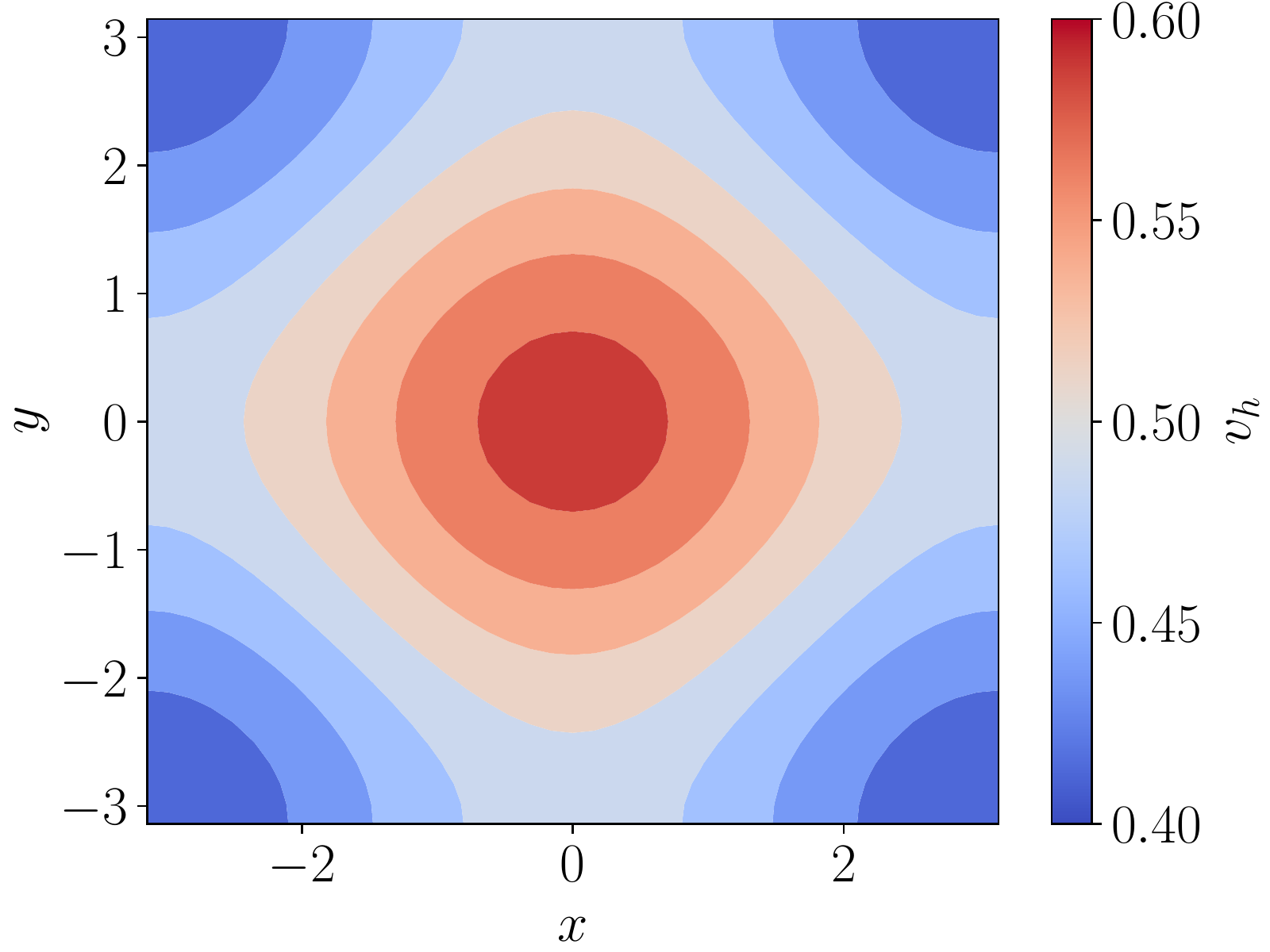}
    \end{subfigure}
        \caption*{Non-stabilized Algorithm}
    \caption{Colormaps of $v_h$ at times $t=0, 0.3$, and $2$.}\label{fig.e1-evolution-vh}
\end{figure}

To somehow ascertain the numerical diffusion introduced by the stabilizing terms, we carry out the same numerical simulation via the algorithm proposed in \cite{GS_RG_2021}, which comprises a standard finite method and a semi-implicit/implicit Euler time-stepping; that is, Algorithms $1$ without the stabilizing terms. No significant differences are highlighted in Figures \ref{fig.e1-evolution-uh} and \ref{fig.e1-evolution-vh} (bottom), implying that the shock-capturing works only in the presence of steep gradients. Figure \ref{fig.e1-profiles} (bottom) shows the profiles at times $t=0$ and $2$ for the non-stabilized discrete solution.   
\begin{figure}
    \begin{subfigure}[b]{0.3\textwidth}
        \includegraphics[width=1.0\textwidth]{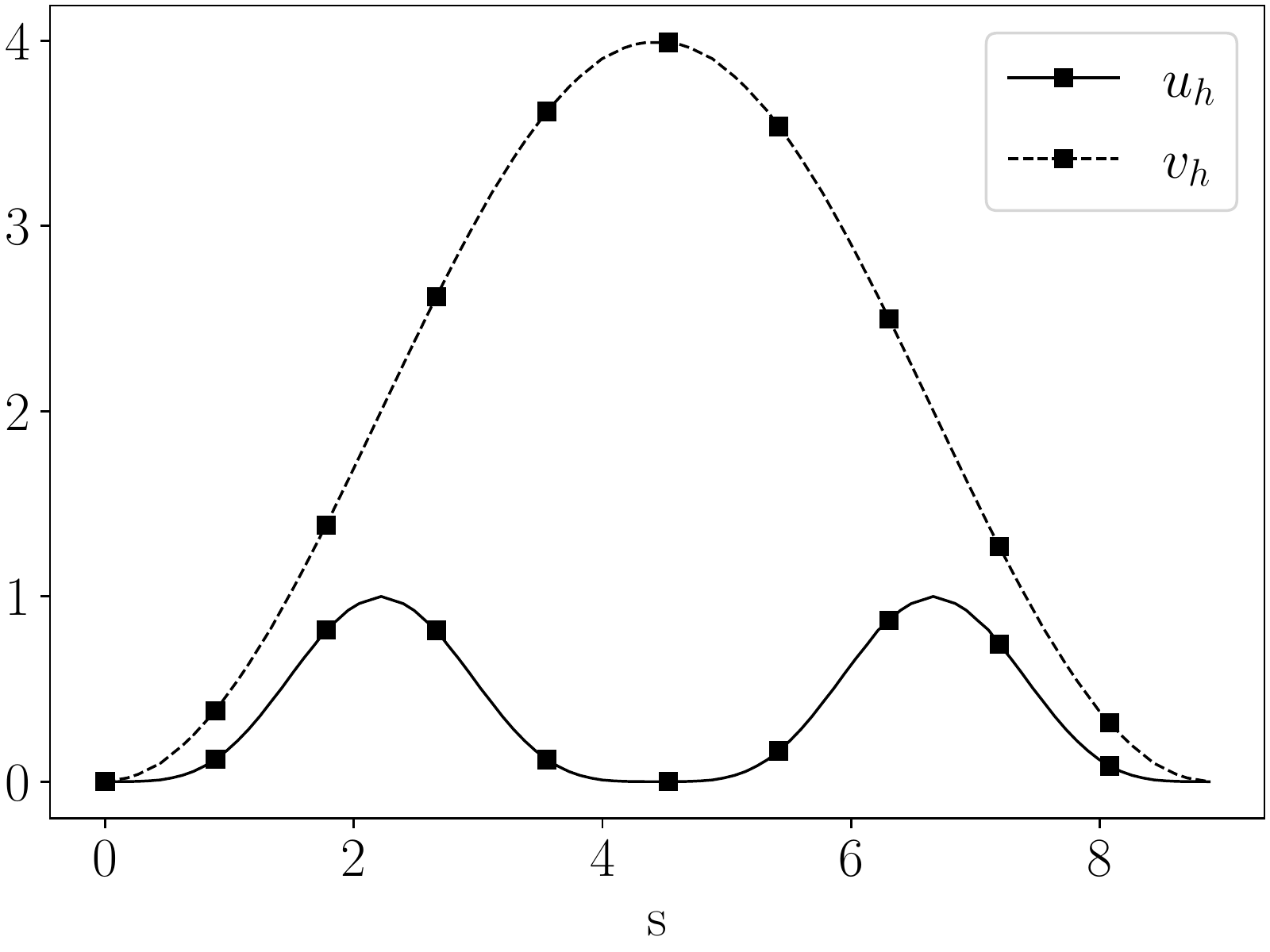}
    \end{subfigure}
    \begin{subfigure}[b]{0.3\textwidth}
        \includegraphics[width=1.0\textwidth]{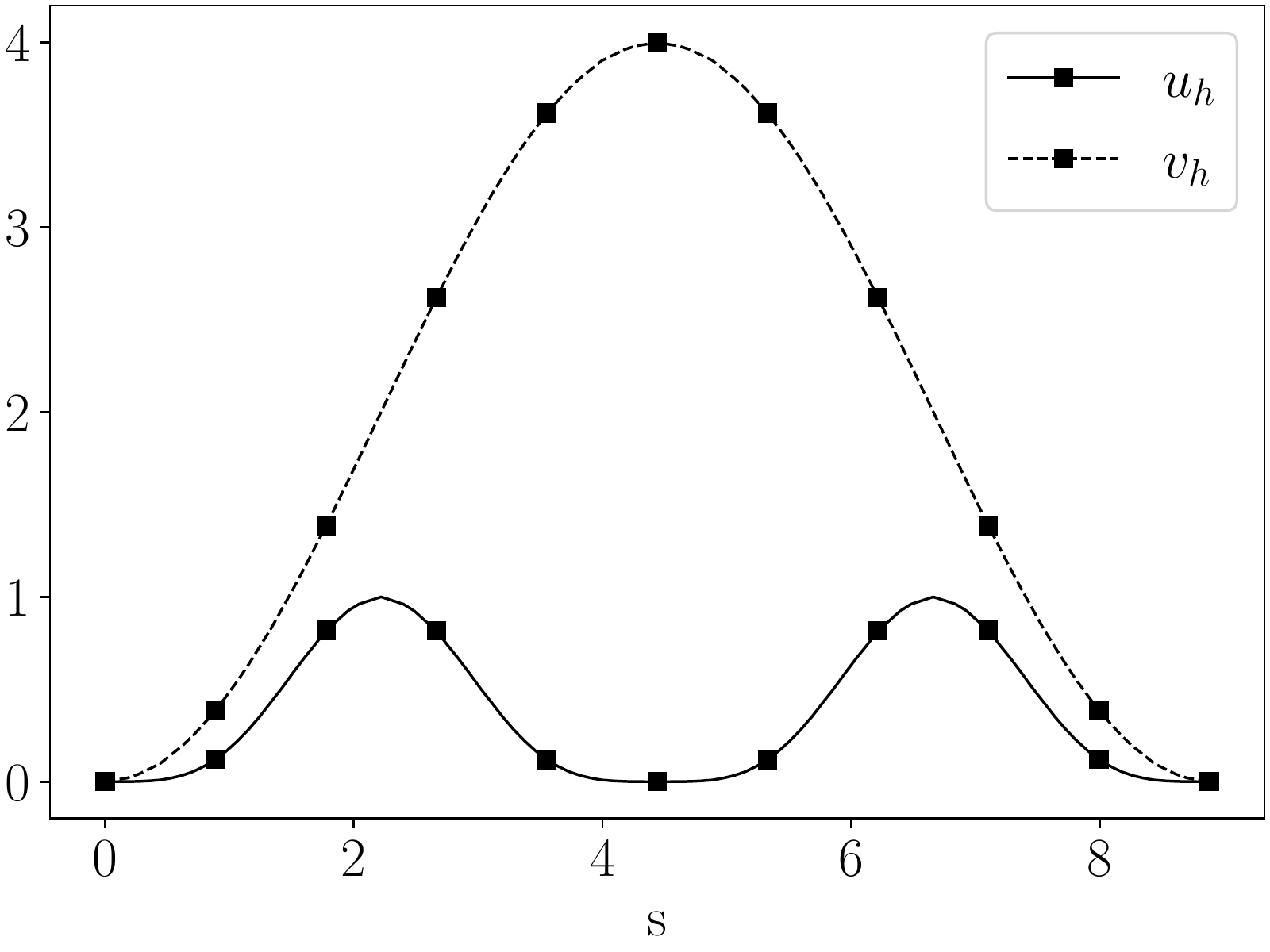}
    \end{subfigure}
    \begin{subfigure}[b]{0.3\textwidth}
        \includegraphics[width=1.0\textwidth]{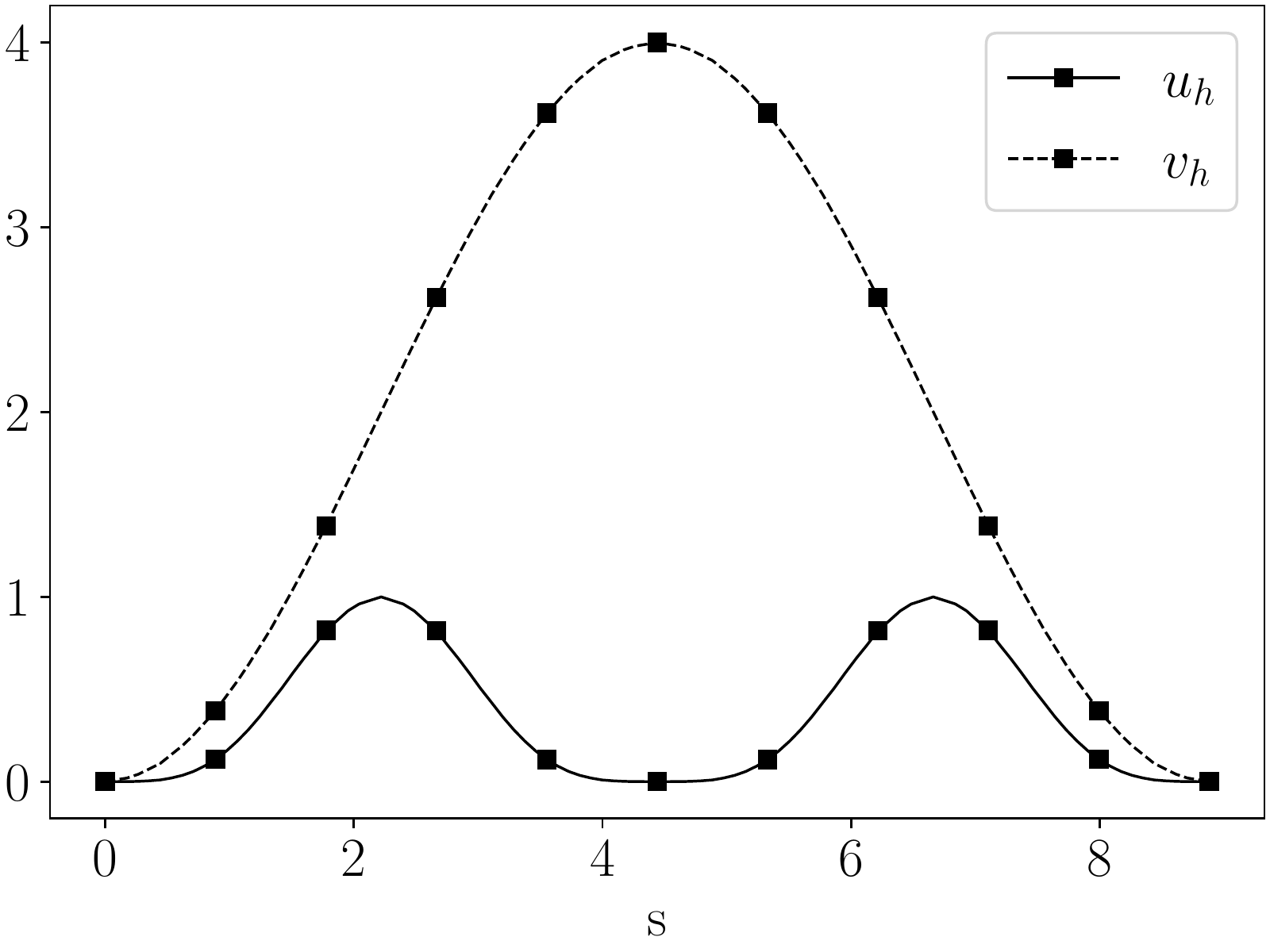}
    \end{subfigure}
    \begin{subfigure}[b]{0.3\textwidth}
        \includegraphics[width=1.0\textwidth]{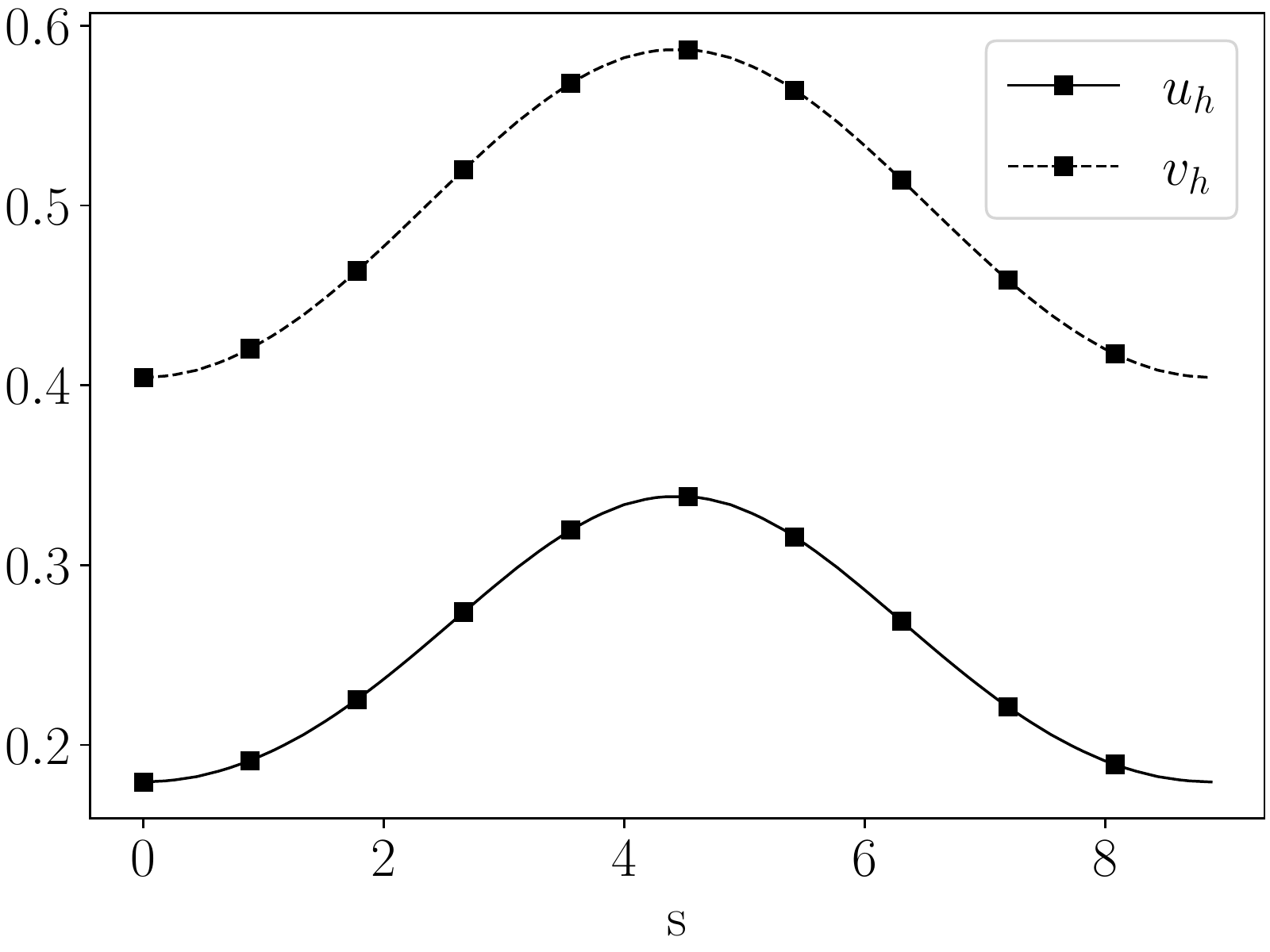}
	    \caption*{Algorithm $1$}
    \end{subfigure}
    \begin{subfigure}[b]{0.3\textwidth}
        \includegraphics[width=1.0\textwidth]{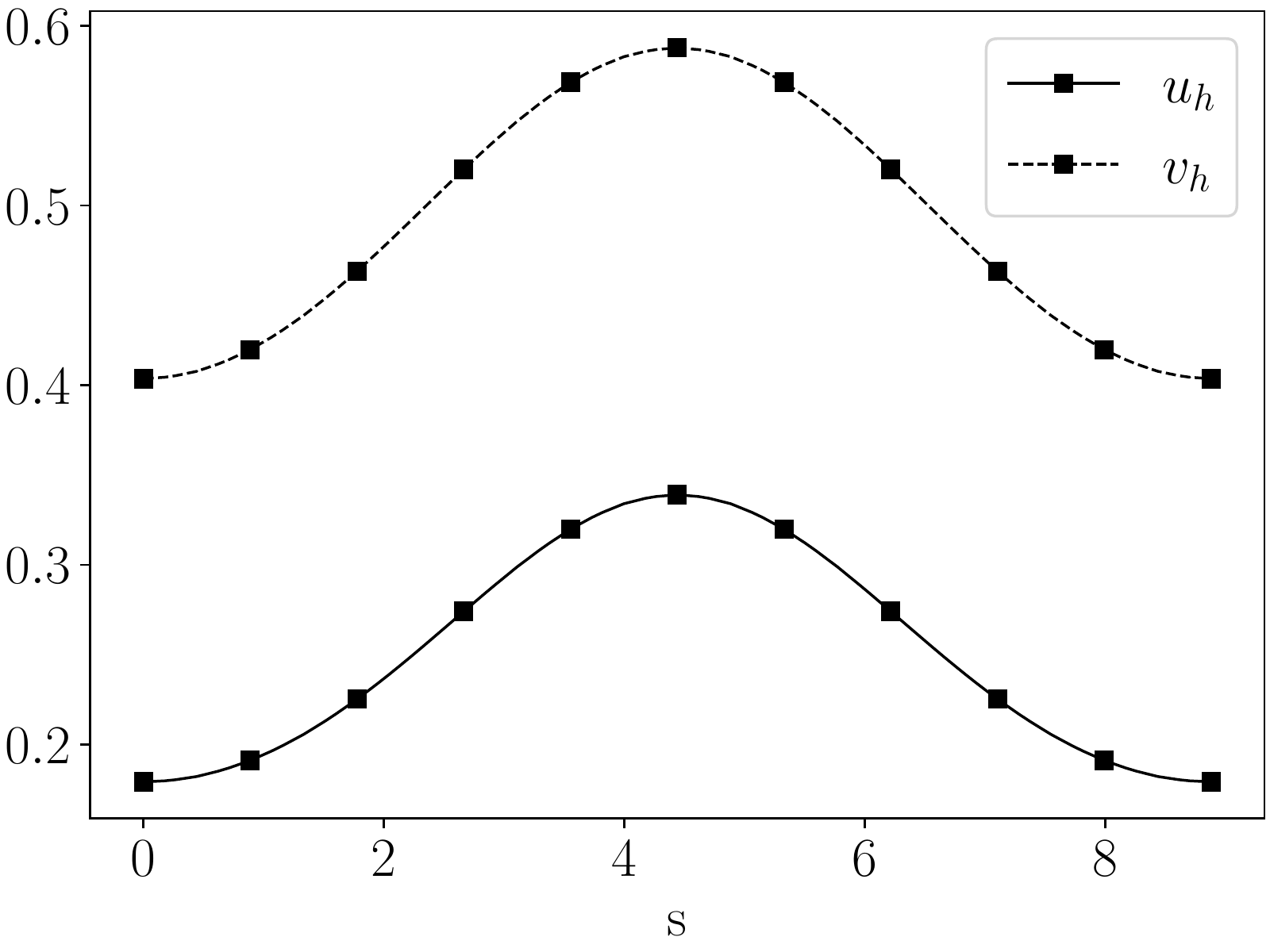}
	    \caption*{Algorithm $2$}
    \end{subfigure}
    \begin{subfigure}[b]{0.3\textwidth}
        \includegraphics[width=1.0\textwidth]{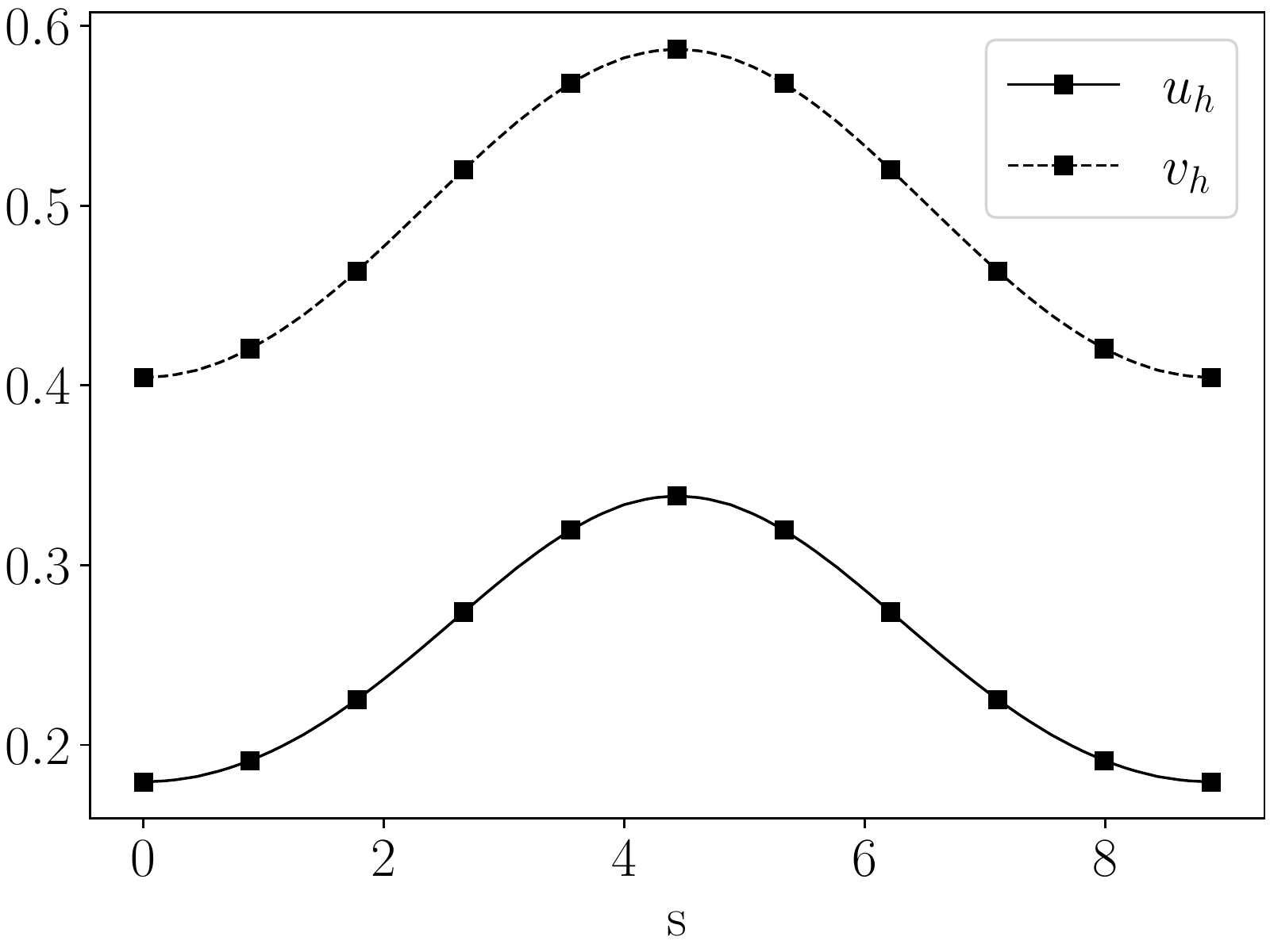}
	    \caption*{Non-stabilized algorithm}
    \end{subfigure}
    \caption{ Profiles of $u_h$ and $v_h$ across the diagonal at times $t=0$ and $2$.}\label{fig.e1-profiles}
\end{figure}
\begin{figure}
        \centering
        \includegraphics[width=0.5\textwidth]{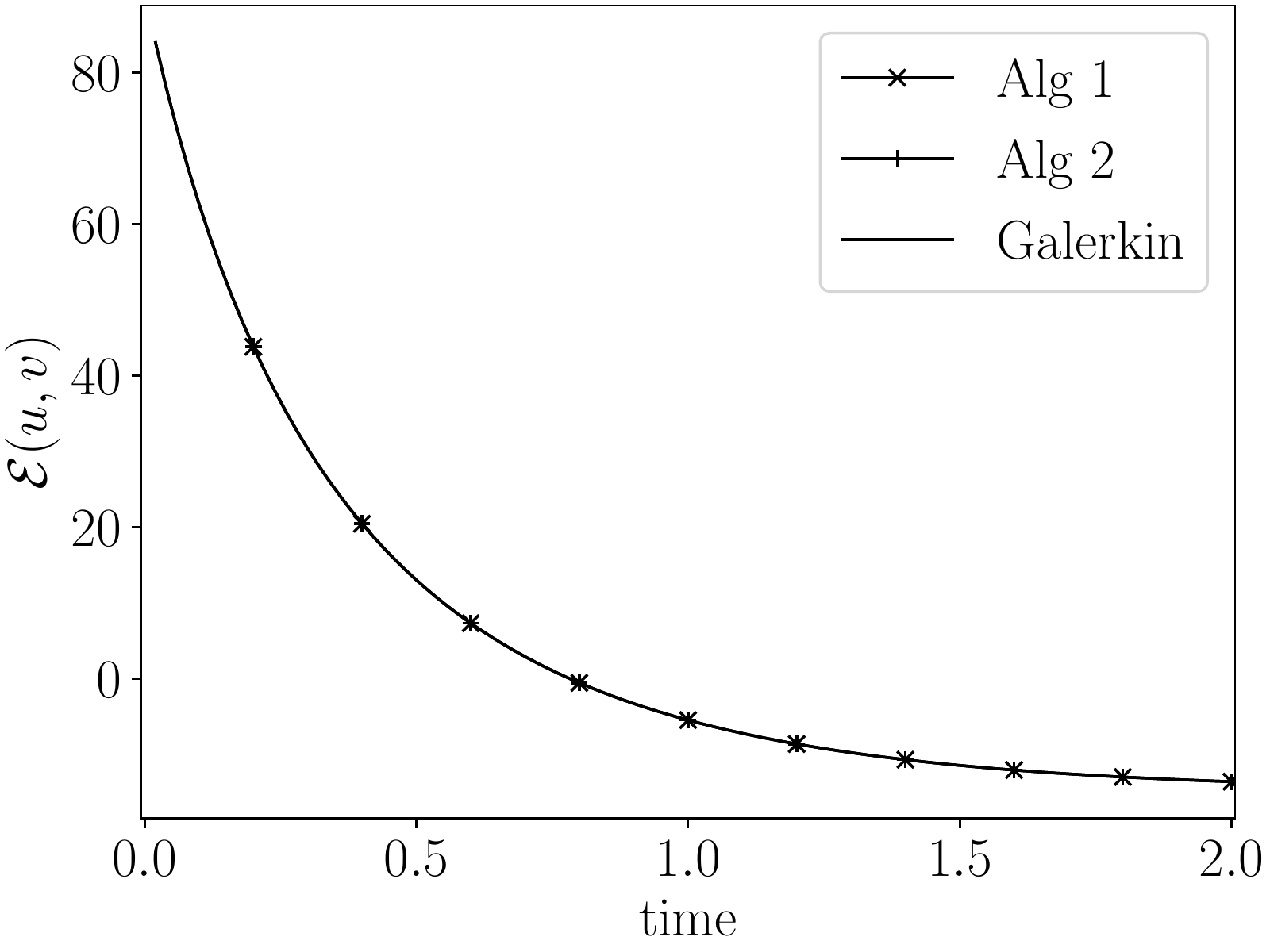}
	    \caption{Evolution of $\mathcal{E}_h(u_h,v_h)$}\label{fig.e1-energy}
\end{figure}
\subsection{Blowup phenomena} In the following two numerical tests, we assess our two algorithms with initial data which generate solutions to the Keller-Segel problem \eqref{KS}-\eqref{BC} that blow up in finite time. In this fashion, the robustness and reliability of discrete solutions computed by Algorithms $1$ and $2$ are examined. The need of using stabilizing terms is a matter of the utmost importance. This was made manifest \cite{GS_RG_2021} when a non-stabilized finite element method failed to prevent discrete solutions from being negative.  In this case, the usage of a positivity-preserving scheme is mandatory. Otherwise, the high gradients yield an oscillatory behavior of discrete solutions taking negative density values. 
\subsubsection{Blowup in the center}
We consider the Keller-Segel equation \eqref{KS}-\eqref{BC} on $\Omega=(0,1)^2$. As opposed to the previous test, both initial conditions $u_0$ and $v_0$ are now concentrated around about the point $(0.5, 0.5)$. In particular, we use initial conditions based on \cite[Example 5.2]{Li_Shu_Yang_2017}, i.e.,
$$
 u_0 = 840\exp\left(-84\left((x-0.5)^2+(y-0.5)^2\right)\right)
 \quad\mbox{ and }\quad
 v_0 = 420\exp\left(-42\left((x-0.5)^2+(y-0.5)^2\right)\right).
$$
As $\int_\Omega u_0(\x)\,\dx\approx31.415926>4\pi$, the continuous solution to the Keller-Segel problem \eqref{KS}-\eqref{BC} should develop a blowup according to the mathematical analysis of \cite{Horstmann_Wang_2001}. 

All of the tests in this example are computed by using a bilinear finite element space for $X_h$ associated with a nonuniformly structured  $80\times 80$ quadrangulation (see Figure \ref{fig.e2-mesh}), which is refined around the point $(0.5,0.5)$, with a minimum mesh size $h_{\rm min}=0.003359$. The time step is $k=10^{-5}$.
\begin{figure}
    \centering
    \includegraphics[width=0.4\textwidth]{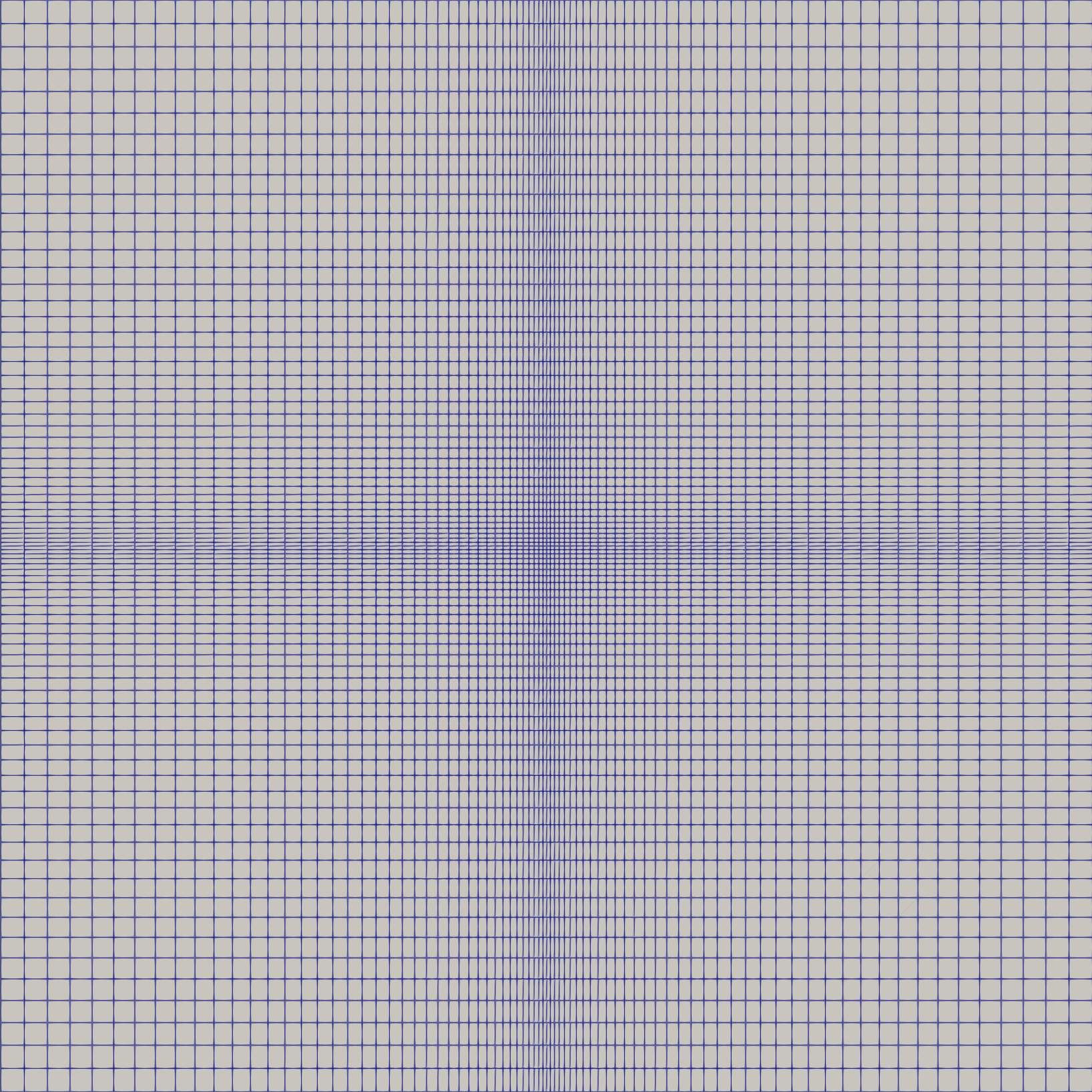}
    \caption{Mesh used to discretize $\Omega=(0,1)^2$.}
    \label{fig.e2-mesh}
\end{figure}

The singularity formation is reported in Figures \ref{fig.e2-evolution} and \ref{fig.e2-evolution-alg2} for Algorithms $1$ and $2$, respectively, at times $t=0, 1.3\cdot10^{-4}$ and $2.5\cdot 10^{-4}$. We observe that the discrete solutions computed by both algorithms evolve, as expected, to a blowup of the cell density in the center of the domain, which results in a dwindling support for the cell density as its $L^\infty(\Omega)$-norm becomes larger and larger. 

Despite the finite-time singularity, Algorithms $1$ and $2$ can preserve positivity over the entire simulation as shown in more detail in Figures \ref{fig.e2-profiles} and \ref{fig.e2-profiles-alg2}, where the cross-sections along the plane $y=0.5$ are depicted.

Figures \ref{fig.e2-maximums} and \ref{fig.e2-maximums-alg2} (left) show the time evolution of the $L^\infty(\Omega)$-norm for the cell and chemoattractant densities. The first time where the largest value, $2.5\cdot10^{6}$, of the cell density in the $L^\infty(\Omega)$-norm is taken is around $t=3\cdot10^{-4}$ for Algorithm $1$ and, slightly larger, at $t=3.5\cdot10^{-4}$  for Algorithm $2$. After that time a steady state seems to be reached for the cell density, where such a value is maintained until the simulation final time $t=4\cdot10^{-4}$. We further emphasize that the $L^\infty(\Omega)$-norm of the chemoattractant density moves roughly from $420$ to $120$ and that a substantial decay is produced once the cell density reaches its largest value for the first time. In Figures \ref{fig.e2-maximums} and \ref{fig.e2-maximums-alg2} (right), the mass progression suggests that the $L^1(\Omega)$-norm is preserved over time for both densities.

As we know from \eqref{E-blowup}, the energy functional $\mathcal{E}_h(u_h, v_h)$ should go to $-\infty$ as time goes to the blouwp one, but rather it stagnates at $-513$ for $t>0.1$. The reason behind this might be that the cell density is mostly supported on the macroelement corresponding to a single node. More precisely, since the $L^1(\Omega)$-norm is conserved, the $L^\infty(\Omega)$-norm is bounded from above when the cell density is only supported on one macroelement. In such a case, for the used uniform 2D Cartesian macroelement occurring the blowup, $\|u^n_h\|_{L^1(\Omega)} = \int_\Omega u^n_{\boldsymbol{a}} \varphi_{\bf a} = u^n_{\boldsymbol{a}} h_{\rm min}^2$. Therefore the $L^\infty(\Omega)$-norm is bounded as $\max_{n\in\{0, \cdots, N\}} \|u^n_h\|_{L^\infty(\Omega)} \le h_{\rm min}^{-2}\|u_0\|_{L^1(\Omega)}$. In particular, $\max_{n\in\{0, \cdots, N\}} \|u^n_h\|_{L^\infty(\Omega)}\lesssim 2.78439\cdot10^6$, which is the largest value that the cell density can take. Perhaps this \textit{locking} effect occurring in the cell growth causes the chemoattractant does not grow locally but diffuses in disagreement with \cite[Prop. 2]{Horstmann_2001}.
 \begin{figure}
    \begin{subfigure}[b]{0.3\textwidth}
        \centering
        \includegraphics[width=1.0\textwidth]{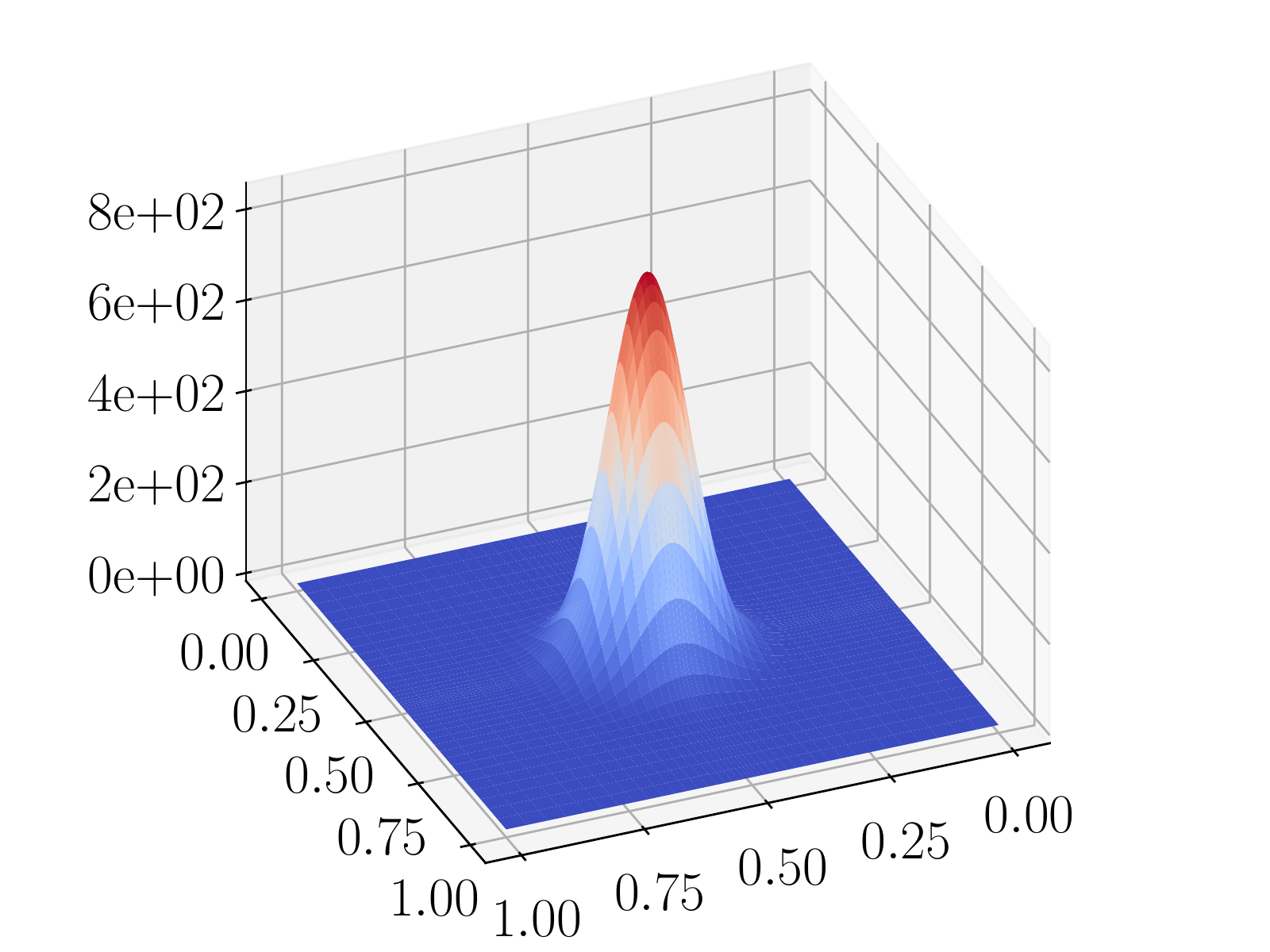}
    \end{subfigure}
    \begin{subfigure}[b]{0.3\textwidth}
        \centering
        \includegraphics[width=1.0\textwidth]{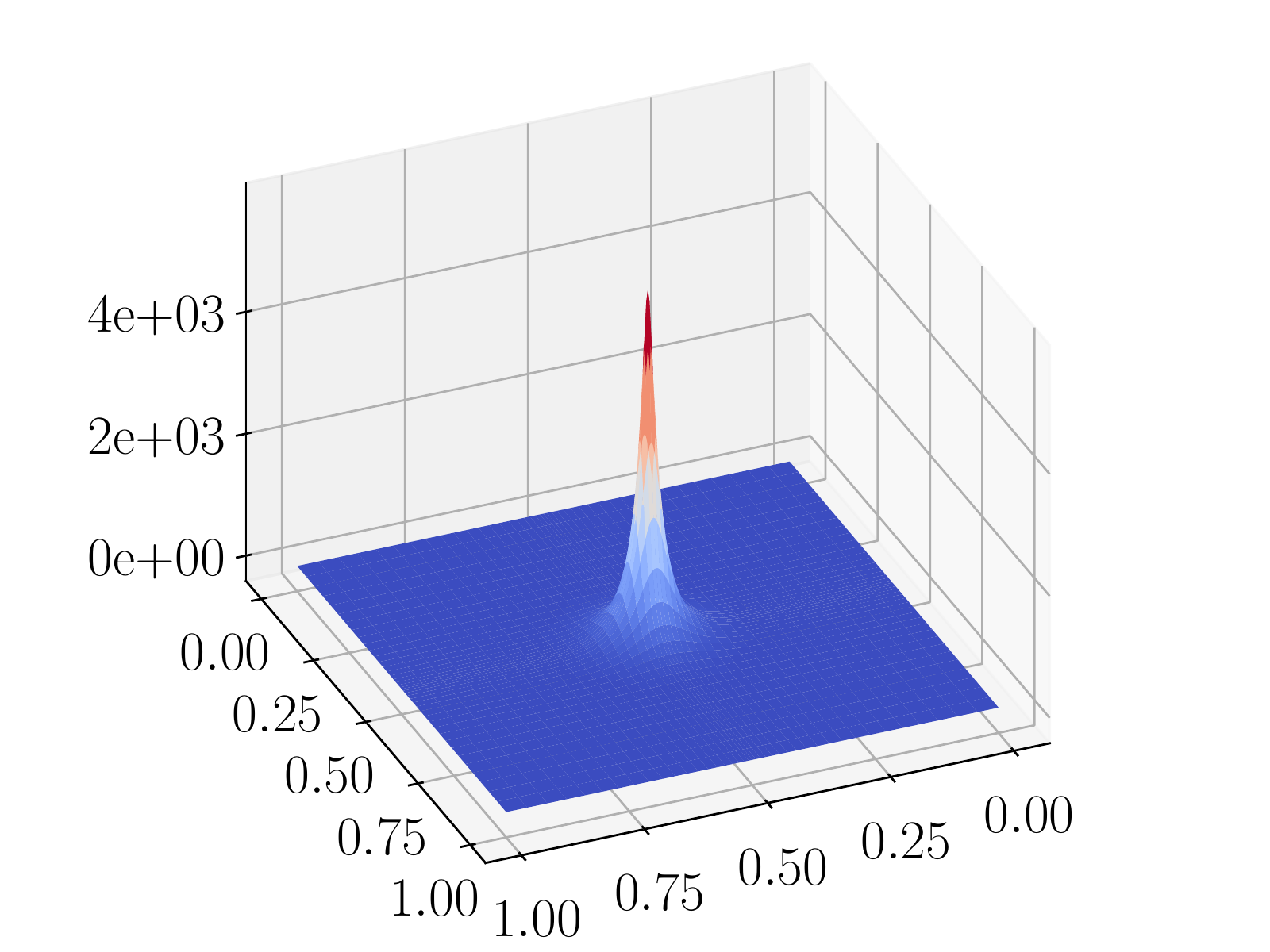}
    \end{subfigure}
    \begin{subfigure}[b]{0.3\textwidth}
        \centering
        \includegraphics[width=1.0\textwidth]{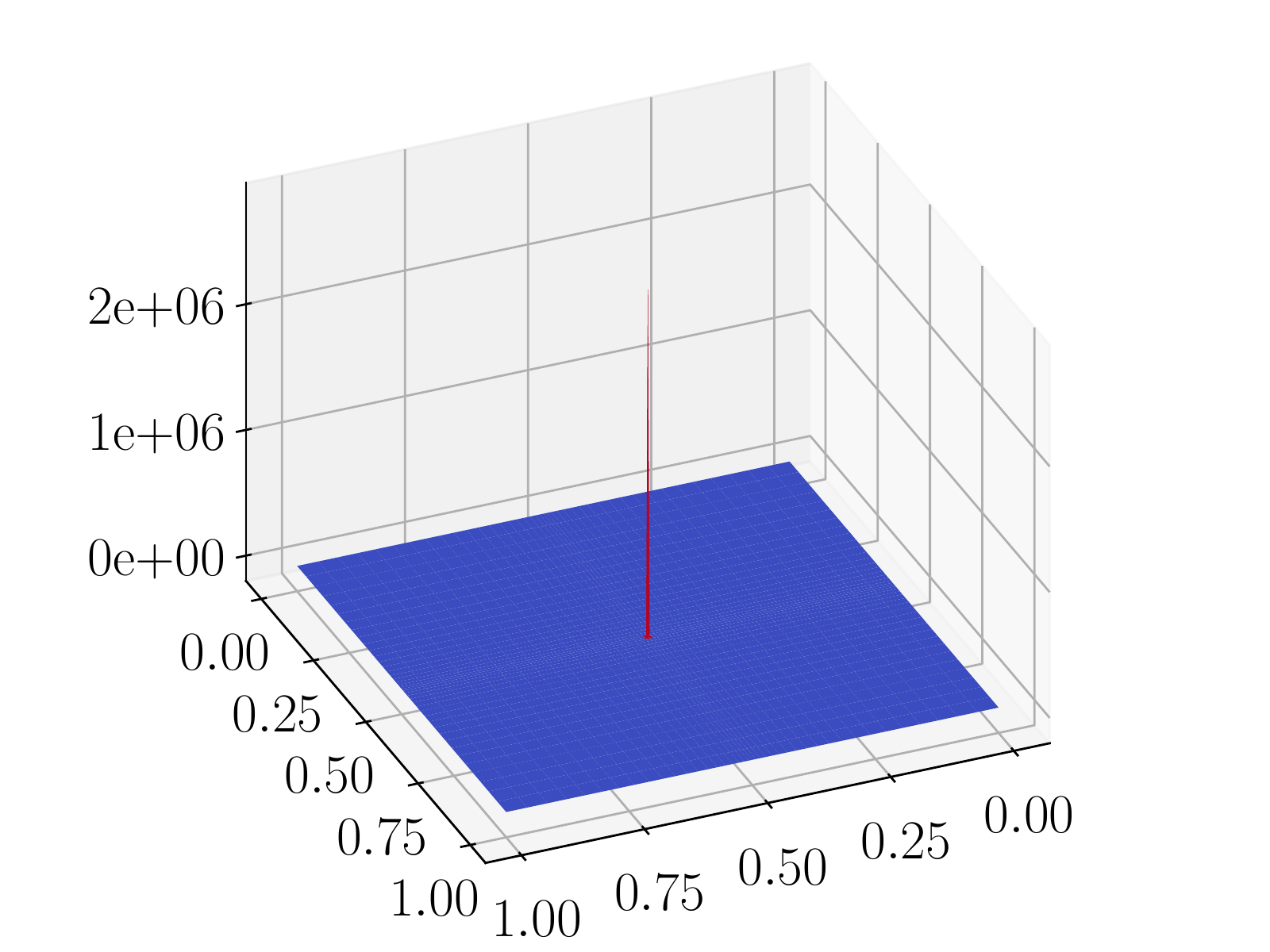}
    \end{subfigure}
        \begin{subfigure}[b]{0.3\textwidth}
        \centering
        \includegraphics[width=1.0\textwidth]{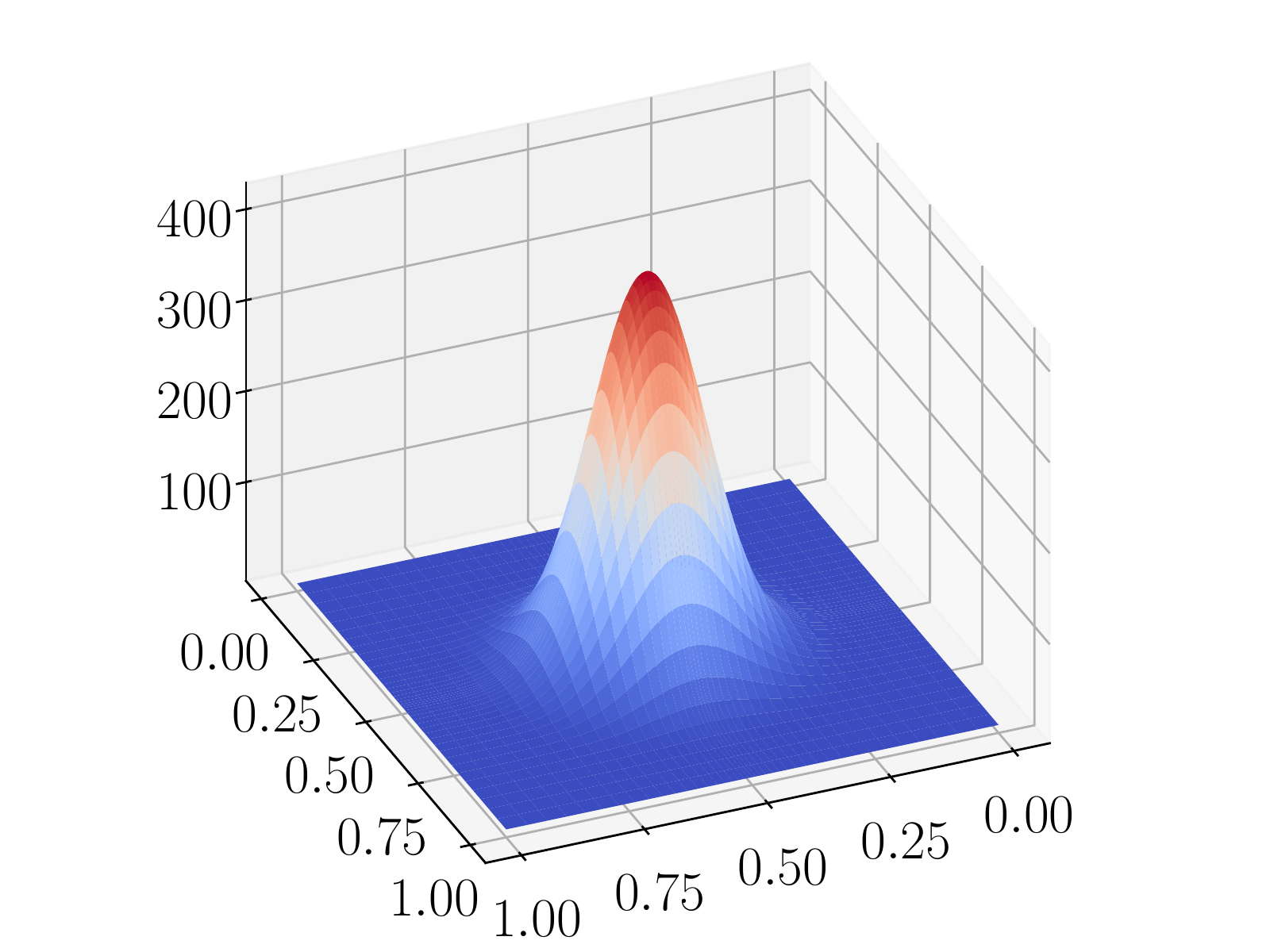}
    \end{subfigure}
    \begin{subfigure}[b]{0.3\textwidth}
        \centering
        \includegraphics[width=1.0\textwidth]{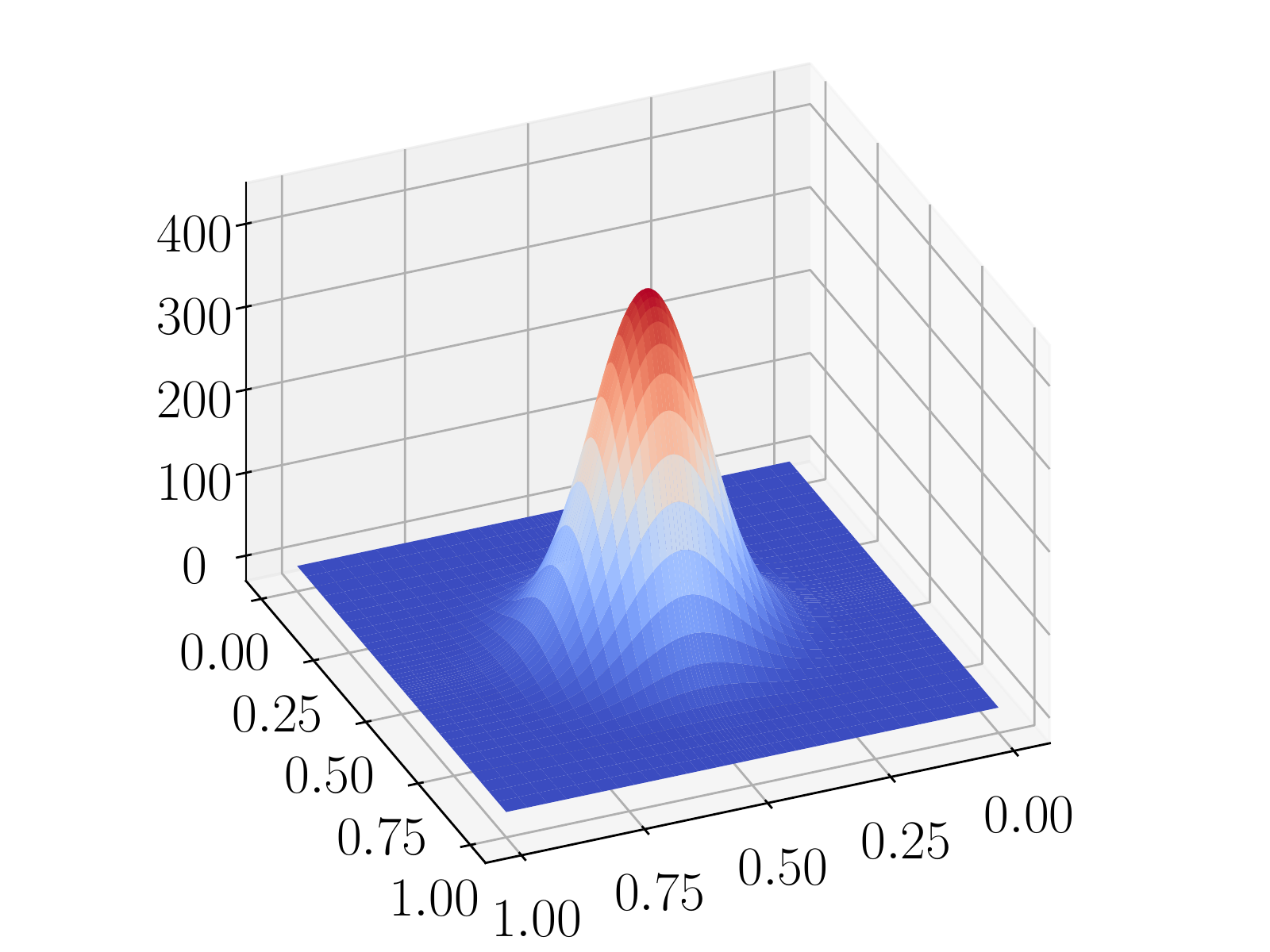}
    \end{subfigure}
    \begin{subfigure}[b]{0.3\textwidth}
        \centering
        \includegraphics[width=1.0\textwidth]{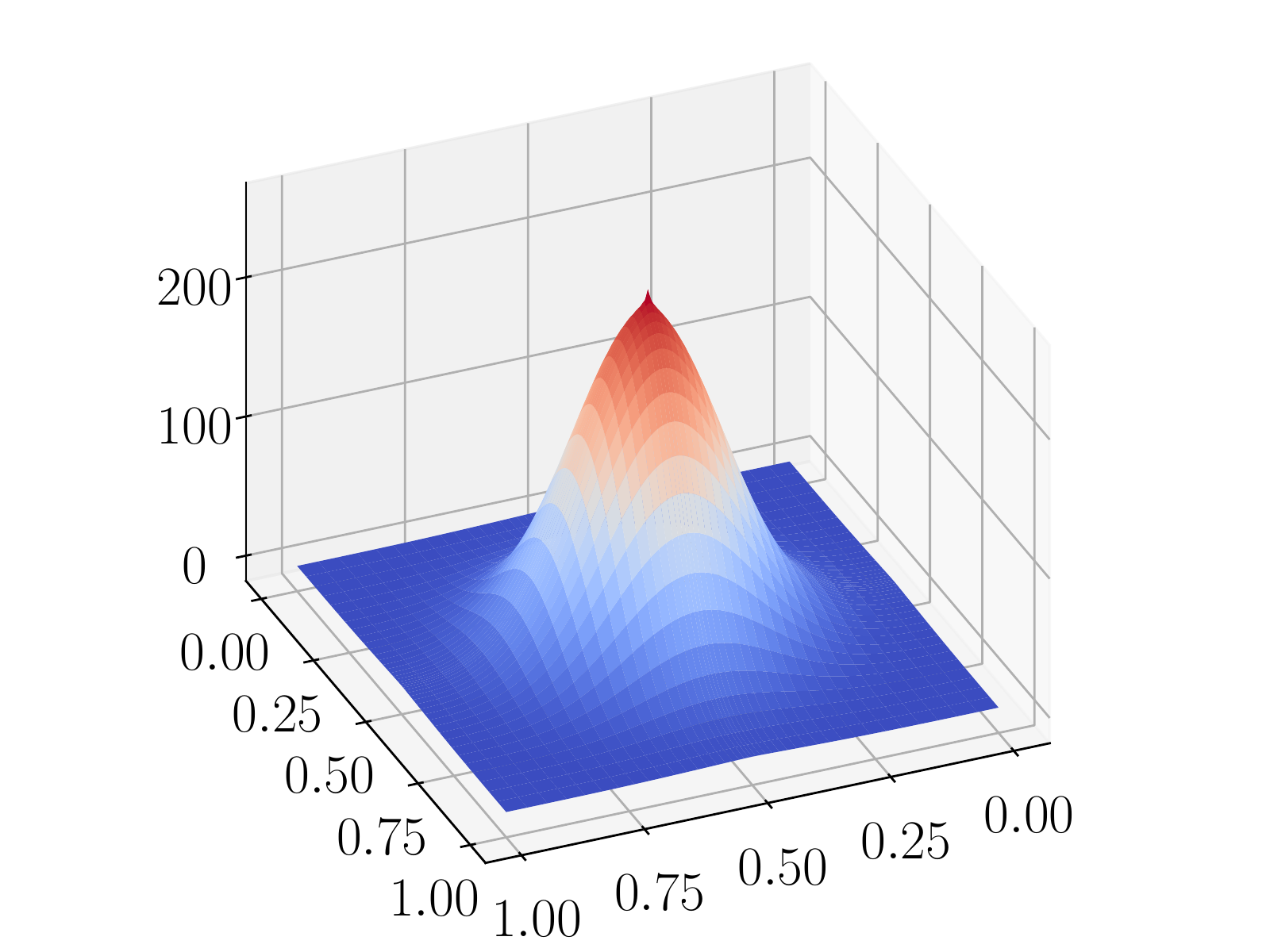}
    \end{subfigure}
    \caption{Algorithm 1: Evolution of $u_h$ (top) and $v_h$ (bottom) at times $t=0$, $2\cdot 10^{-5}$ and $5\cdot 10^{-3}$.}\label{fig.e2-evolution}
\end{figure}
\begin{figure}
    \begin{subfigure}[b]{0.3\textwidth}
        \centering
        \includegraphics[width=1.0\textwidth]{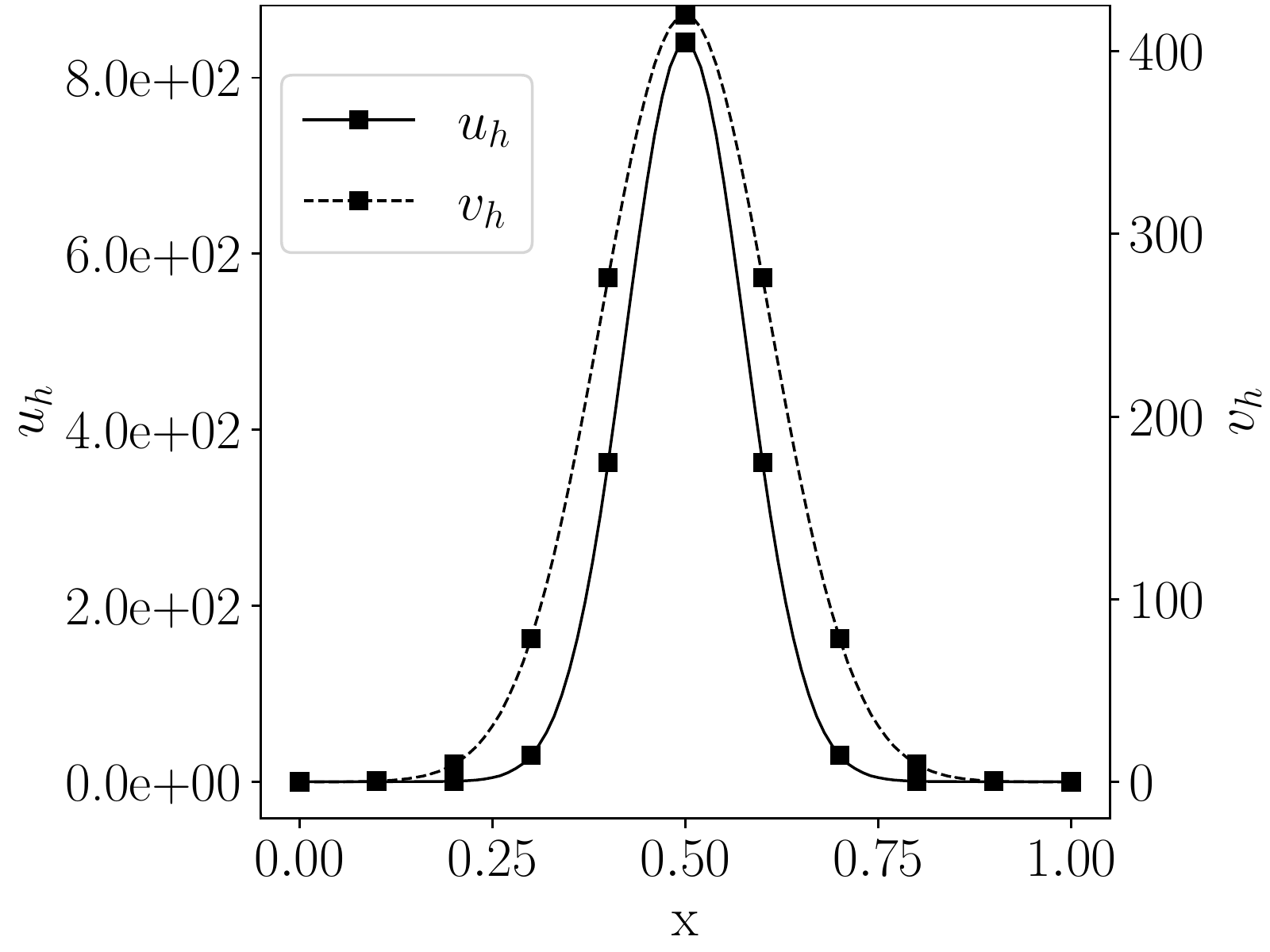}
    \end{subfigure}
    \begin{subfigure}[b]{0.3\textwidth}
        \centering
        \includegraphics[width=1.0\textwidth]{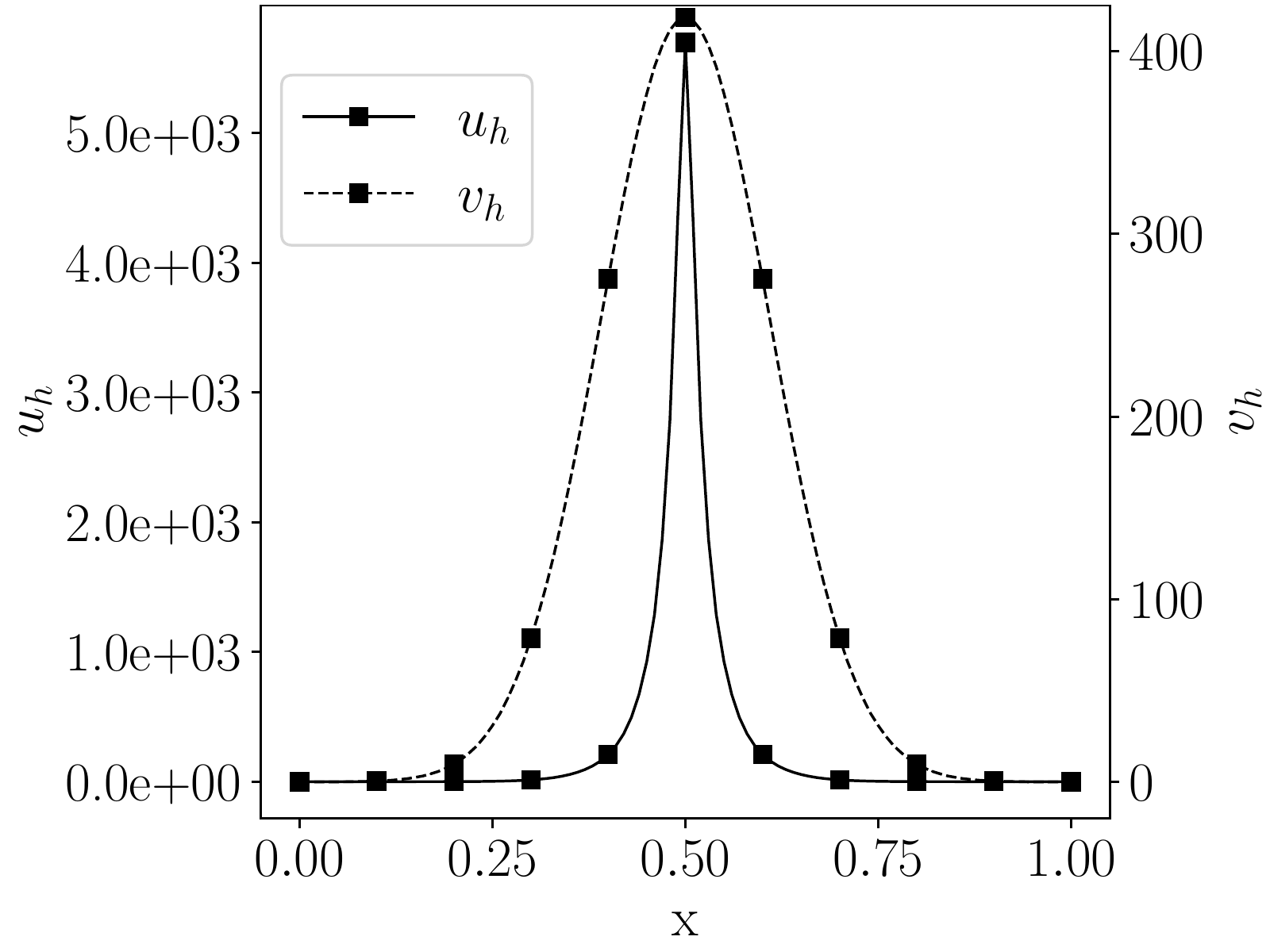}
    \end{subfigure}
    \begin{subfigure}[b]{0.3\textwidth}
        \centering
        \includegraphics[width=1.0\textwidth]{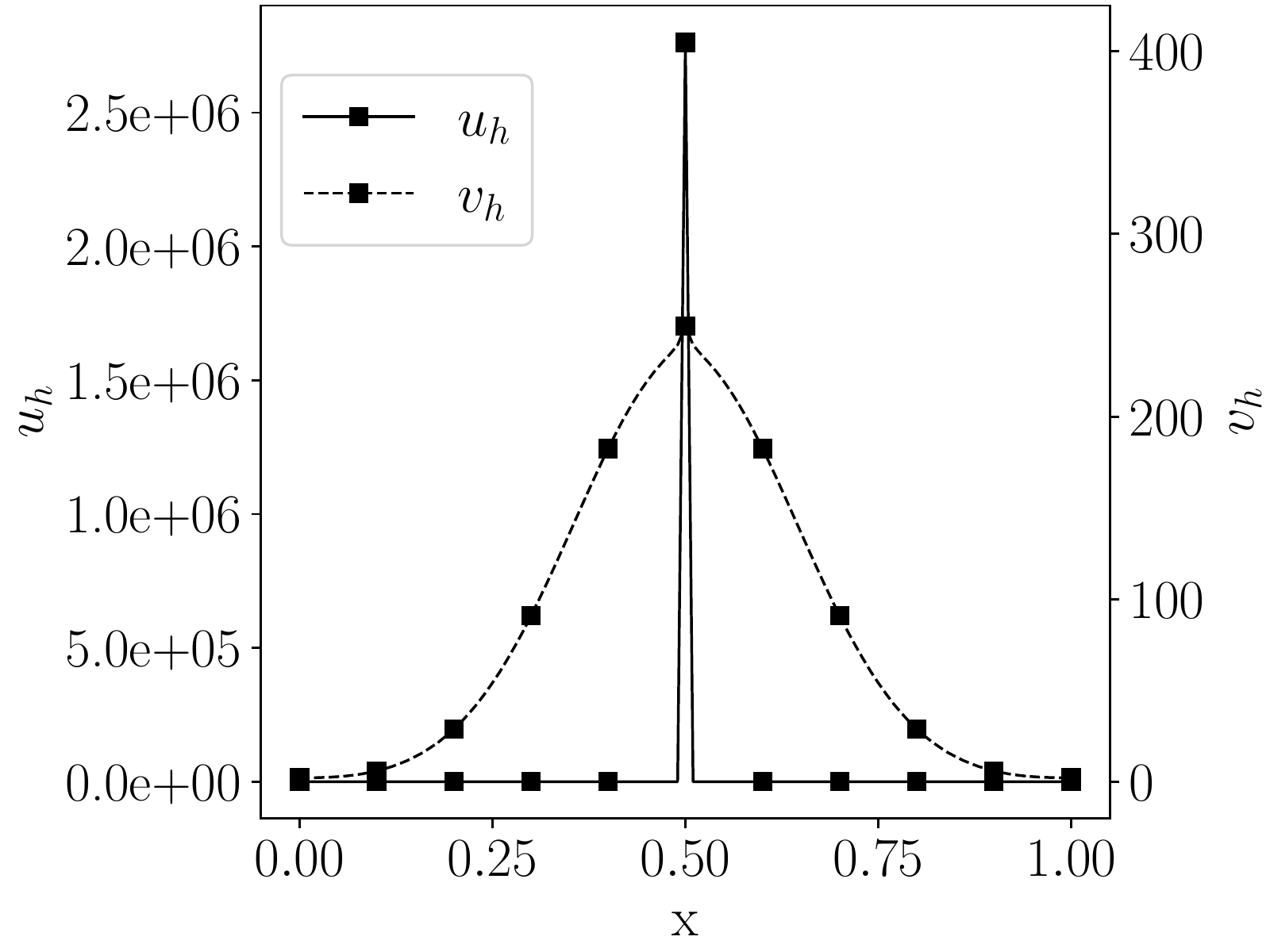}
    \end{subfigure}
    \caption{Algorthim 1: Profiles of $u_h$ and $v_h$ along the plane $y=0.5$ at times $t=0$, $2\cdot 10^{-5}$ and $5\cdot 10^{-3}$.}\label{fig.e2-profiles}
\end{figure}
\begin{figure}
    \centering
    \includegraphics[width=0.4\textwidth]{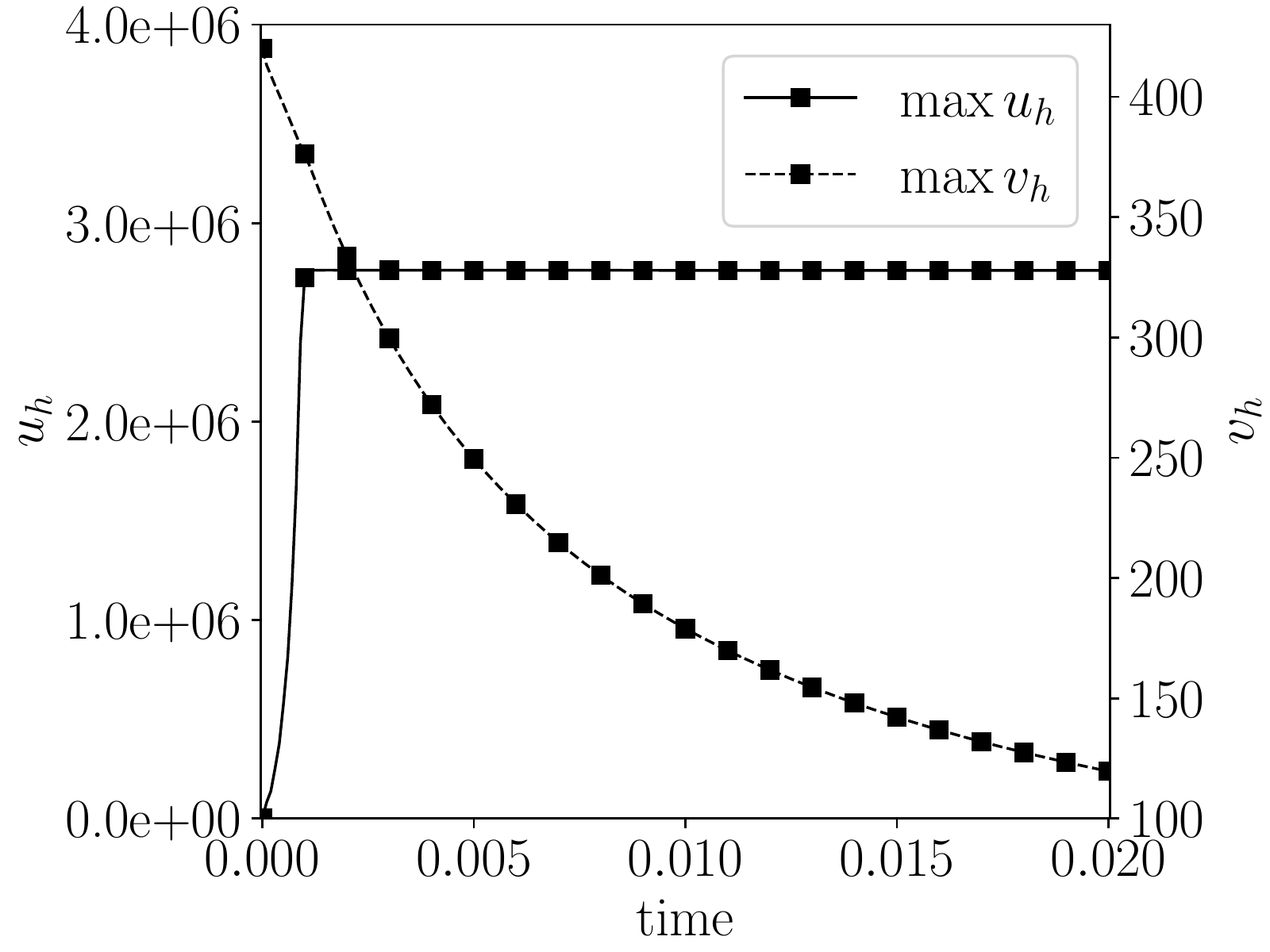}
    \includegraphics[width=0.4\textwidth]{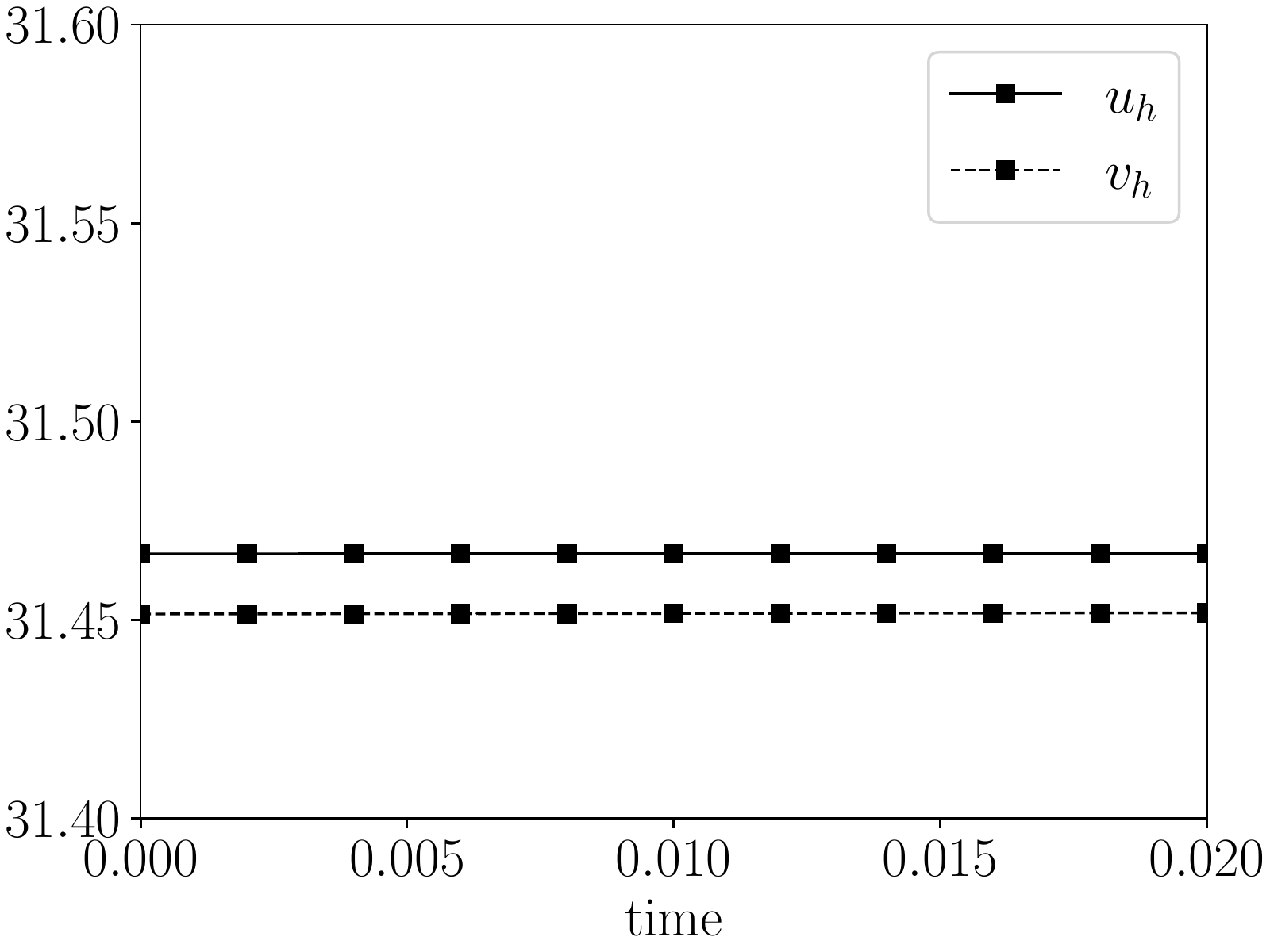}
    \caption{Algorithm 1: Evolution of $\|u_h\|_{L^\infty(\Omega)}$ and $\|v_h\|_{L^\infty(\Omega)}$ (left), and $\|u_h\|_{L^1(\Omega)}$ and $\|v_h\|_{L^1(\Omega)}$ (right).}
    \label{fig.e2-maximums}
\end{figure}
\begin{figure}
    \begin{subfigure}[b]{0.32\textwidth}
        \centering
        \includegraphics[width=1.0\textwidth]{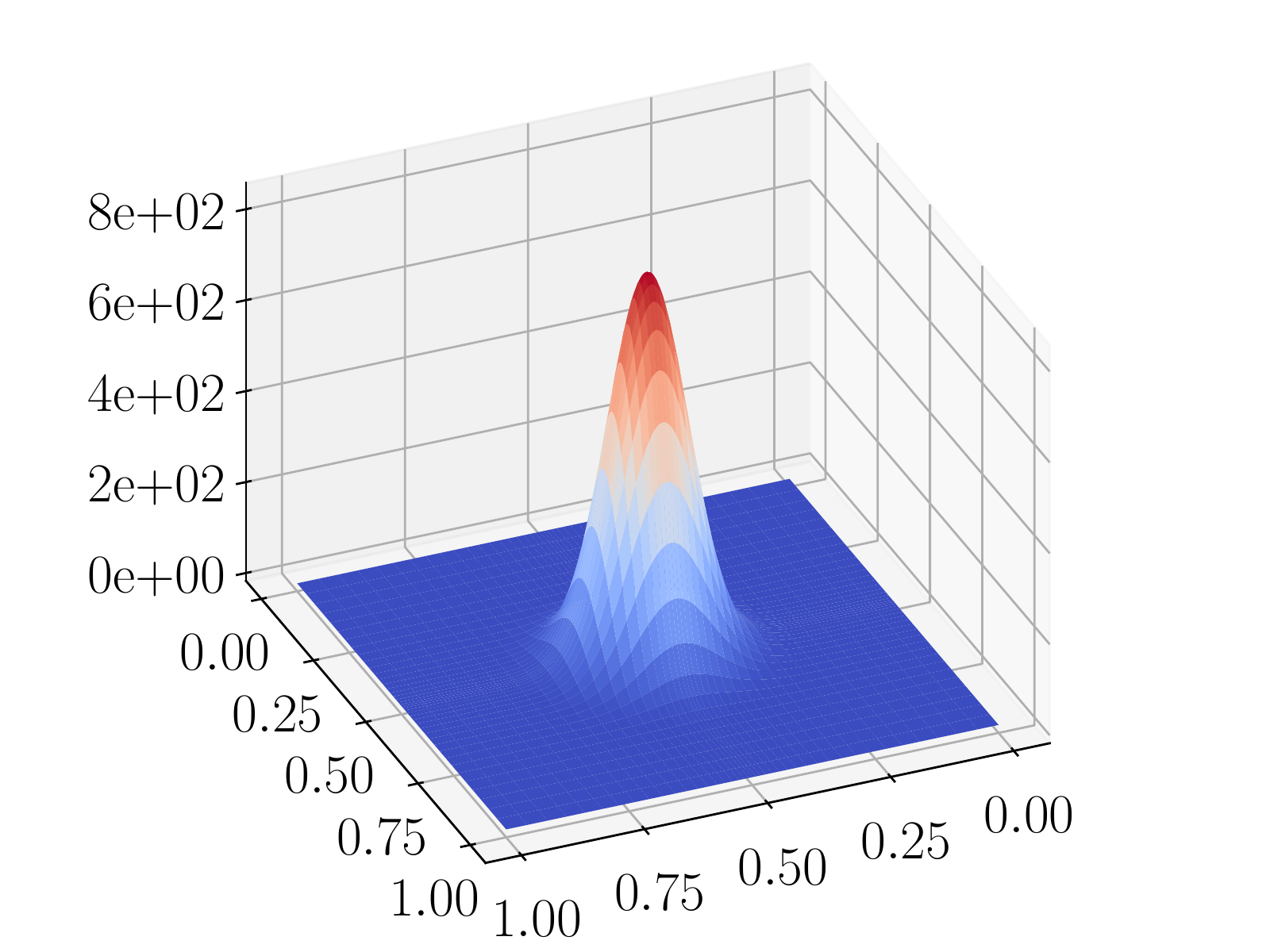}
    \end{subfigure}
    \begin{subfigure}[b]{0.32\textwidth}
        \centering
        \includegraphics[width=1.0\textwidth]{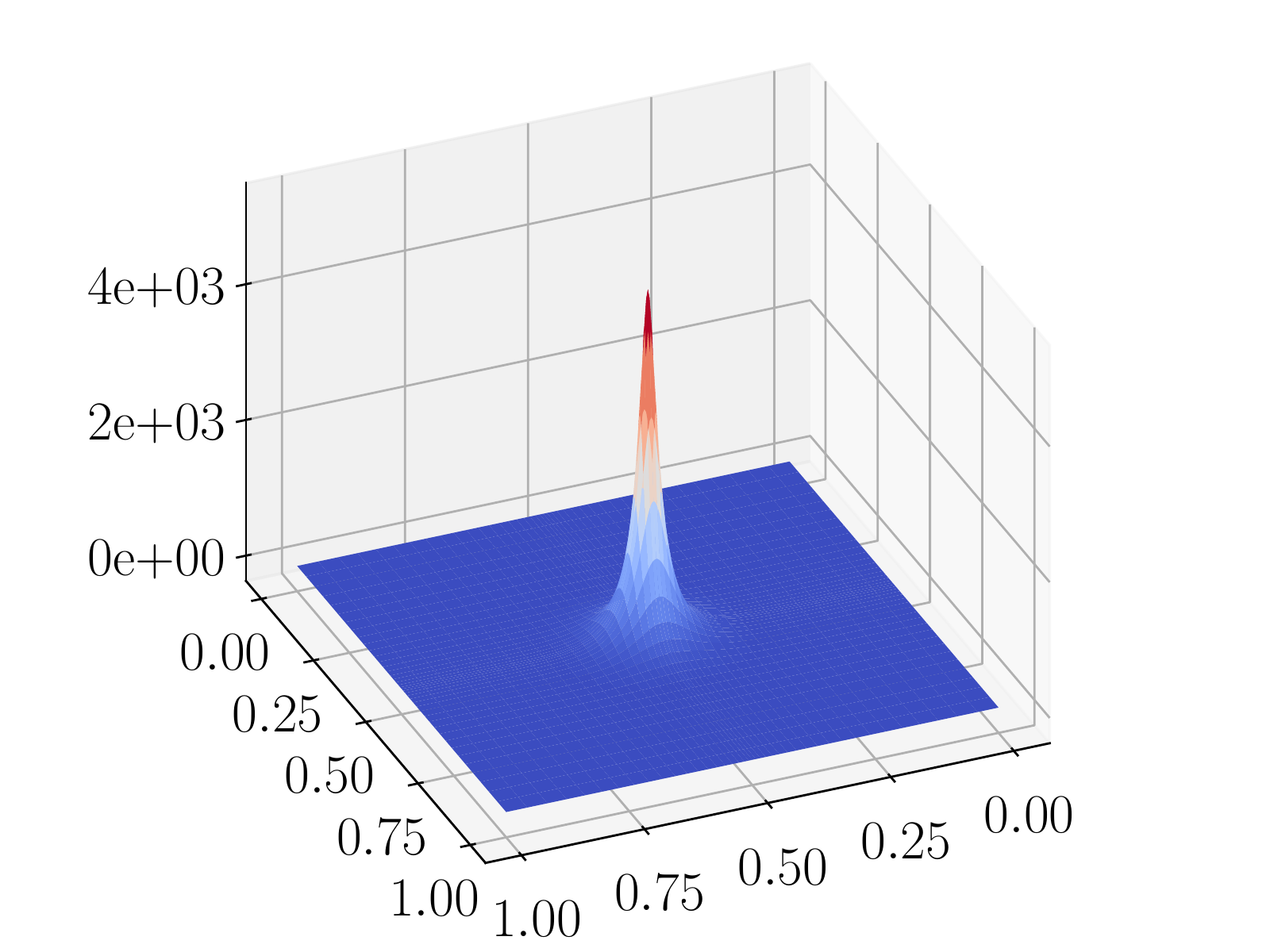}
    \end{subfigure}
    \begin{subfigure}[b]{0.32\textwidth}
        \centering
        \includegraphics[width=1.0\textwidth]{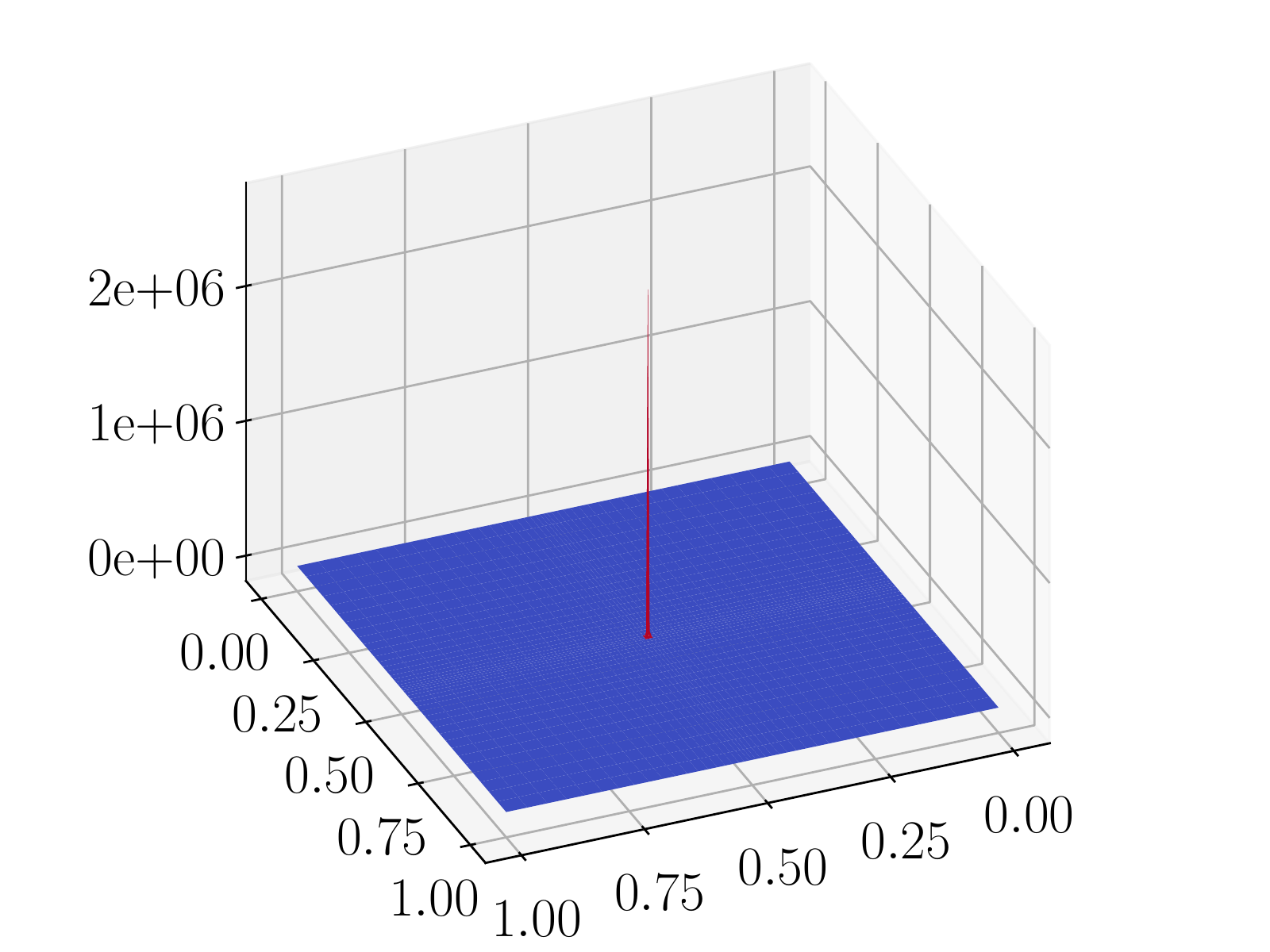}
    \end{subfigure}
    \begin{subfigure}[b]{0.32\textwidth}
        \centering
        \includegraphics[width=1.0\textwidth]{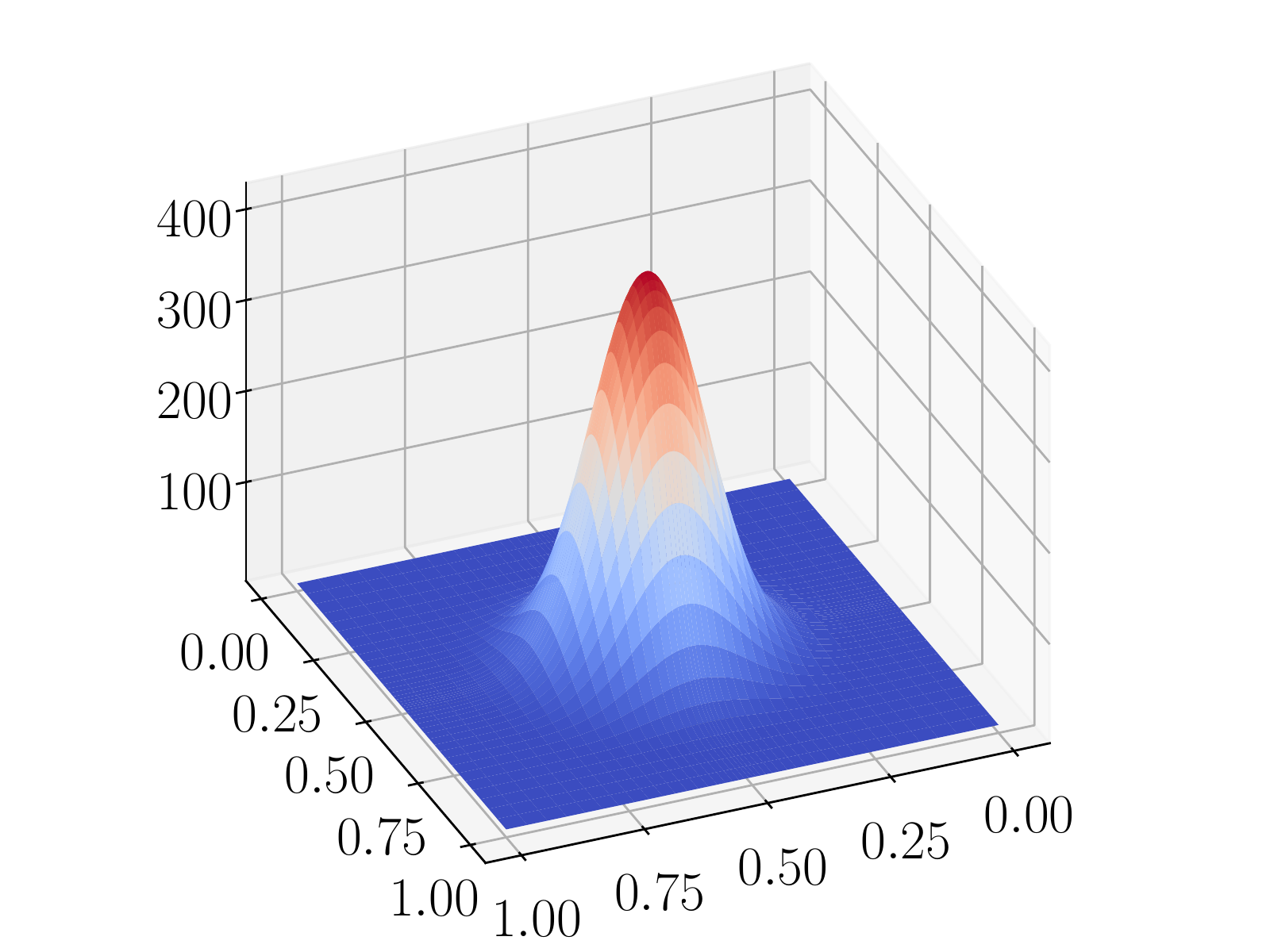}
    \end{subfigure}
    \begin{subfigure}[b]{0.32\textwidth}
        \centering
        \includegraphics[width=1.0\textwidth]{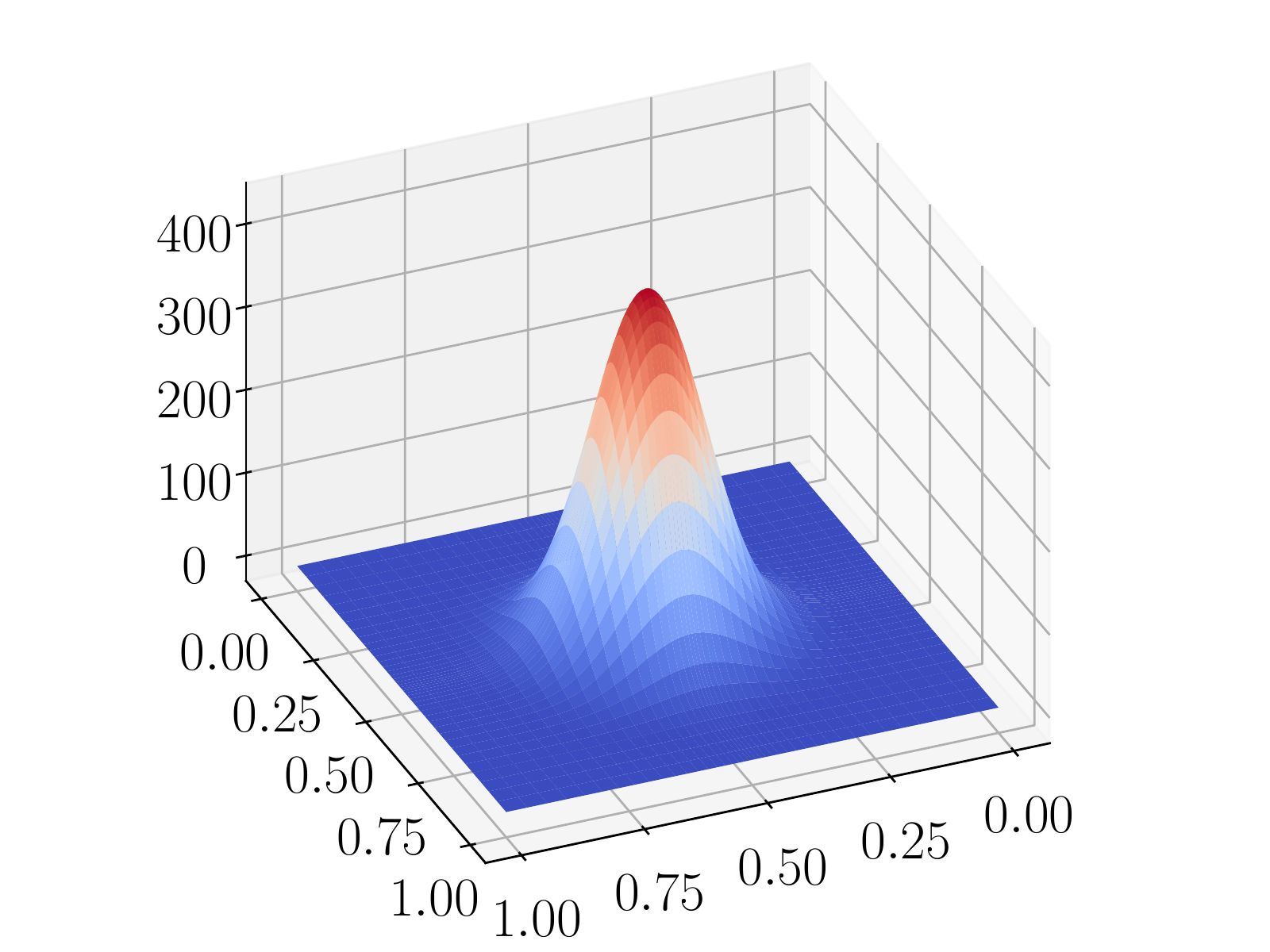}
    \end{subfigure}
    \begin{subfigure}[b]{0.32\textwidth}
        \centering
        \includegraphics[width=1.0\textwidth]{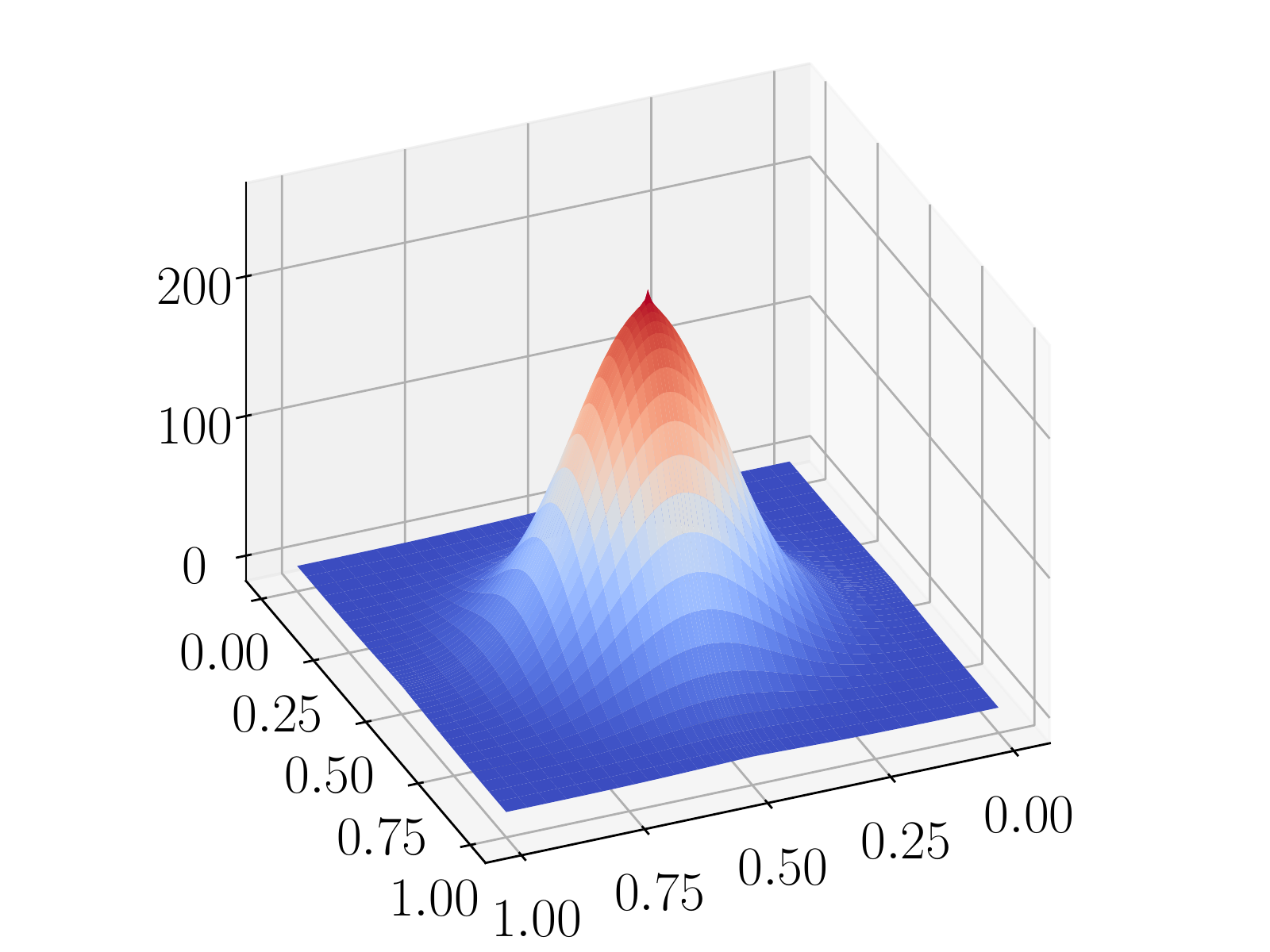}
    \end{subfigure}
    \caption{Algorithm 2: Evolution of $u_h$ and $v_h$ at times $t=0$, $2\cdot 10^{-5}$ and $5\cdot 10^{-3}$.}\label{fig.e2-evolution-alg2}
\end{figure}
\begin{figure}
    \begin{subfigure}[b]{0.32\textwidth}
        \centering
        \includegraphics[width=1.0\textwidth]{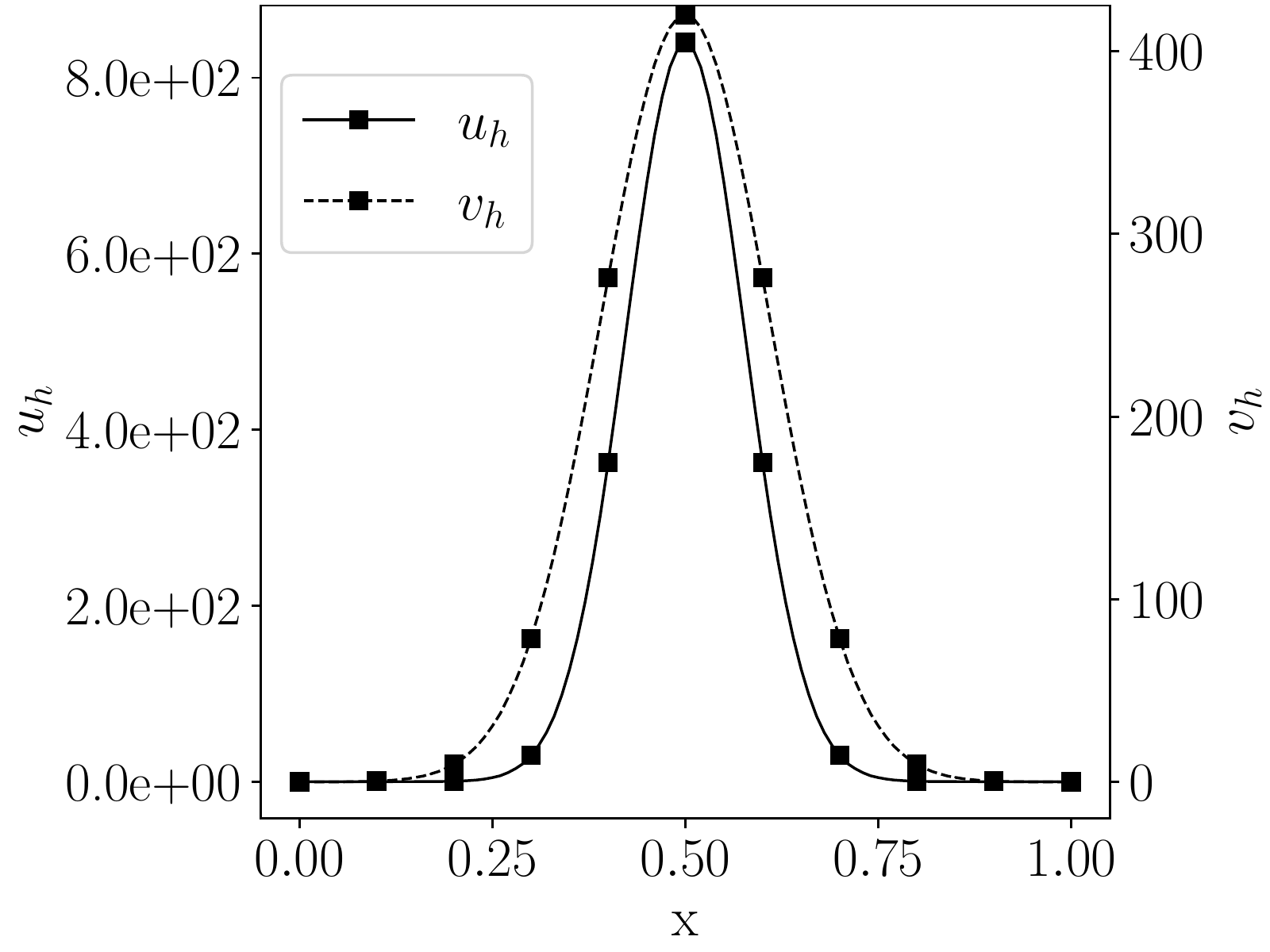}
    \end{subfigure}
    \begin{subfigure}[b]{0.32\textwidth}
        \centering
        \includegraphics[width=1.0\textwidth]{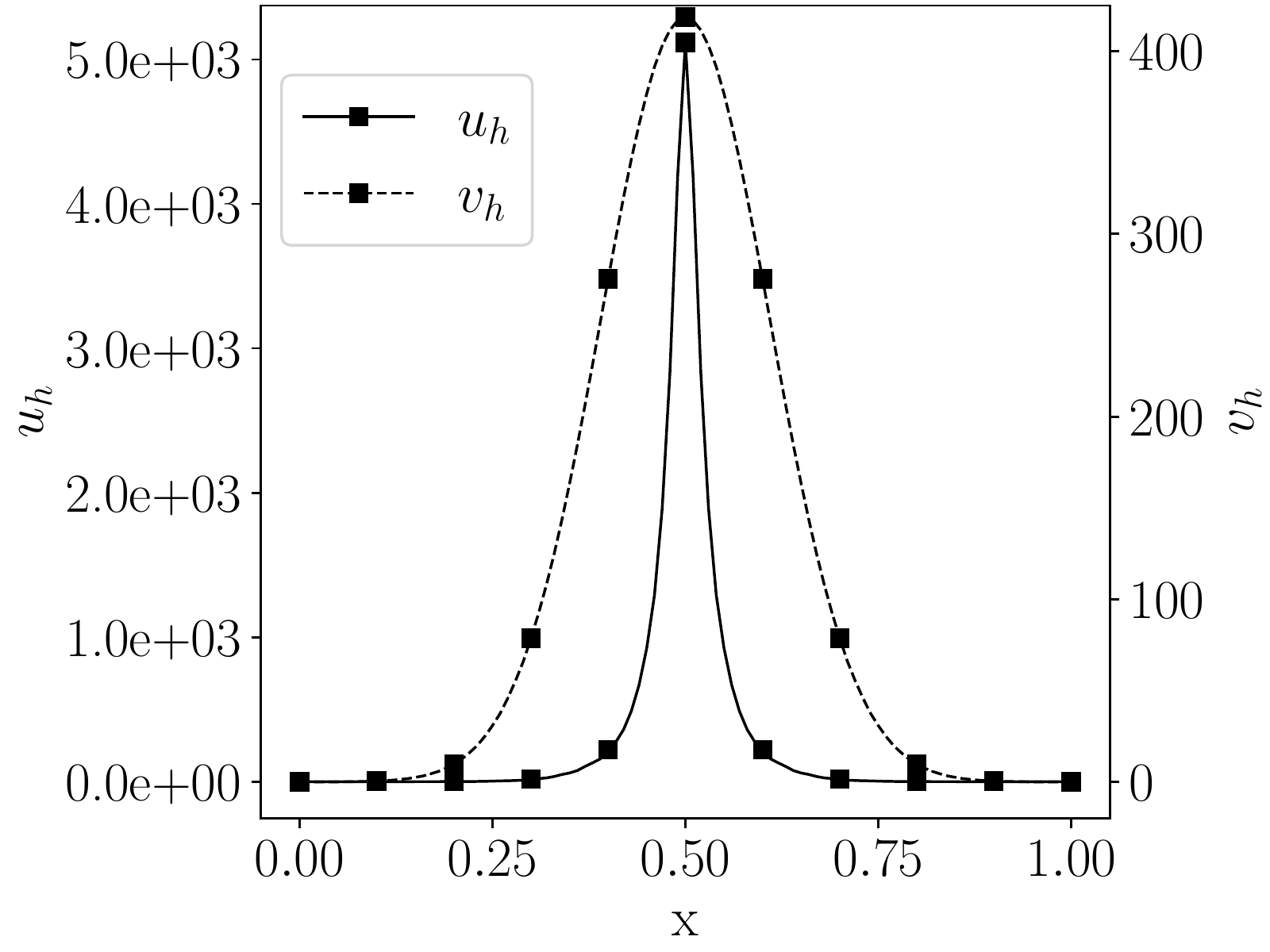}
    \end{subfigure}
    \begin{subfigure}[b]{0.32\textwidth}
        \centering
        \includegraphics[width=1.0\textwidth]{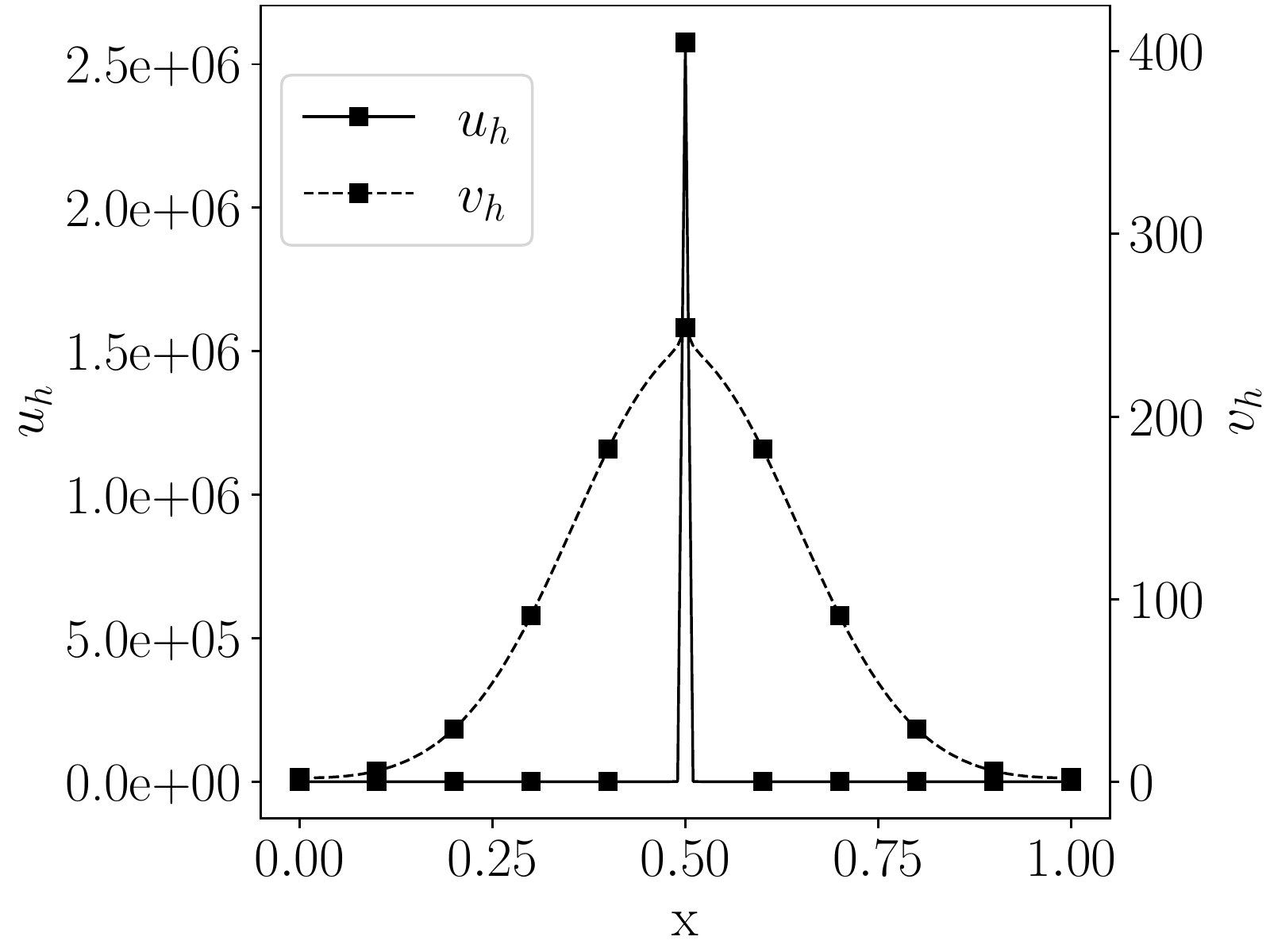}
    \end{subfigure}
    \caption{Algorithm 2: Profiles of $u_h$ and $v_h$ at $y=0.5$ at times $t=0$, $2\cdot 10^{-5}$ and $5\cdot 10^{-3}$.}\label{fig.e2-profiles-alg2}
\end{figure}
\begin{figure}
    \centering
    \includegraphics[width=0.4\textwidth]{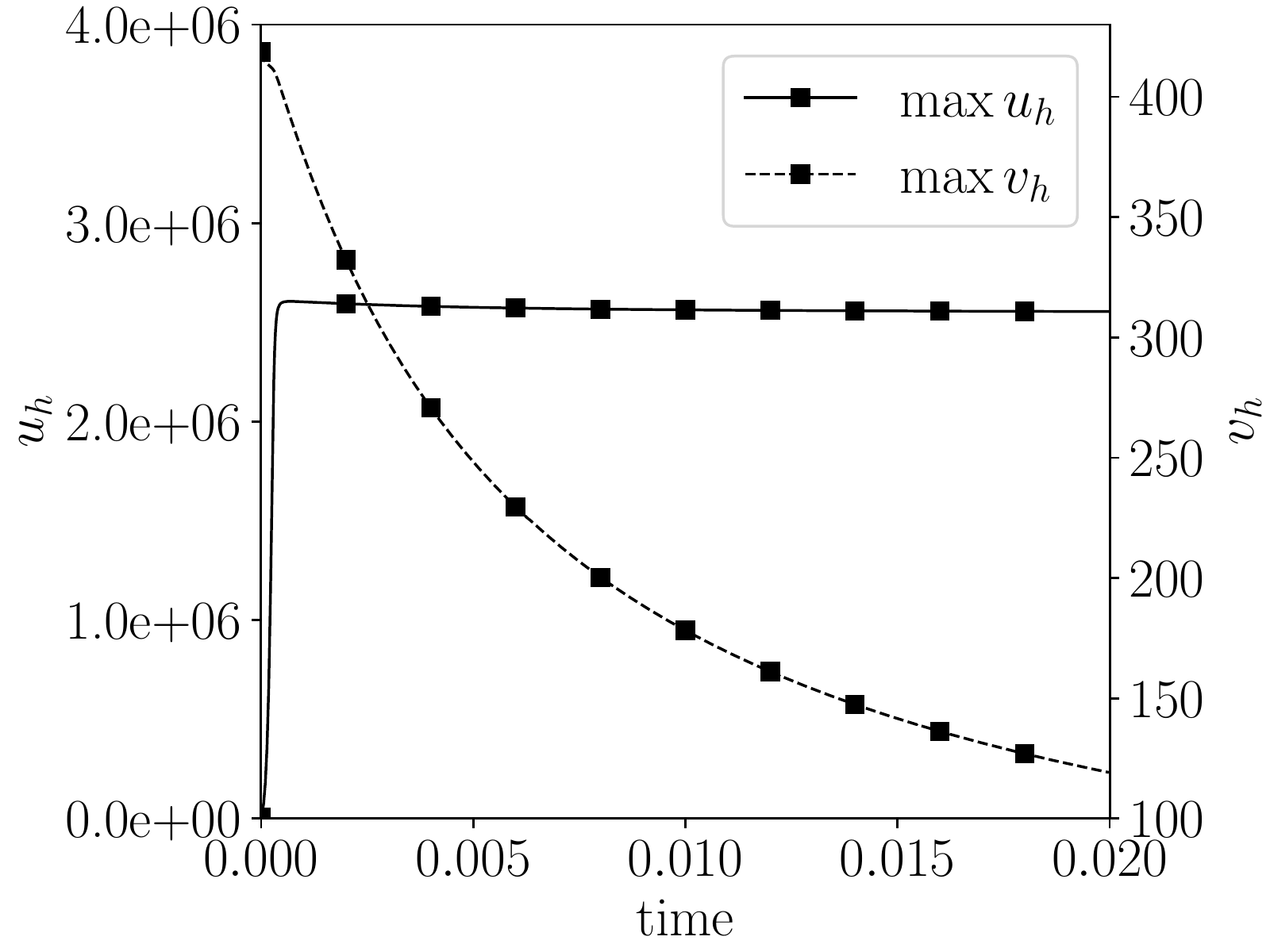}
    \includegraphics[width=0.4\textwidth]{centered_blow-up_alg2/conservation.pdf}
    \caption{Algorithm 2: Evolution of $\|u_h\|_{L^\infty(\Omega)}$ and $\|v_h\|_{L^\infty(\Omega)}$ (left), and $\|u_h\|_{L^1(\Omega)}$ and $\|v_h\|_{L^1(\Omega)}$ (right).}
    \label{fig.e2-maximums-alg2}
\end{figure}
\begin{figure}
    \centering
    \includegraphics[width=0.4\textwidth]{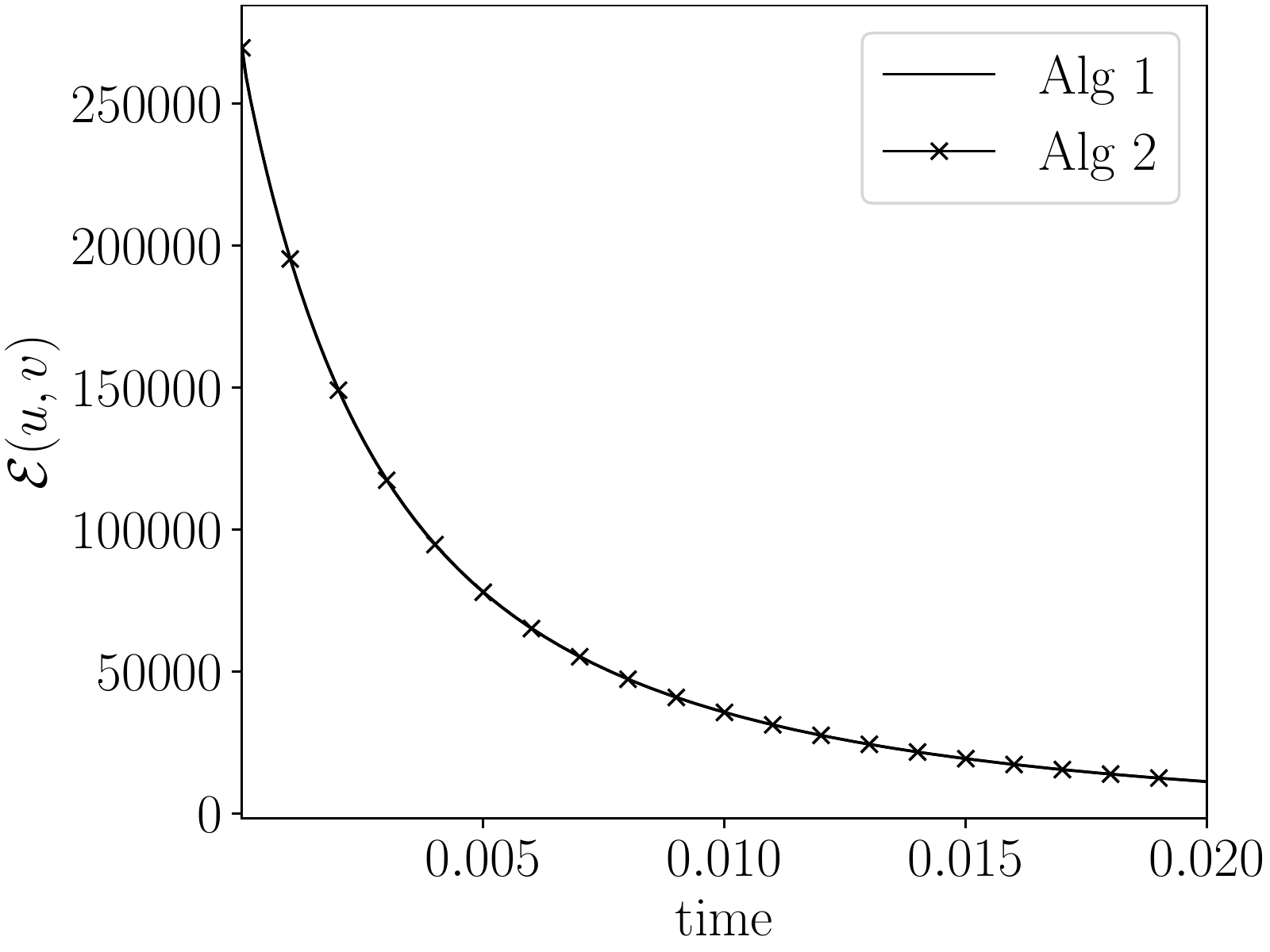}
    \caption{Evolution of $\mathcal{E}_h(u_h,v_h)$.}
    \label{fig.e2-energy}
\end{figure}

\subsubsection{Blow-up at the boundary} This final example is taken from \cite[Section 3.2]{Strehl_Sokolov_Kuzmin_Turek_2010}. It is assumed that  $\Omega=\textrm{B}((0, 0.1);1))$ is a ball of diameter $1$ centered at $(0,0.1)$, and the initial conditions $u_0$ and $v_0$ are
\begin{equation}
    u_0(x, y) = 1000 e^{-100(x^2+y^2)} \quad \text{and} \quad v_0(x, y) = 0.
\end{equation}
Here $\int_\Omega u_0(\x)\,\dx\approx 31.44722>4\pi$, which results in a continuous solution, which should have a singularity as well, but on this occasion occurs on the boundary. 

From a numerical point of view, this example is more demanding than the previous one since the maximum location moves over time. 

We use a time step of $k=10^{-4}$ and an unstructured mesh consisting of $5000$ triangles with a minimum mesh size $h_{\rm min}\approx 0.0087$ (see Figure \ref{fig.e3-mesh}). As in the previous example, we also solve the problem by using both algorithms and comparing their results. 
\begin{figure}
    \centering
    \includegraphics[width=0.4\textwidth,trim={30cm 0 35cm 0},clip]{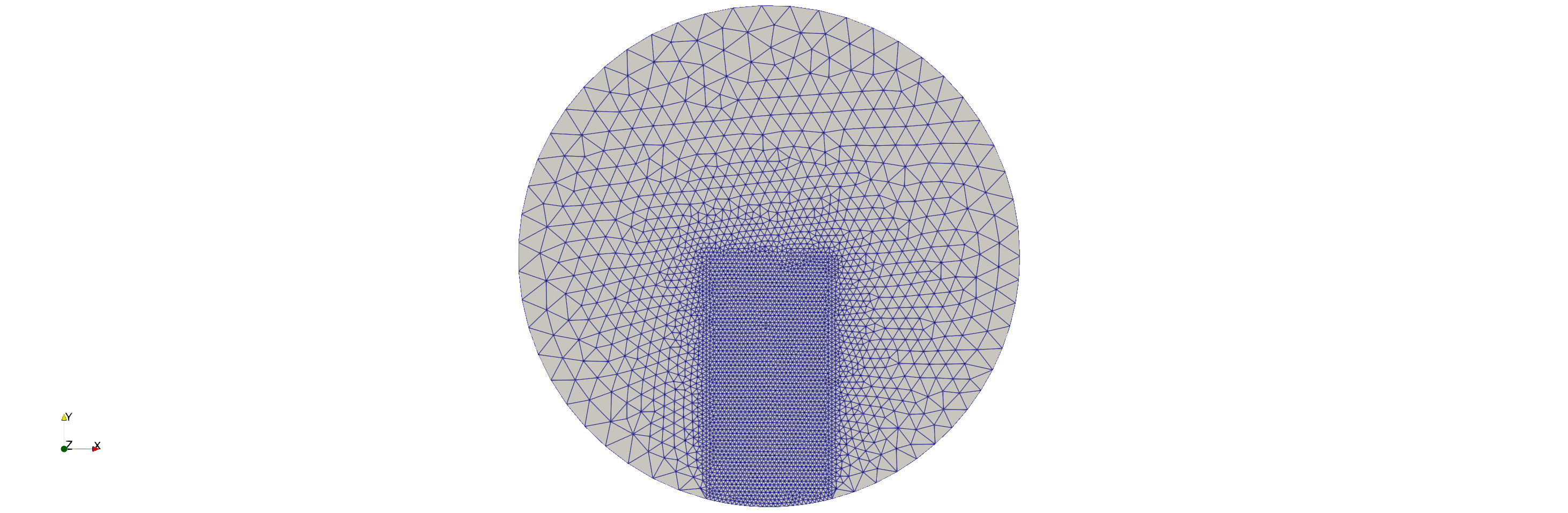}
    \caption{Mesh used to discretize $\Omega=\textrm{B}((0,0.1); 1)$.}
    \label{fig.e3-mesh}
\end{figure}

Figures  \ref{fig.e3-evolution} and \ref{fig.e3-evolution-alg2} show the cell and chemoattractant densities at $3$ different times for Algorithms $1$ and $2$, respectively; the initial condition, which has $\|u_0\|_{L^\infty(\Omega)}=10^{3}$ and $\|v_0\|_{L^\infty(\Omega)}=0$ as indicated in Figures \ref{fig.e3-profiles} and \ref{fig.e3-profiles-alg2}; the discrete solutions at $t=0.15$, which clearly shows how the point where the maximum is attained is found in the midway between the points $(0,0)$ and $(0, -0.4)$; and the discrete solution at $t=2$ just as the maximum is reached at the boundary whose support is one macroelement consisting of three triangles.       
\begin{figure}
    \begin{subfigure}[b]{0.32\textwidth}
        \centering
        \includegraphics[width=1.0\textwidth]{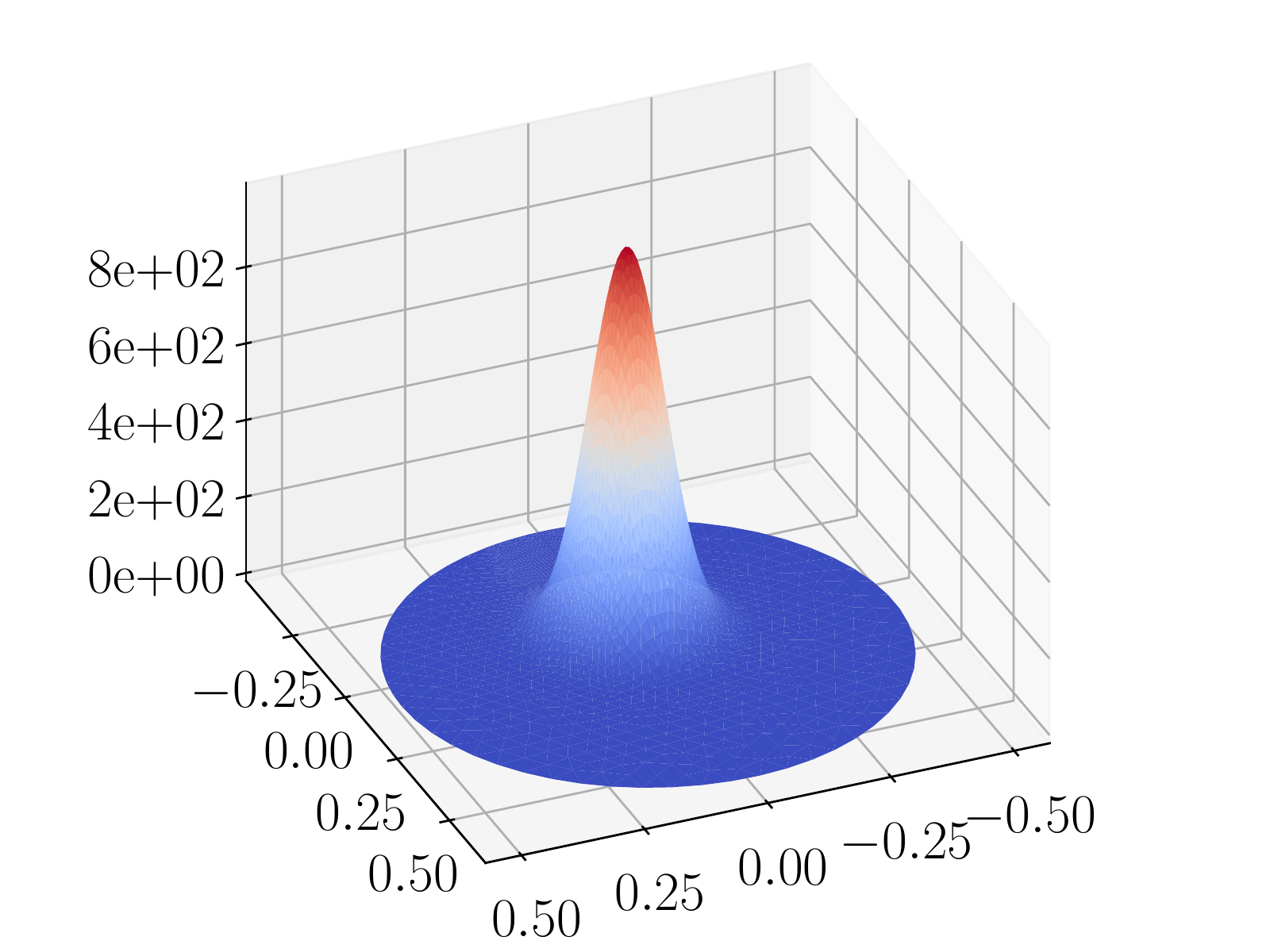}
    \end{subfigure}
    \begin{subfigure}[b]{0.32\textwidth}
        \centering
        \includegraphics[width=1.0\textwidth]{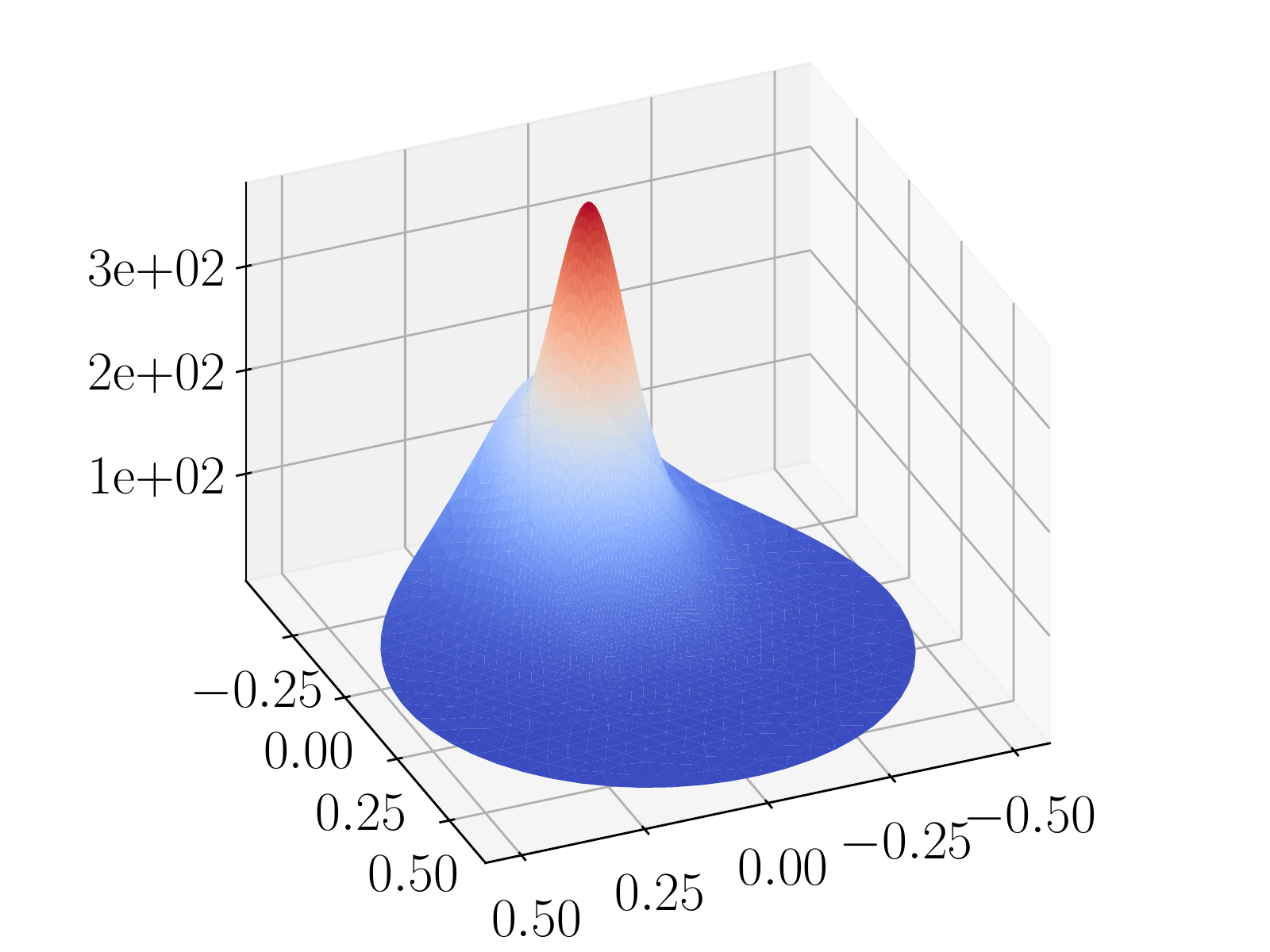}
    \end{subfigure}
    \begin{subfigure}[b]{0.32\textwidth}
        \centering
        \includegraphics[width=1.0\textwidth]{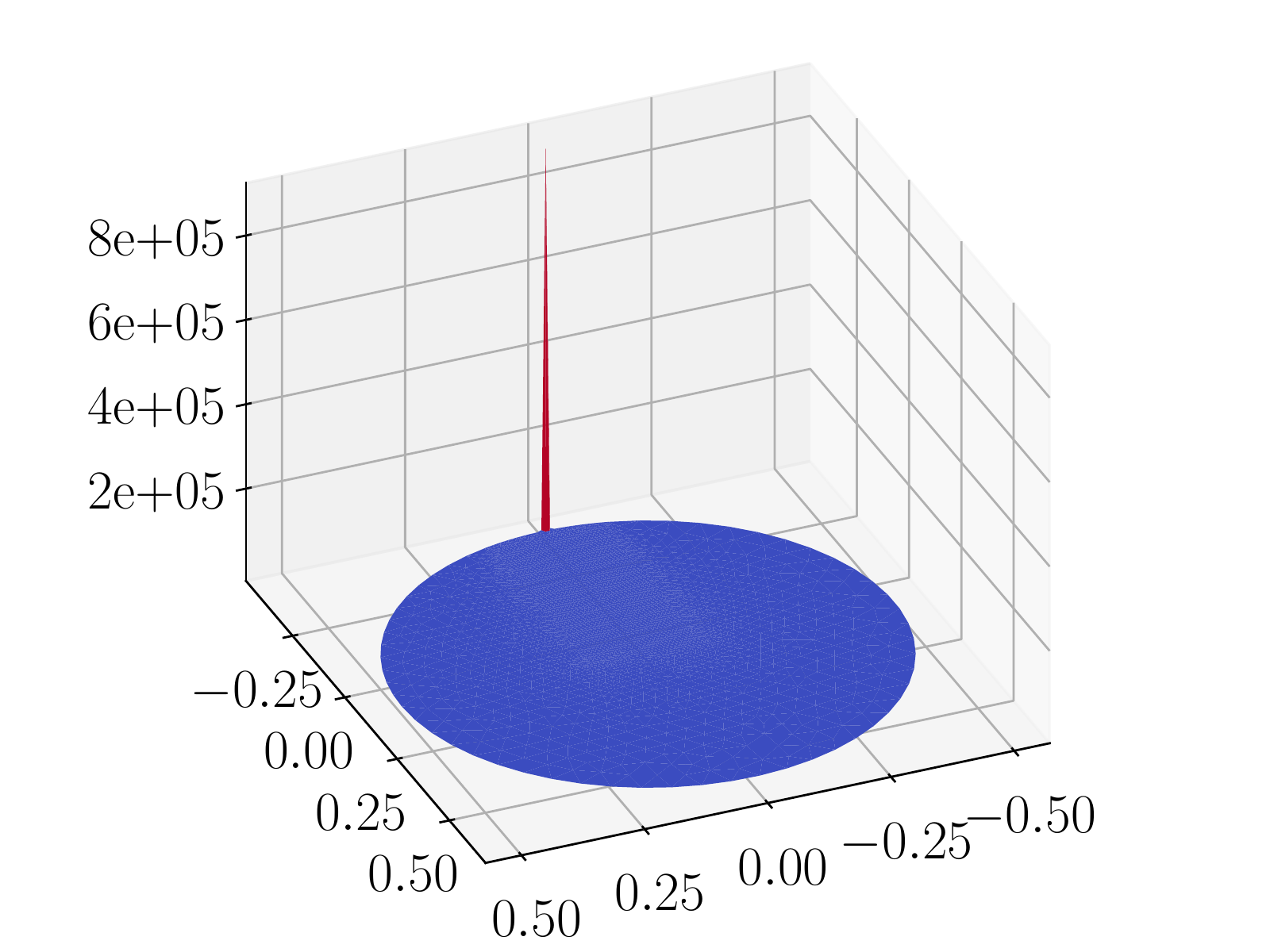}
    \end{subfigure}
    \begin{subfigure}[b]{0.32\textwidth}
        \centering
        \includegraphics[width=1.0\textwidth]{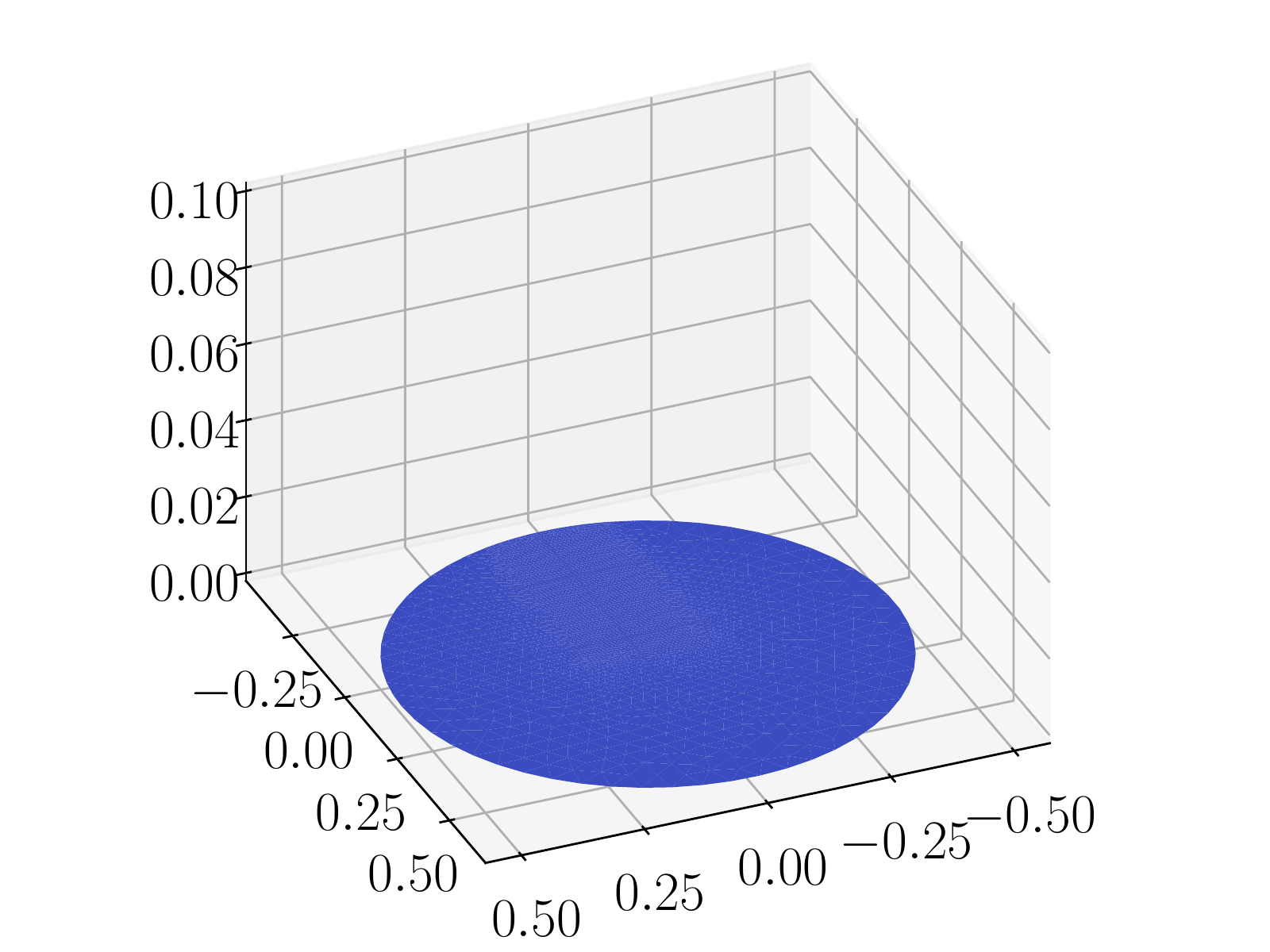}
    \end{subfigure}
    \begin{subfigure}[b]{0.32\textwidth}
        \centering
        \includegraphics[width=1.0\textwidth]{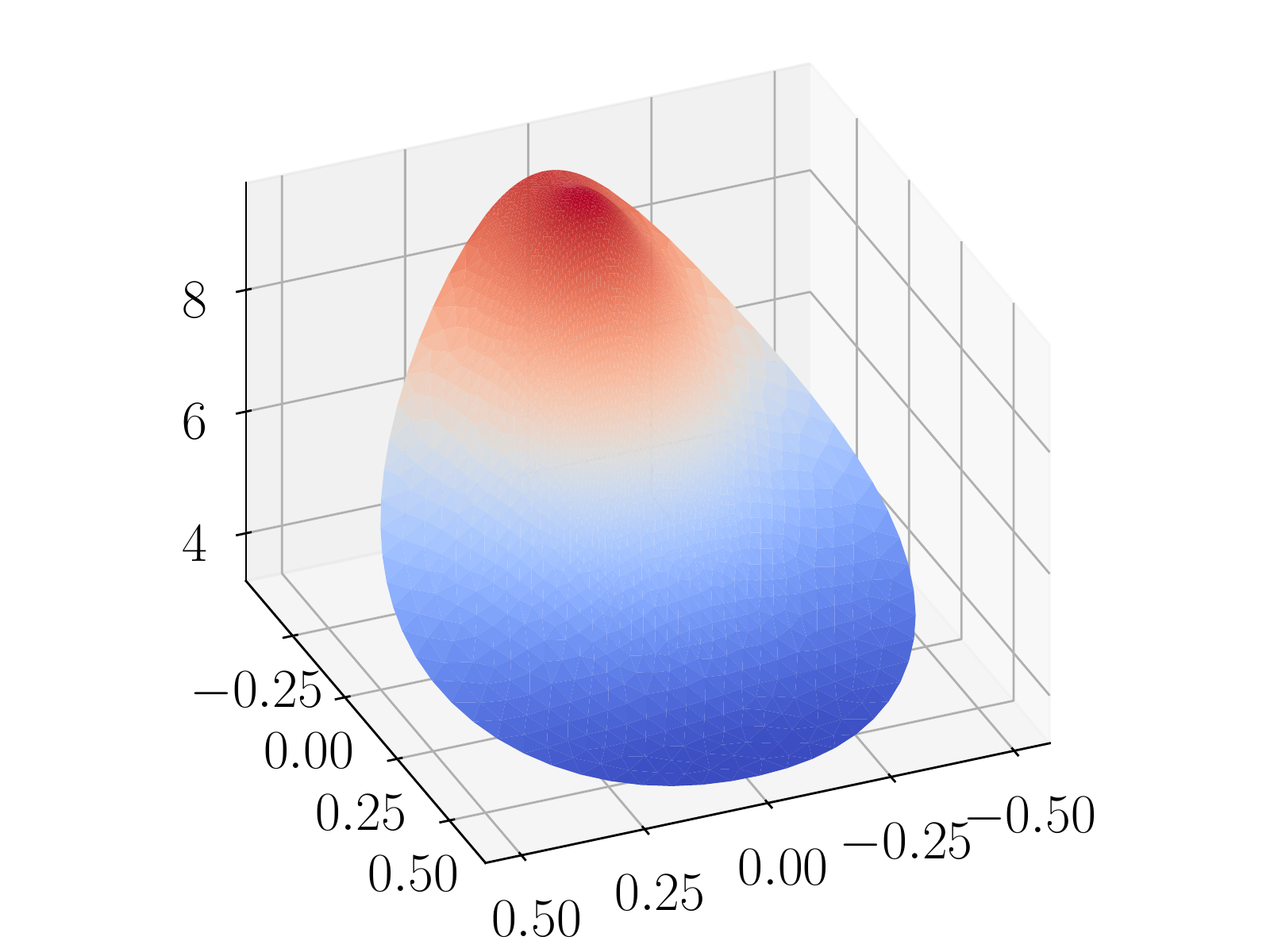}
    \end{subfigure}
    \begin{subfigure}[b]{0.32\textwidth}
        \centering
        \includegraphics[width=1.0\textwidth]{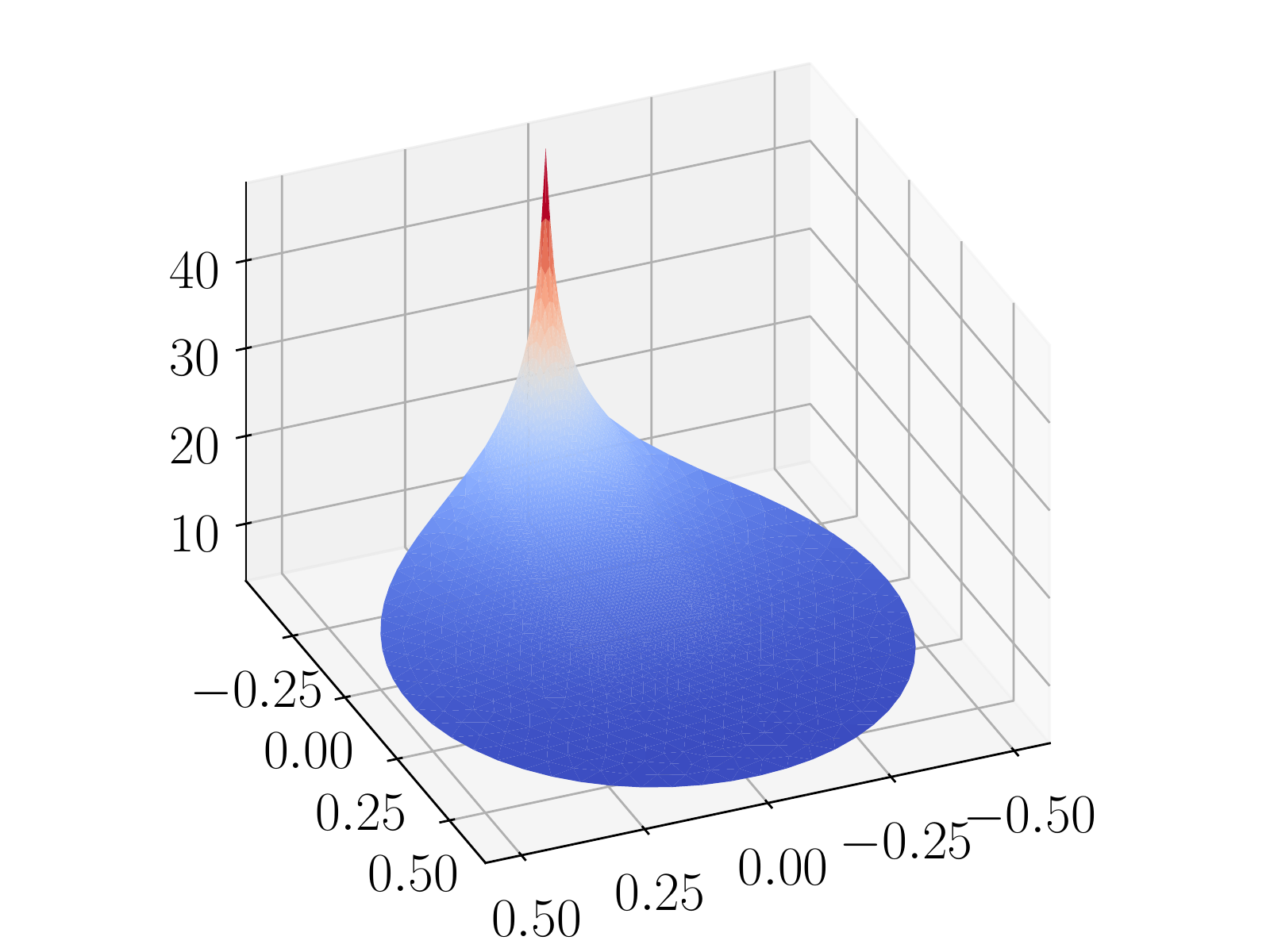}
    \end{subfigure}
    \caption{Algorithm 1: Evolution of $u_h$ (top) and $v_h$ (bottom) at times $t=0$, $0.15$ and $0.2$.}\label{fig.e3-evolution}
\end{figure}
\begin{figure}
    \begin{subfigure}[b]{0.32\textwidth}
        \centering
        \includegraphics[width=1.0\textwidth]{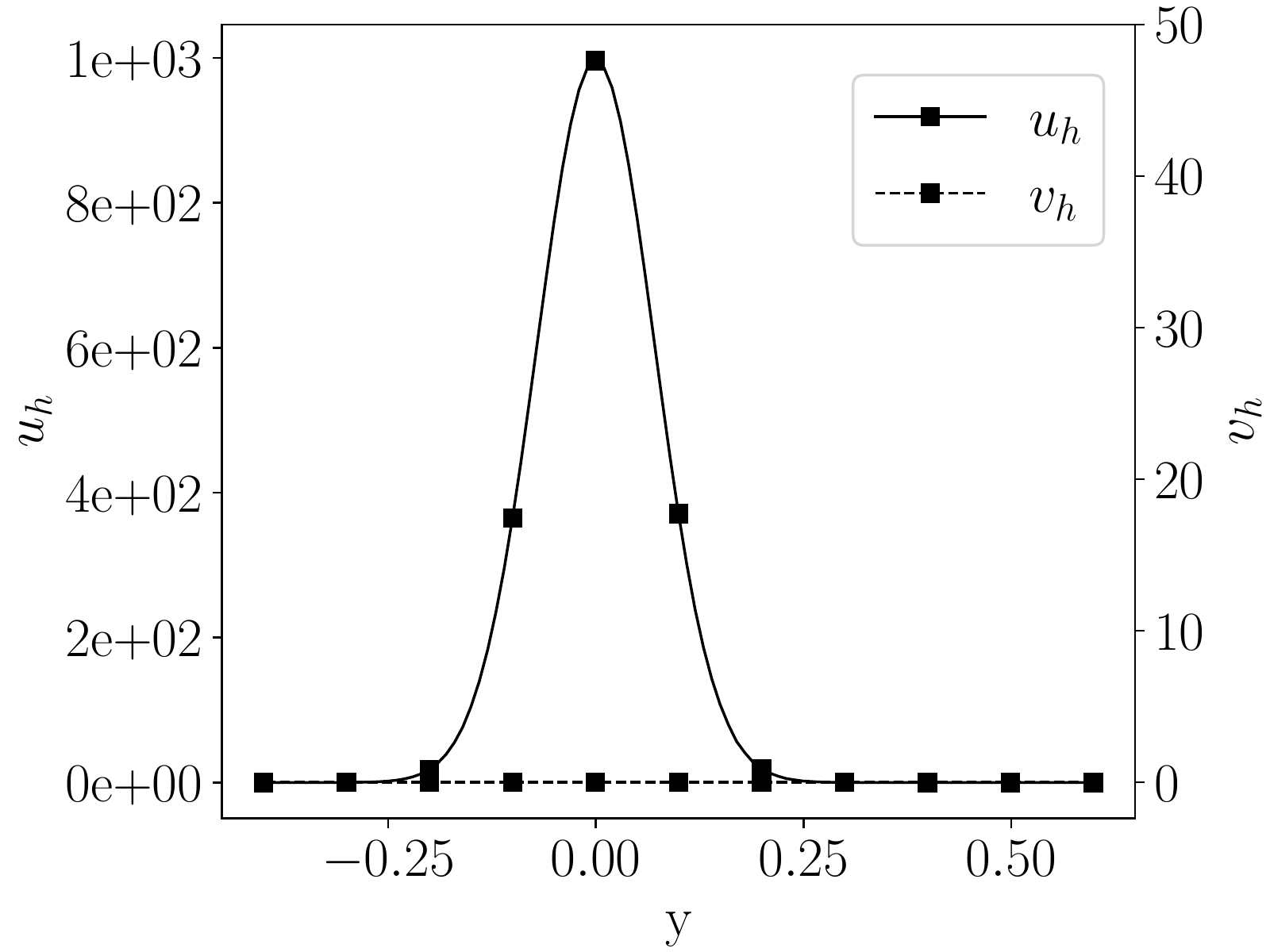}
    \end{subfigure}
    \begin{subfigure}[b]{0.32\textwidth}
        \centering
        \includegraphics[width=1.0\textwidth]{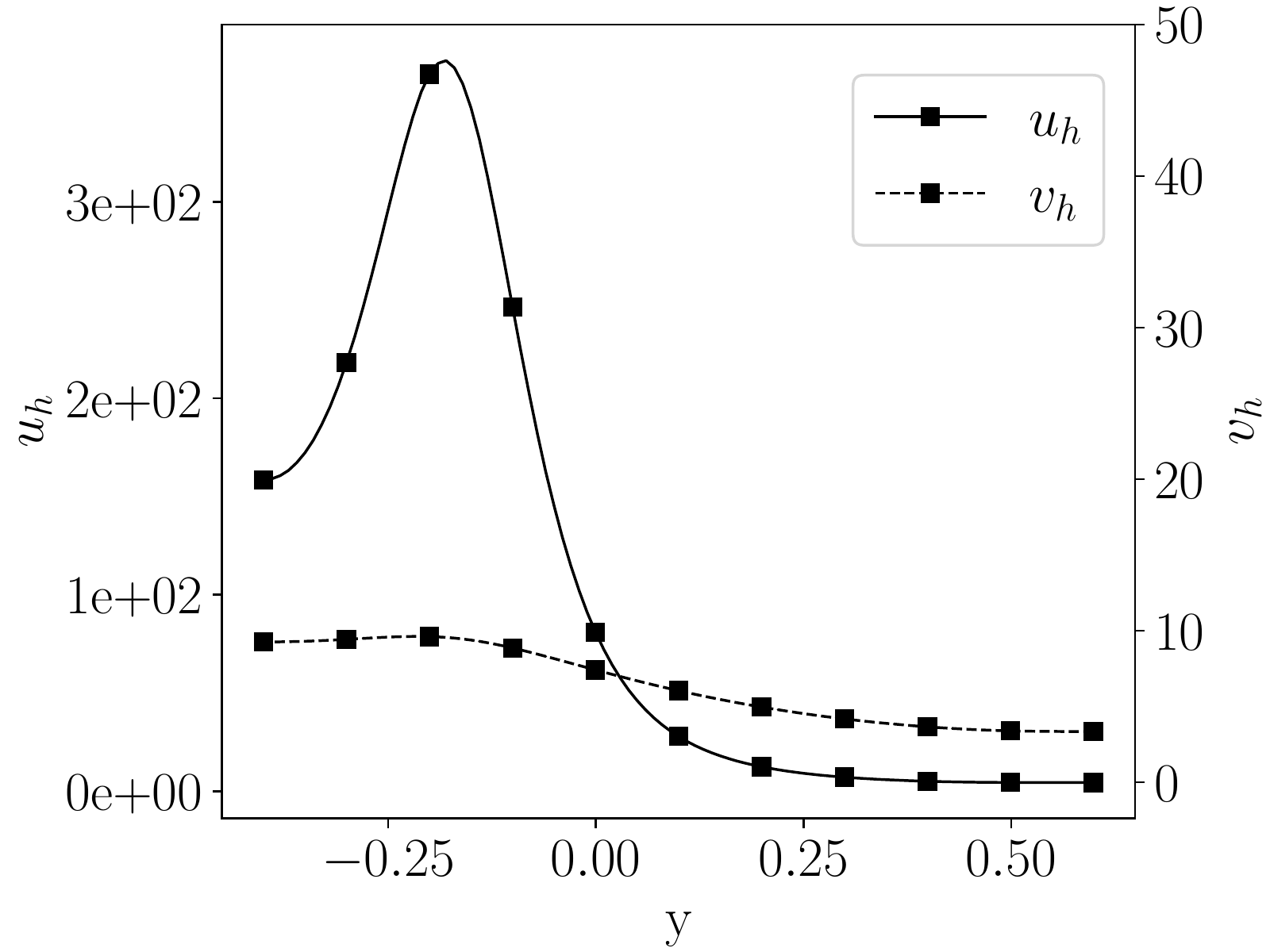}
    \end{subfigure}
    \begin{subfigure}[b]{0.32\textwidth}
        \centering
        \includegraphics[width=1.0\textwidth]{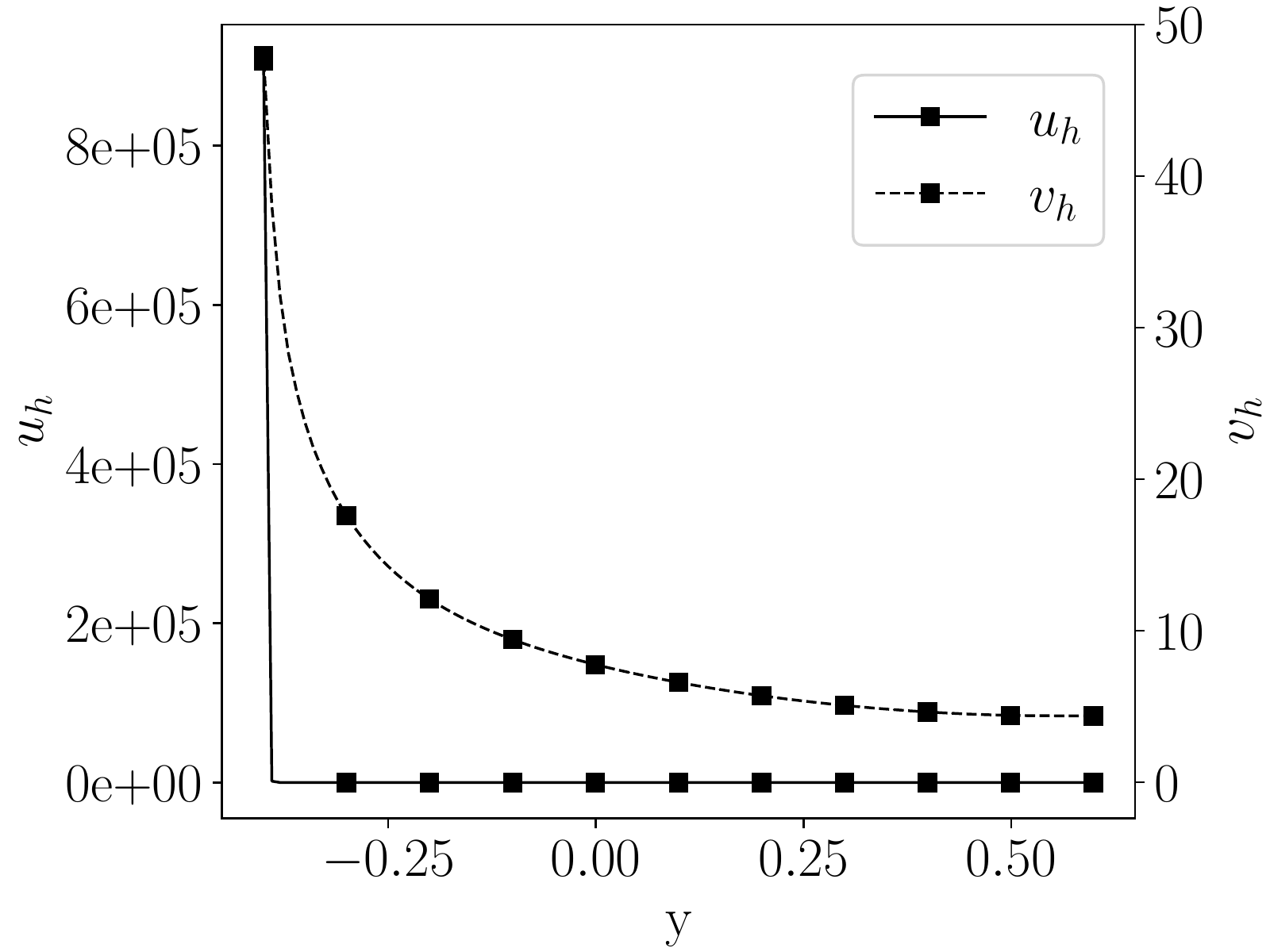}
    \end{subfigure}
    \caption{Algorithm 1: Profiles of $u_h$ and $v_h$ along the plane $x=0$ at times $t=0$, $0.15$ and $0.2$.}\label{fig.e3-profiles}
\end{figure}
In Figures \ref{fig.e3-maximums} and \ref{fig.e3-maximums-alg2} (left), we depict the evolution of the $L^\infty(\Omega)$-norm of the cell and chemoattractant densities whose maxima overgrow simultaneously on the boundary as seen in Figures \ref{fig.e3-profiles} and \ref{fig.e3-maximums-alg2}. It is observed that the blowup time is about $0.19$, reaching a cell-density value of approximately $9.1\cdot 10^{5}$ for Algorithm $1$ and $7.2\cdot 10^{5}$ for  Algorithm~$2$, which arrives at the locking stage earlier. For the mesh used in this experiment, the locking cell density corresponds to $\max_{n\in\{0, \cdots, N\}}\|u^n_h\|_{L^\infty(\Omega)}\leq \frac{4}{h_{\rm min}^2\sqrt{3}} \|u_0\|_{L^1(\Omega)}$ $\approx9.6\cdot 10^5$. Here the $L^\infty(\Omega)$-norm of the chemoattractant density undergoes a rapid growth at the singular point as well, which does not occur in the center-blowup setting. The dynamics of the $L^1(\Omega)$-norm plotted in Figures \ref{fig.e3-maximums} and \ref{fig.e3-maximums-alg2} (right) indicates that the chemoattractant mass increases; this contrasts with Figures \ref{fig.e2-maximums} and \ref{fig.e2-maximums-alg2} (right), where the $L^1(\Omega)$-norm keeps constant practically. As for the energy functional $\mathcal{E}_h(u_h, v_h)$, a considerably pronounced drop in magnitude arises before the computational blowup time.    
\begin{figure}
    \centering
    \includegraphics[width=0.4\textwidth]{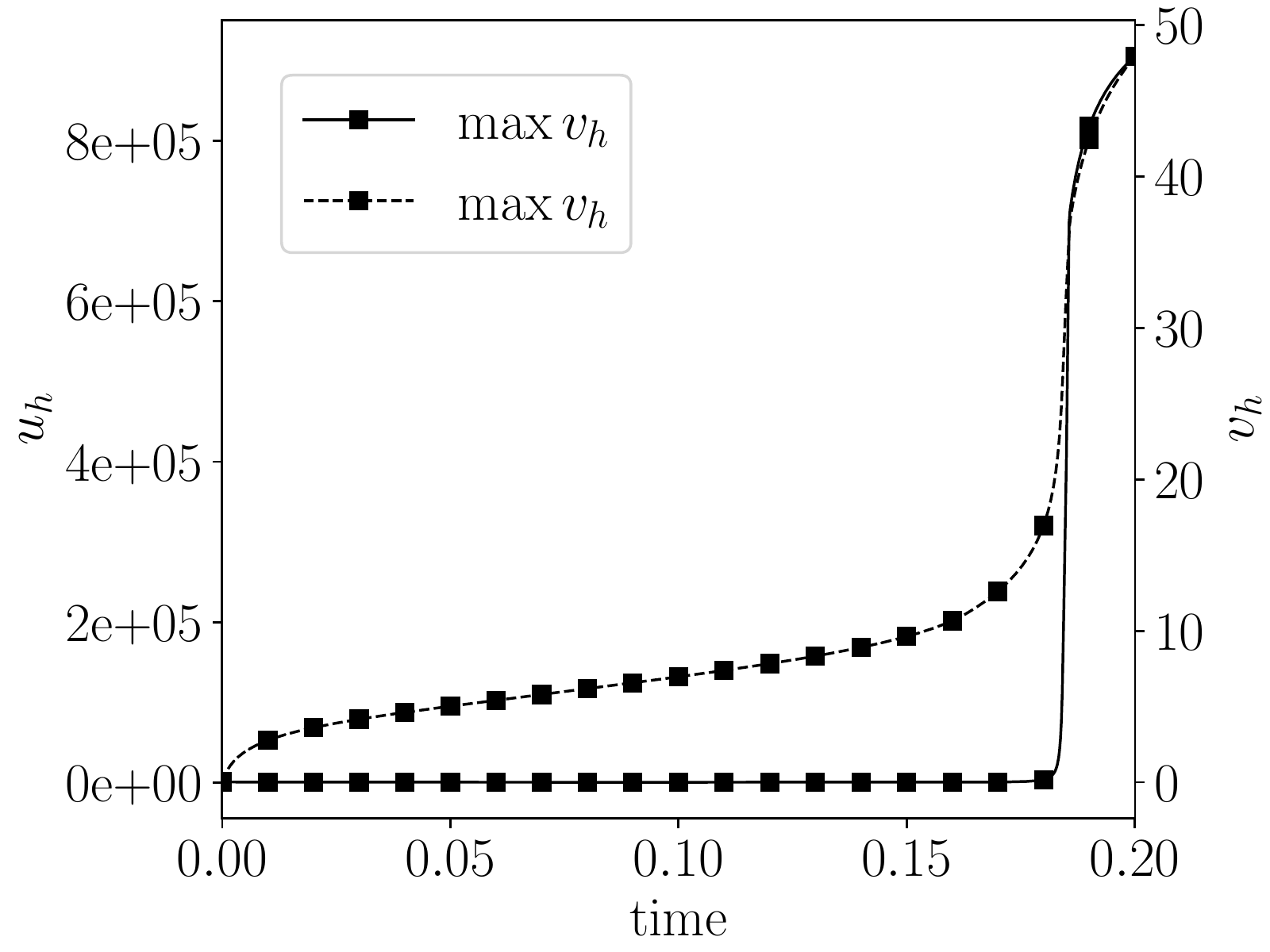}
    \includegraphics[width=0.4\textwidth]{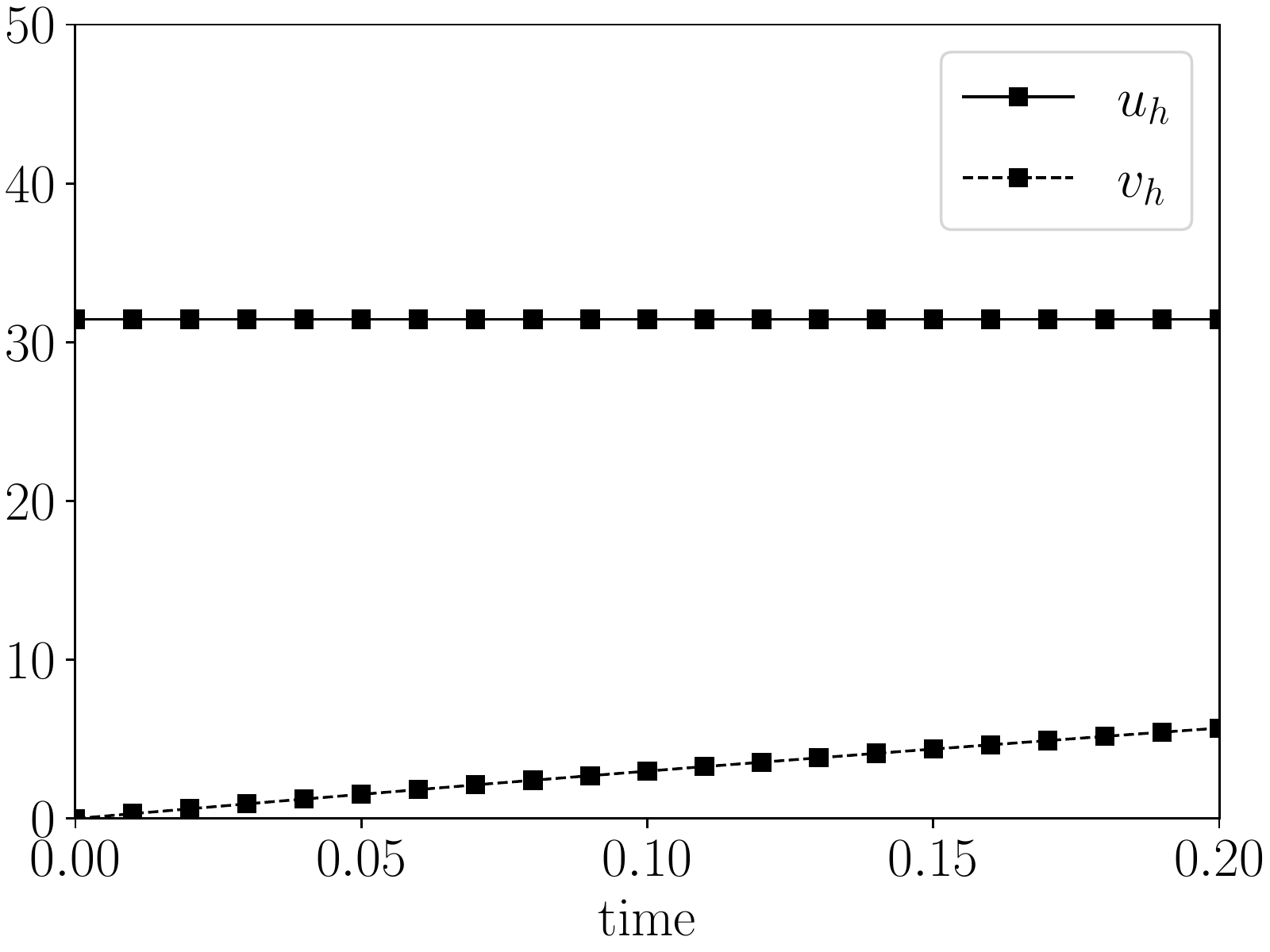}
    \caption{Algorithm 1: Evolution of $\|u_h\|_{L^\infty(\Omega)}$ and $\|v_h\|_{L^\infty(\Omega)}$ (left), and $\|u_h\|_{L^1(\Omega)}$ and $\|v_h\|_{L^1(\Omega)}$ (right).}
    \label{fig.e3-maximums}
\end{figure}
\begin{figure}
    \begin{subfigure}[b]{0.32\textwidth}
        \centering
        \includegraphics[width=1.0\textwidth]{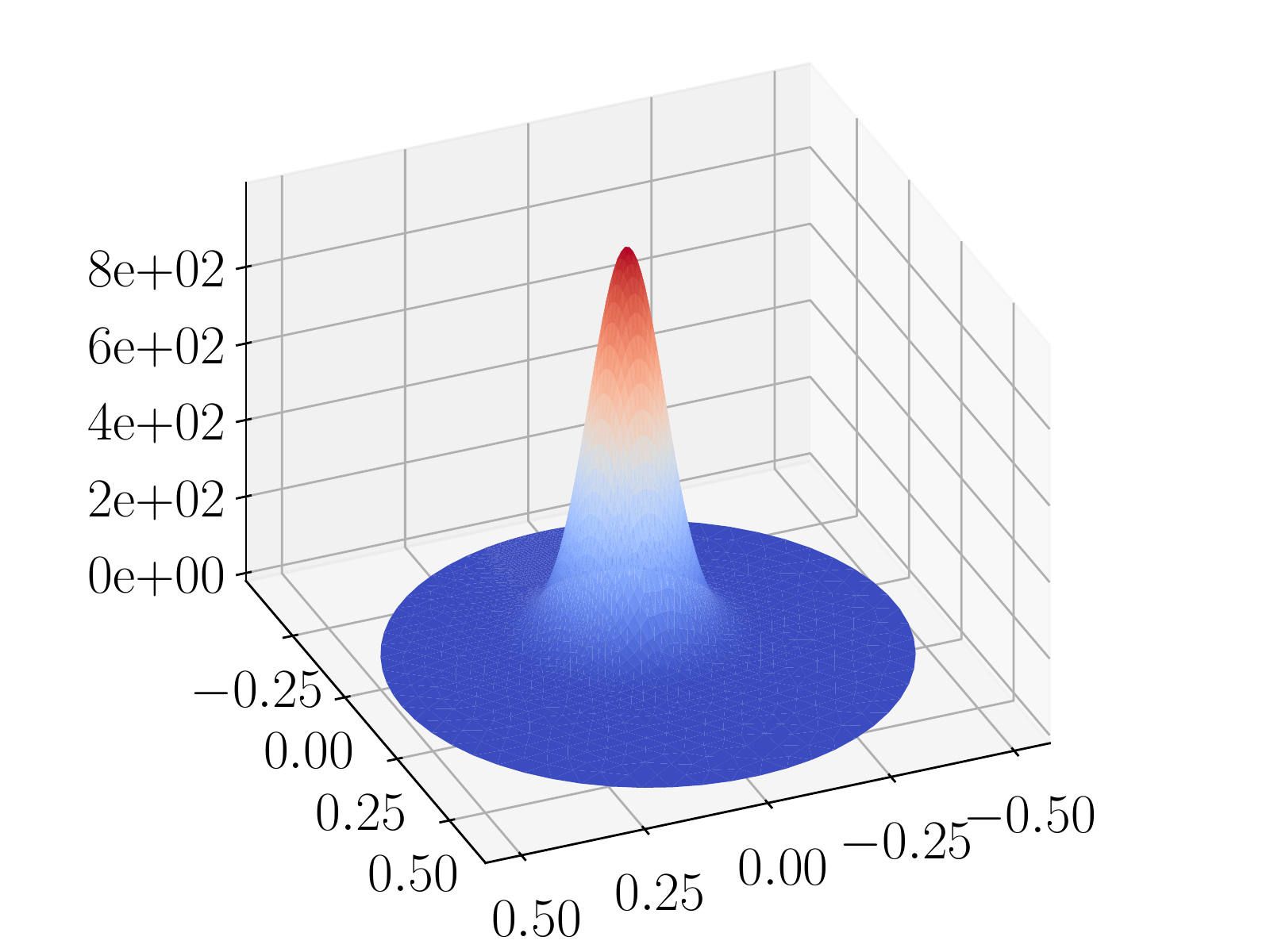}
    \end{subfigure}
    \begin{subfigure}[b]{0.32\textwidth}
        \centering
        \includegraphics[width=1.0\textwidth]{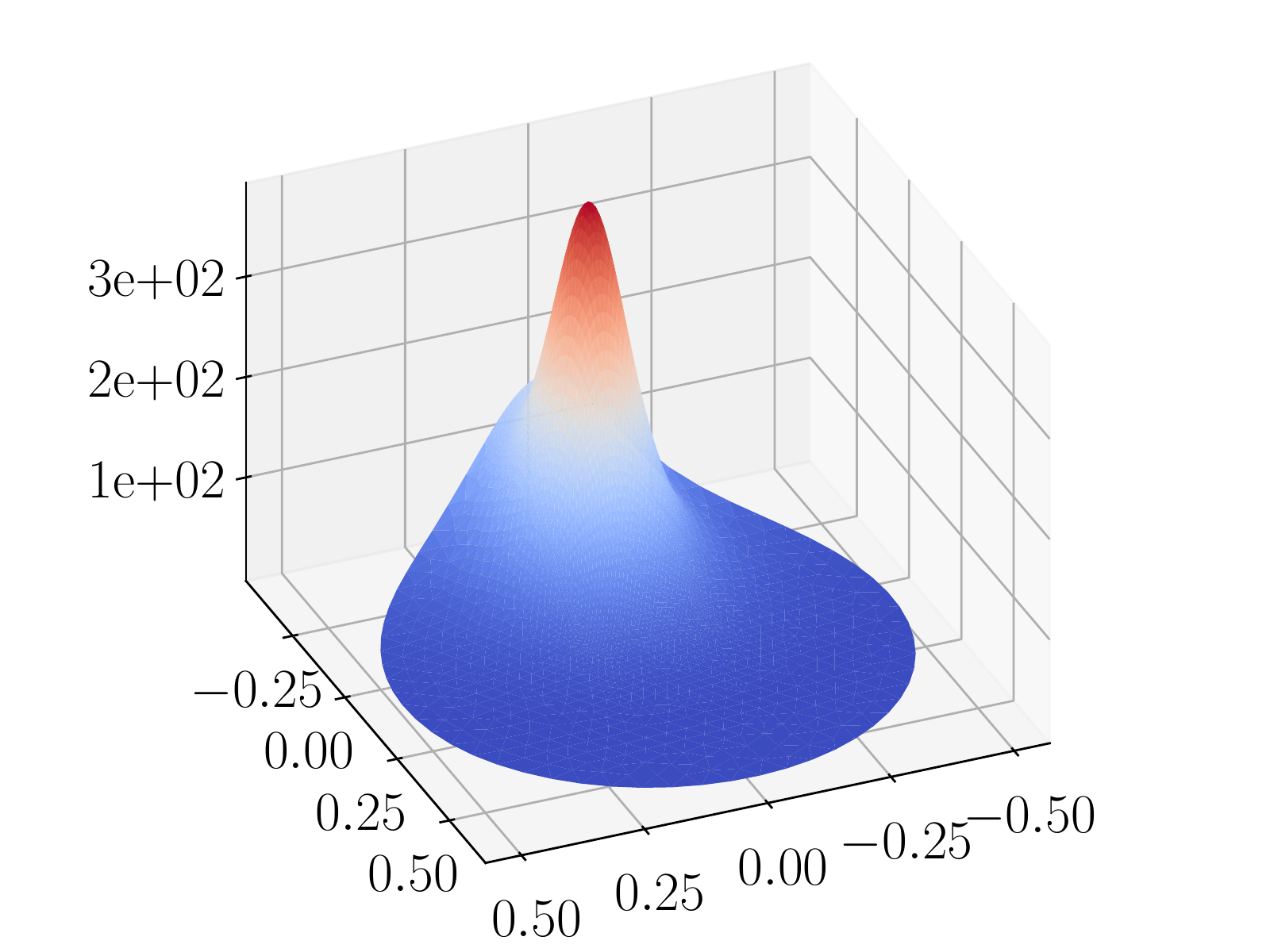}
    \end{subfigure}
    \begin{subfigure}[b]{0.32\textwidth}
        \centering
        \includegraphics[width=1.0\textwidth]{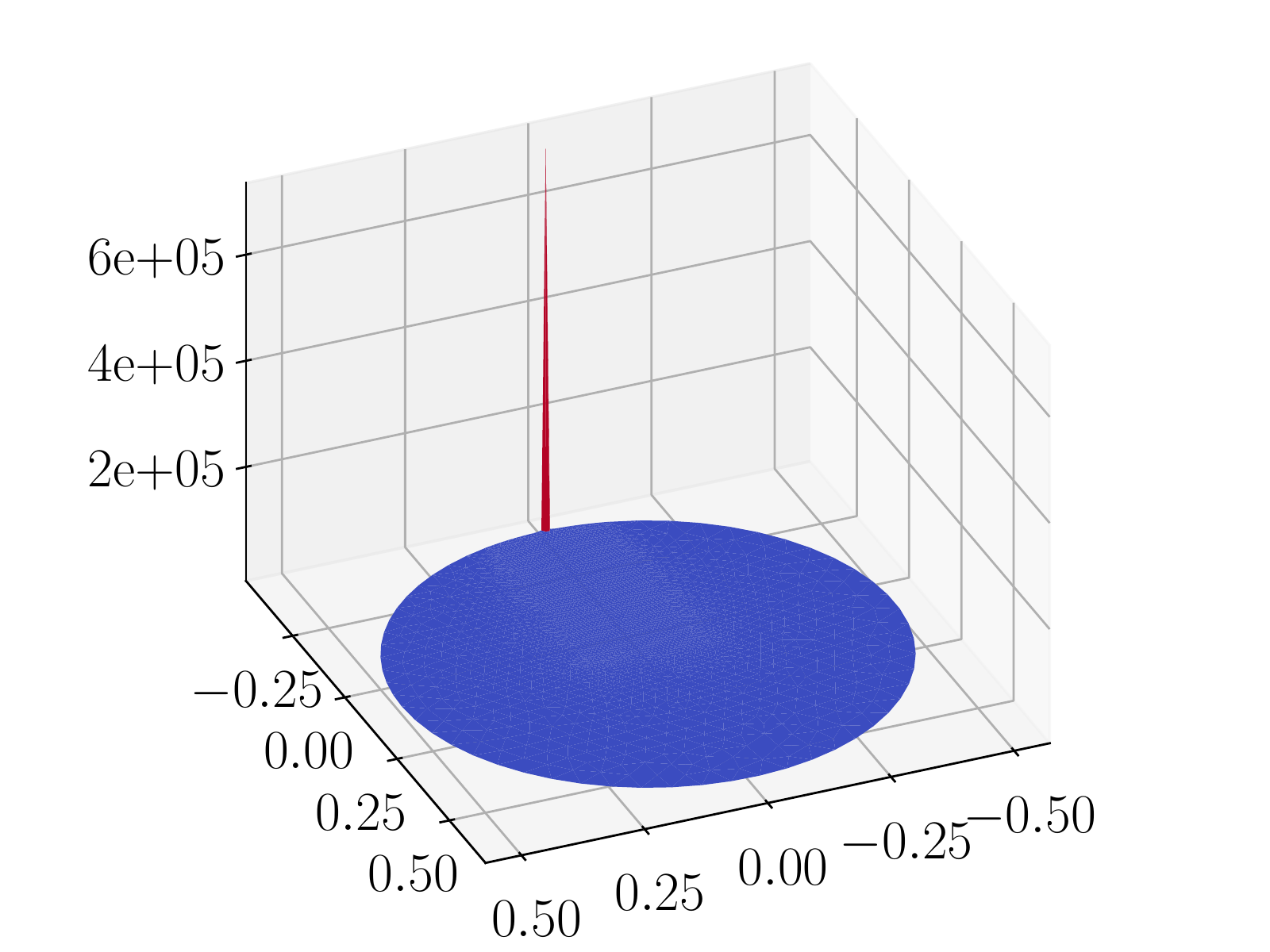}
    \end{subfigure}
    \begin{subfigure}[b]{0.32\textwidth}
        \centering
        \includegraphics[width=1.0\textwidth]{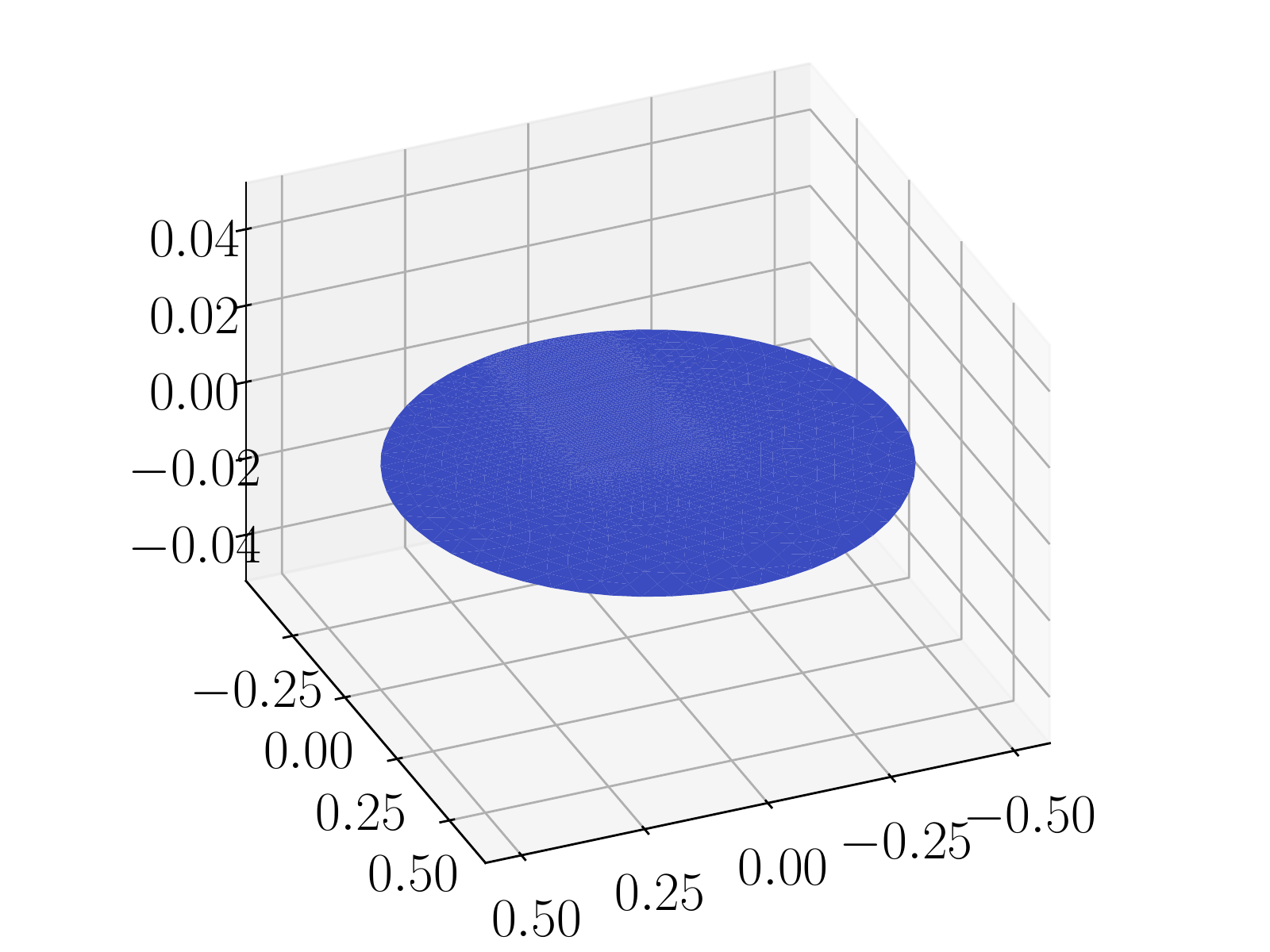}
    \end{subfigure}
    \begin{subfigure}[b]{0.32\textwidth}
        \centering
        \includegraphics[width=1.0\textwidth]{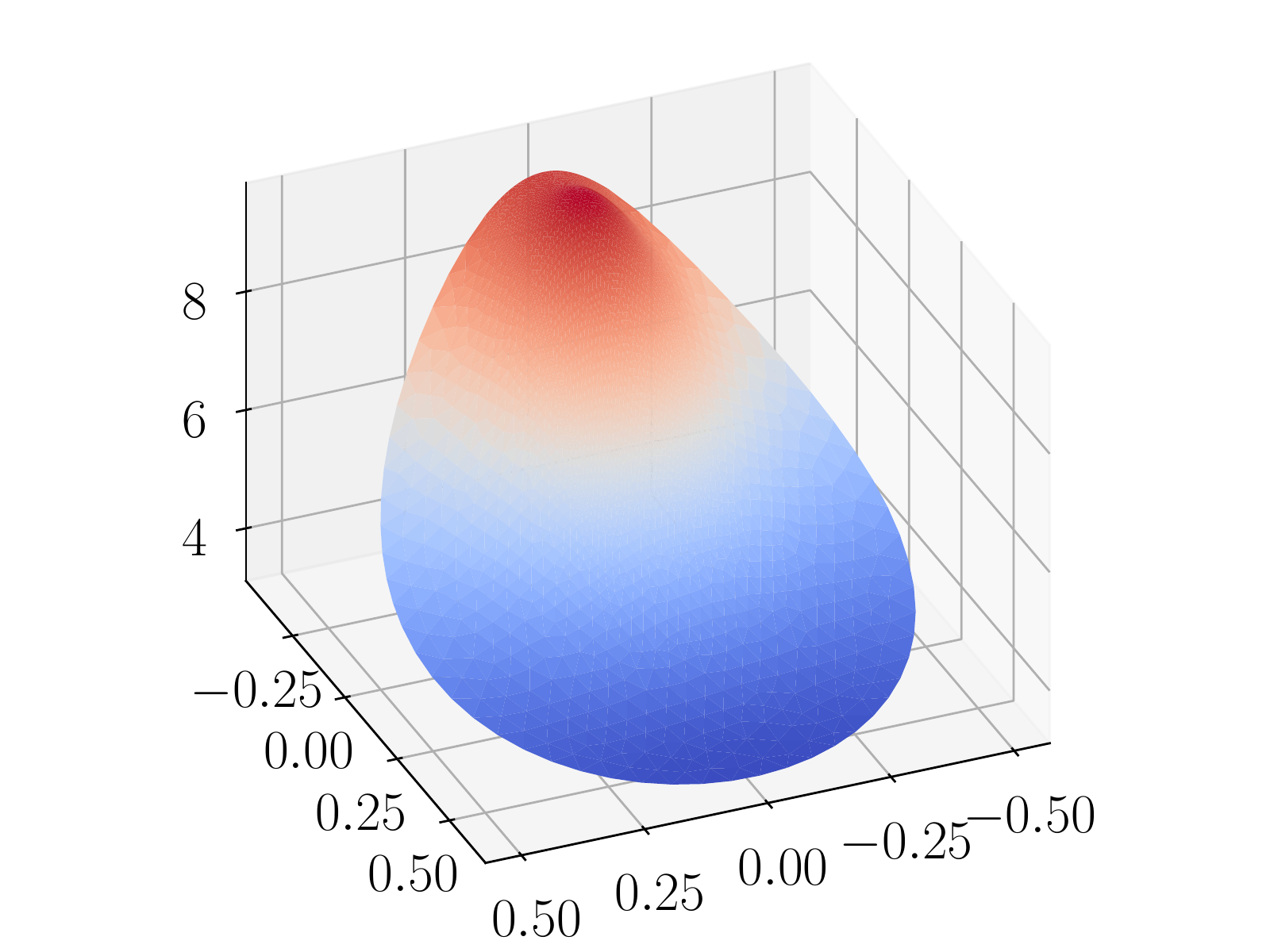}
    \end{subfigure}
    \begin{subfigure}[b]{0.32\textwidth}
        \centering
        \includegraphics[width=1.0\textwidth]{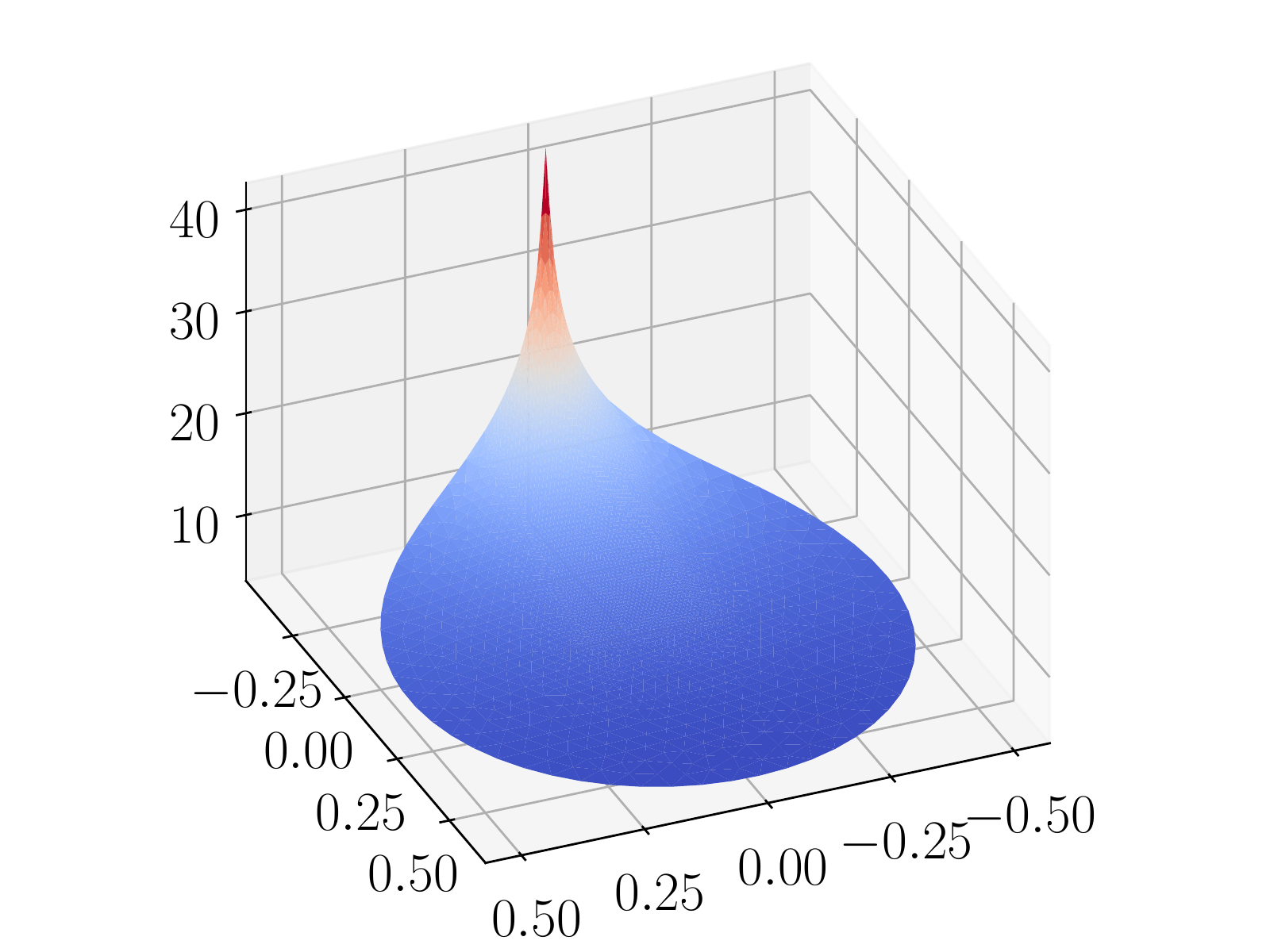}
    \end{subfigure}
    \caption{Algorithm 2: Evolution of $u_h$ (top) and $v_h$ (bottom) at times $t=0$, $0.15$ and $0.2$.}\label{fig.e3-evolution-alg2}
\end{figure}
\begin{figure}
    \begin{subfigure}[b]{0.32\textwidth}
        \centering
        \includegraphics[width=1.0\textwidth]{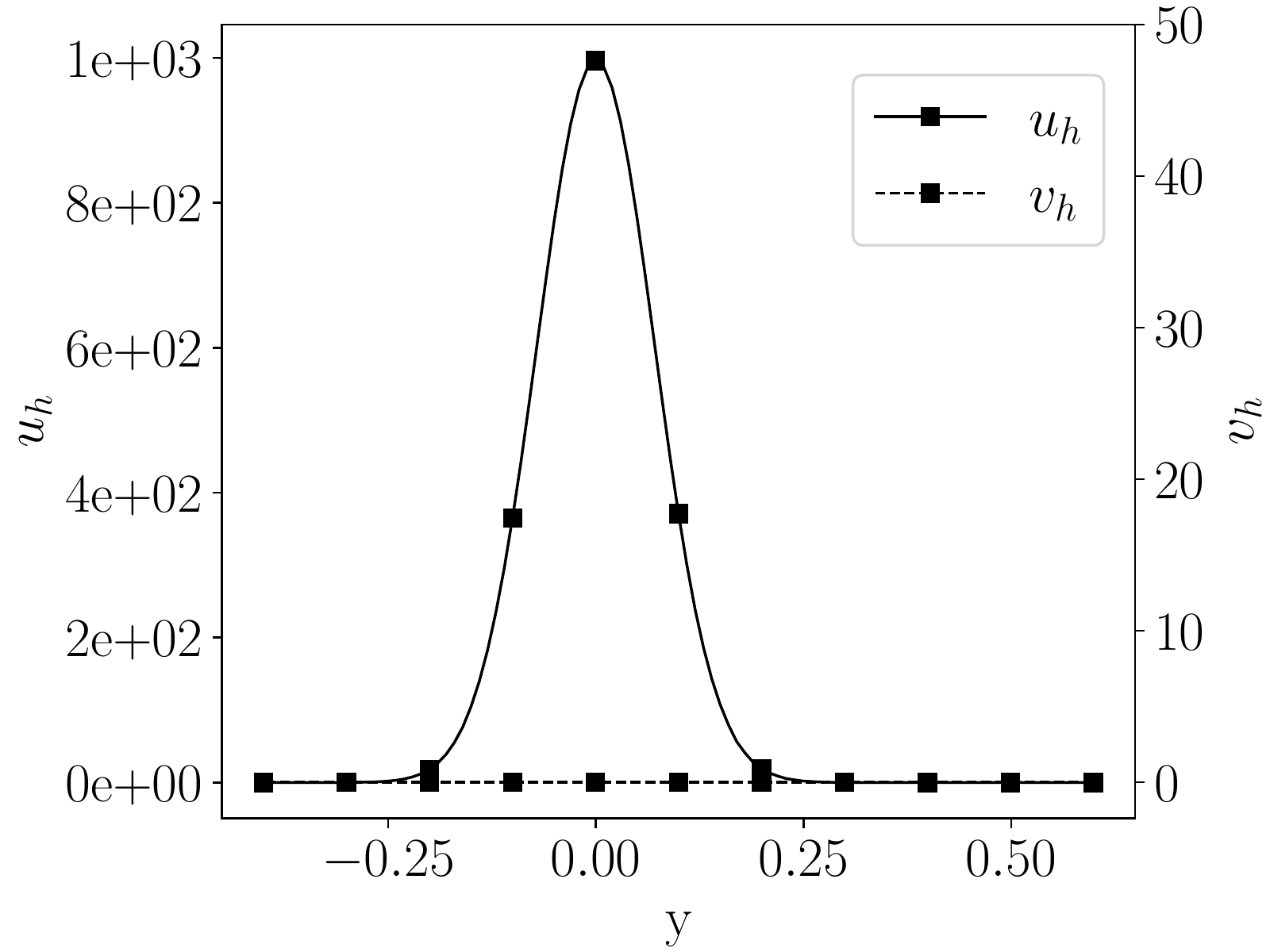}
    \end{subfigure}
    \begin{subfigure}[b]{0.32\textwidth}
        \centering
        \includegraphics[width=1.0\textwidth]{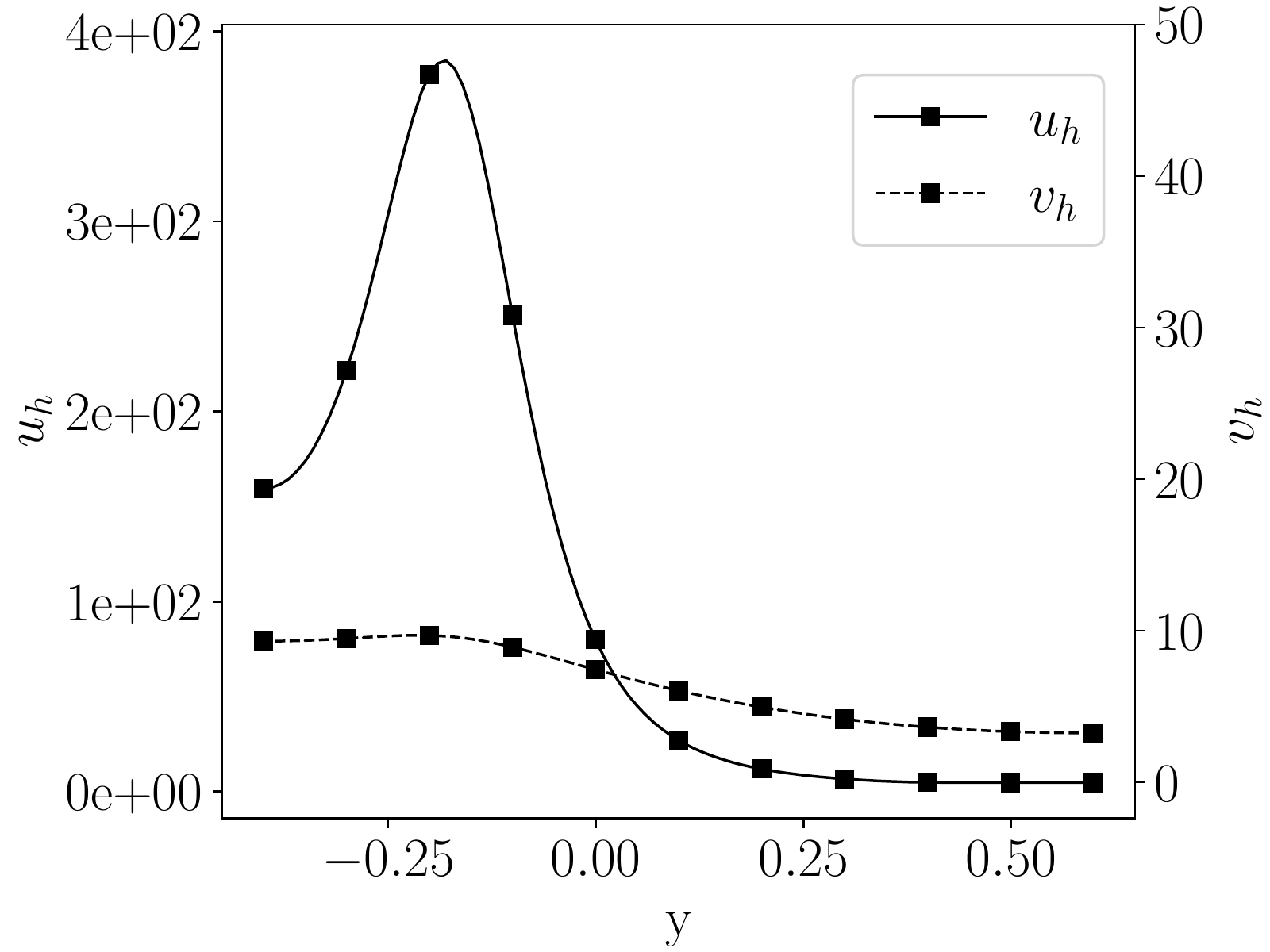}
    \end{subfigure}
    \begin{subfigure}[b]{0.32\textwidth}
        \centering
        \includegraphics[width=1.0\textwidth]{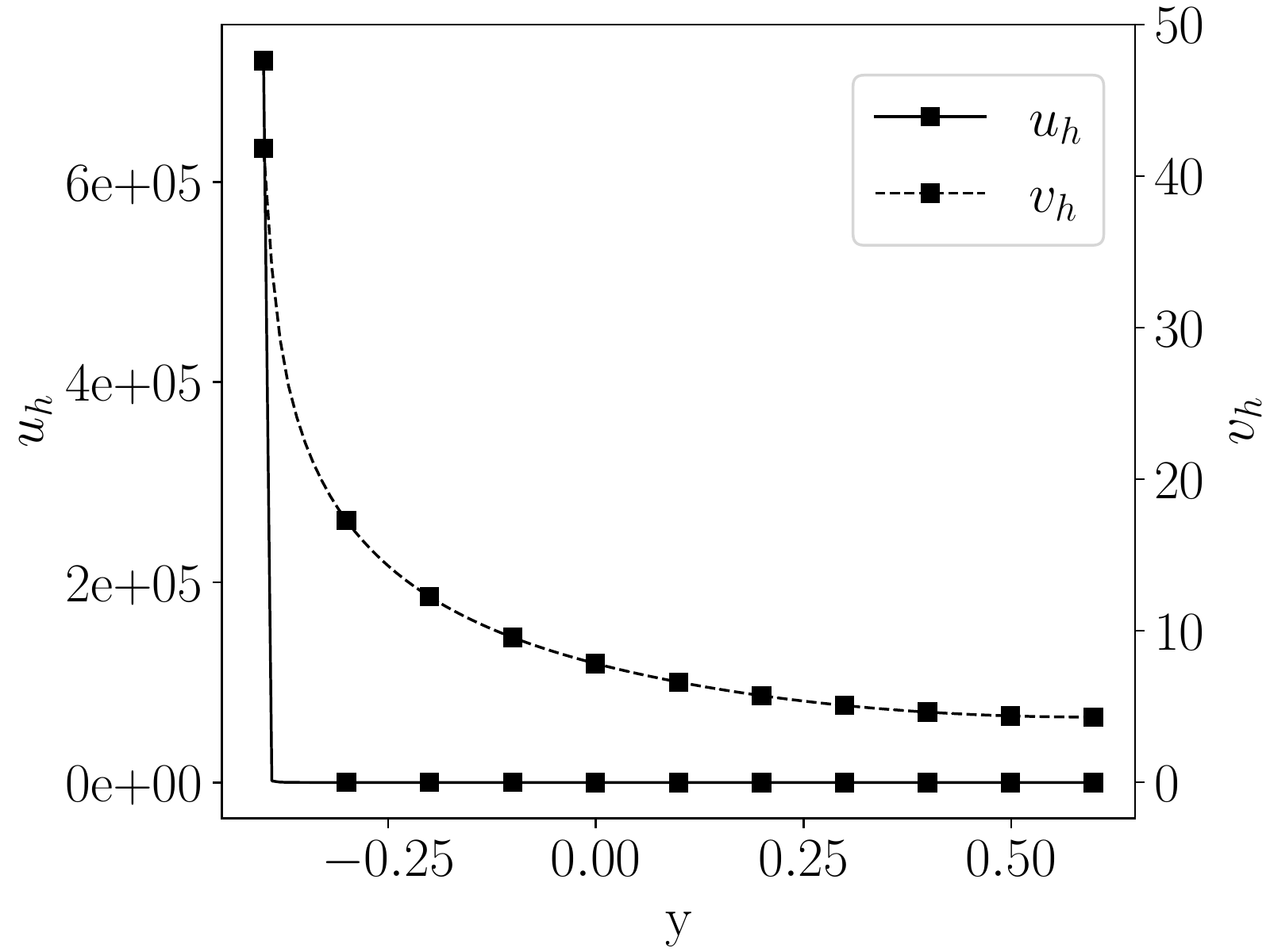}
    \end{subfigure}
    \caption{Algorithm 2: Profiles of $u_h$ and $v_h$ along the plane $x=0$ at times $t=0$, $0.15$ and $0.2$.}\label{fig.e3-profiles-alg2}
\end{figure}
\begin{figure}
    \centering
    \includegraphics[width=0.4\textwidth]{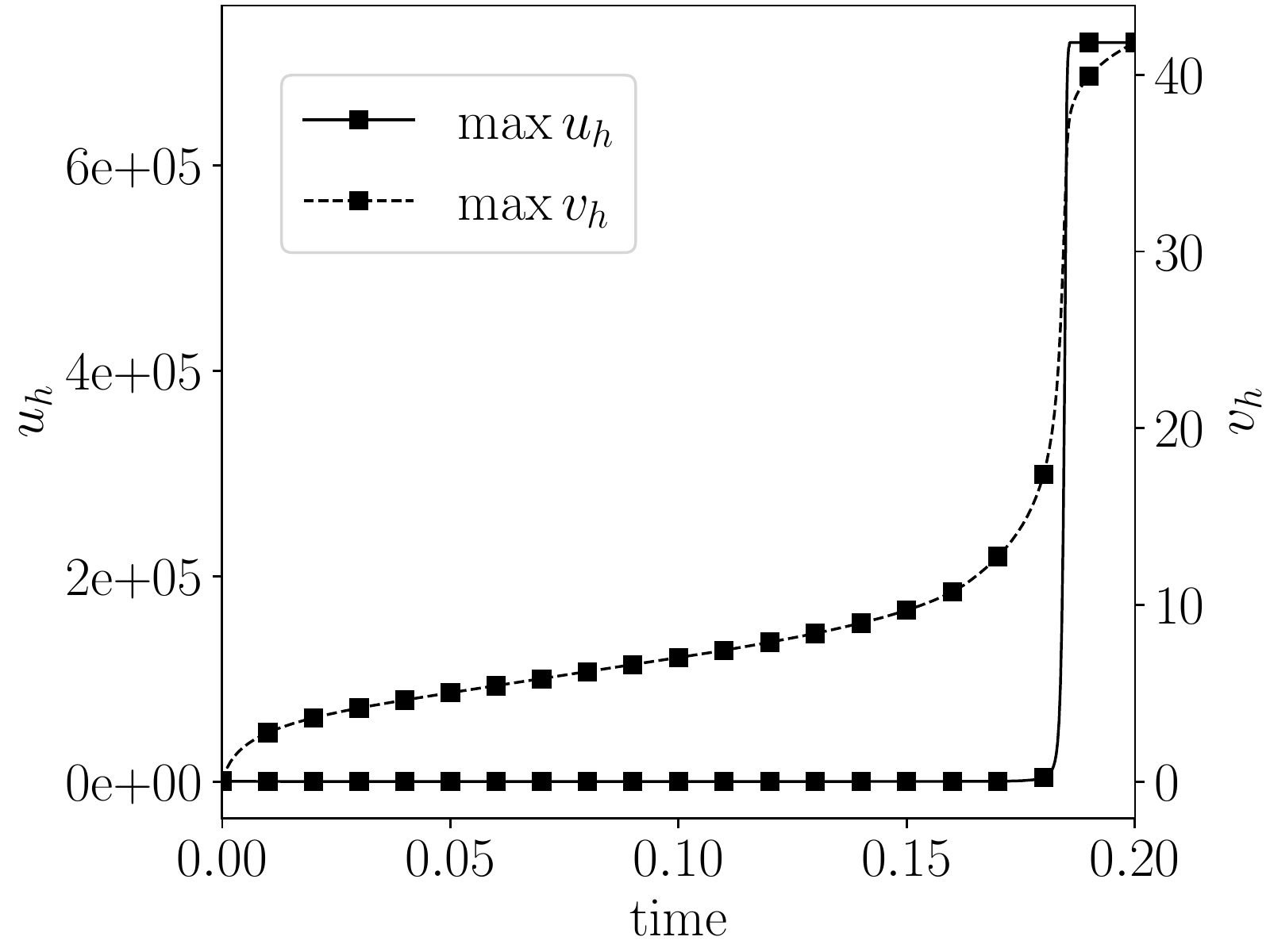}
    \includegraphics[width=0.4\textwidth]{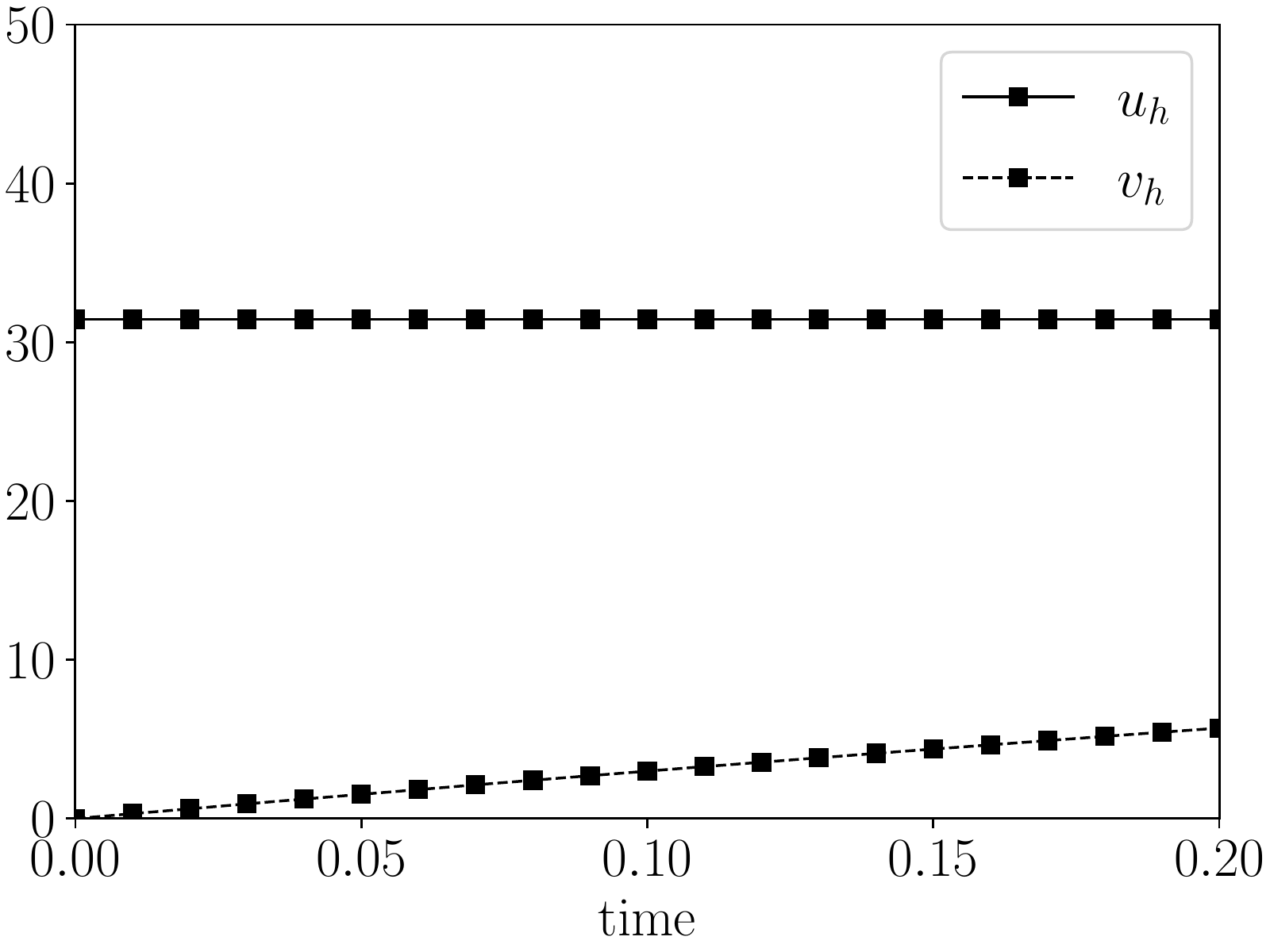}
    \caption{Algorithm 2: Evolution of $\|u_h\|_{L^\infty(\Omega)}$  and $\|v_h\|_{L^\infty(\Omega)}$, and $\|u_h\|_{L^1(\Omega)}$ and $\|v_h\|_{L^1(\Omega)}$.}
    \label{fig.e3-maximums-alg2}
\end{figure}
\begin{figure}
    \centering
    \includegraphics[width=0.4\textwidth]{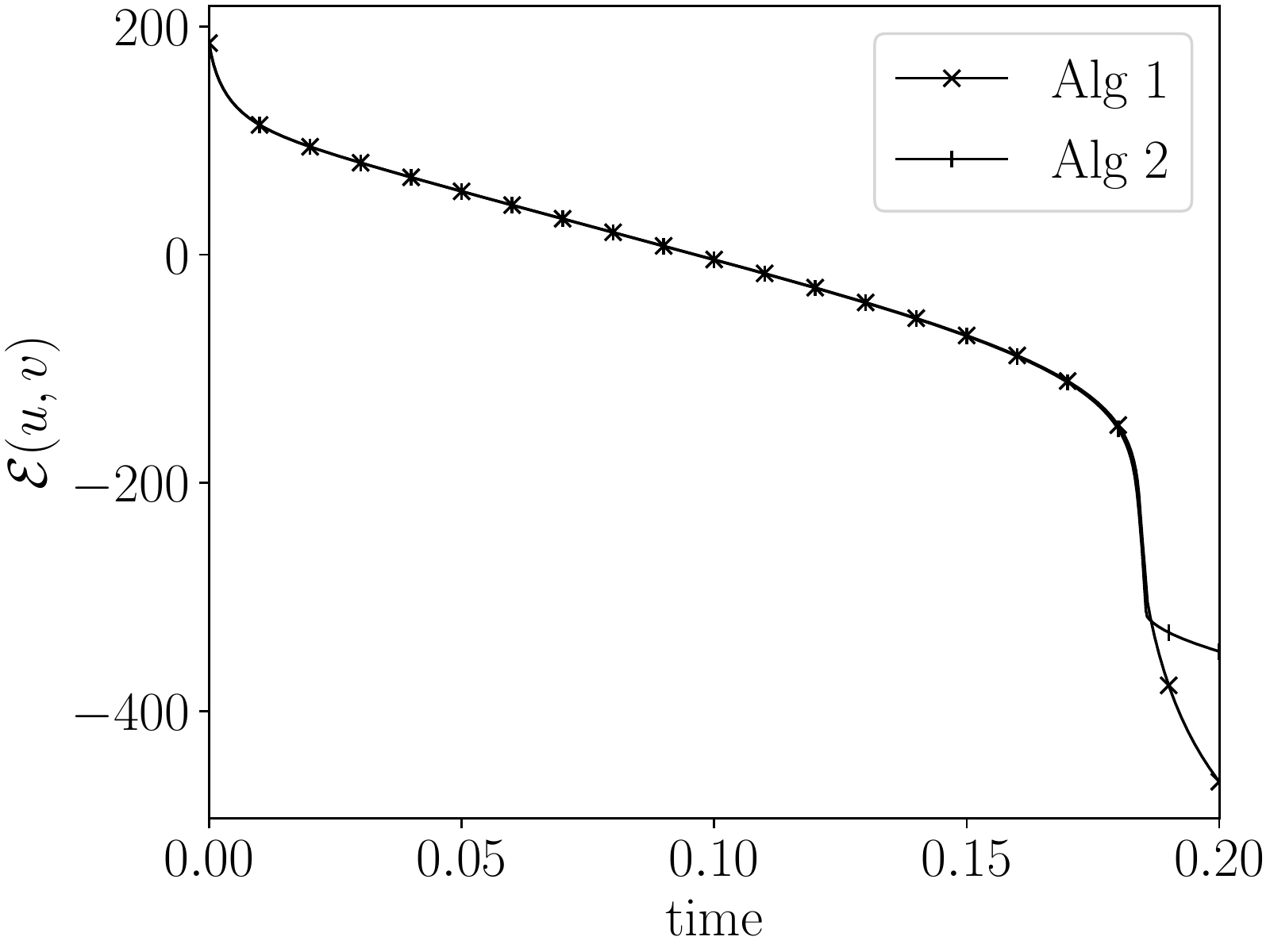}
    \caption{Algorithm 2: Evolution of $\mathcal{E}(u_h,v_h)$.}
    \label{fig.e3-energy}
\end{figure}

\end{document}